\documentclass[12pt]{article}
\usepackage{amsmath,amssymb}
\usepackage{a4wide}
\usepackage{enumerate}
\usepackage[greek, frenchb, english]{babel}
\usepackage[T1]{fontenc} 
\usepackage[dvips,pdftex
]{graphicx}
\usepackage{color} 
\usepackage[colorlinks=true, urlcolor=blue,linkcolor=blue, citecolor=blue]{hyperref}
\usepackage{aeguill}
\usepackage{rotating}
\usepackage{units}
\usepackage{graphicx}
\newtheorem{thm}{Theorem}[section]
\newtheorem{prop}{Proposition}[section]
\newtheorem{lem}{Lemma}[section]
\newtheorem{rem}{Remark}[section]
\newenvironment{proof}[1][Proof.]{\par\noindent\textbf{#1} }
{\hfill~$\square$\\\medskip}
\newcommand{\CQFD}{\hfill $\square$\\}
\newcommand{\ind}{\mathbf{1}}
\definecolor{Red}{rgb}{1.00, 0.00, 0.00}

\def\aa{\boldsymbol{\alpha}}
\def\kk{\boldsymbol{k}}
\def\ttt{\boldsymbol{t}}
\def\lcin{\mbox{LCI}_n}
\def\bbR{{\mathbb R}}
\def\bbe{{\mathbb E}}
\def\ee{{\mathbb E}}
\def\bbn{{\mathbb N}}
\def\bbp{{\mathbb P}}
\def\bbz{{\mathbb Z}}
\def\pp{{\mathbb P}}
\def\nit{{\mathbb N}}
\def\cala{{\cal A}}
\def\call{{\cal L}}
\def\calv{{\cal V}}
\def\Ccal{{\cal C}}
\def\var{\mathop{\rm Var}}
\def\cov{\mathop{\rm Cov}}
\def\Var{\mathop{\rm Var}}
\def\indep{\perp \!\!\!\!\perp}

\numberwithin{equation}{section}
\allowdisplaybreaks
\def\bwedge{\textstyle\bigwedge}
\def\cvP{\stackrel{\pp}{\longrightarrow}}
\allowdisplaybreaks

\begin{document}
\title{
On the limiting law of the length of the longest common and increasing subsequences in random words}
\author{Jean-Christophe Breton\thanks{IRMAR, UMR 6625, Universit\'e de Rennes 1, 
263 Avenue du G\'en\'eral Leclerc CS 74205, 35042, Rennes, France, 
\href{mailto:jean-christophe.breton@univ-rennes1.fr}{jean-christophe.breton@univ-rennes1.fr}.
Many thanks to the School of Mathematics of the Georgia Institute of 
Technology for several visits during which part of this work was done.} 
\and Christian Houdr\'e\thanks{School of Mathematics, Georgia
 Institute of Technology, Atlanta, GA 30332, USA, \href{mailto:houdre@math.gatech.edu}{houdre@math.gatech.edu}. 
Research supported in part by the 
grant \#246283 from the Simons Foundation and by a Simons Foundation Fellowship grant \#267336.  
Many thanks to the Centre Henri Lebesgue of the Universit\'e de Rennes 1, 
the D\'epartement MAS of 
\'Ecole Centrale Paris, to the LPMA of the Universit\'e Pierre et Marie Curie and to CIMAT, 
Gto, Mexico for their hospitality, while this work was in progress. 
\newline\indent
Keywords:  Longest Common Subsequence, Longest Increasing Subsequence, Random Words, 
Random Matrices, Donsker's Theorem,
Optimal Alignment, Last Passage Percolation.
\newline\indent
MSC 2010: 05A05, 60C05, 60F05.}}

\maketitle
\centerline{\`A la M\'emoire de Marc Yor}
\begin{abstract}
Let $X=(X_i)_{i\ge 1}$ and $Y=(Y_i)_{i\ge 1}$ be two sequences of independent 
and identically distributed (iid) random variables taking their values, uniformly, 
in a common totally ordered finite alphabet. Let $\lcin$ be the length of the longest 
common and (weakly) increasing 
subsequence of $X_1\cdots X_n$ and $Y_1\cdots Y_n$. As $n$ grows without bound, 
and when properly centered and scaled, $\lcin$ is shown to converge, in distribution, 
towards a Brownian functional that we identify.
\end{abstract}


\section{Introduction}
\label{sec:intro}

We analyze below the asymptotic behavior of the length of the longest common 
subsequence in random words with an additional (weakly) increasing requirement. 
Although it has been studied from an algorithmic point of view in computer science, 
bio-informatics, or statistical physics (see, for instance, \cite{CZFYZ}, 
\cite{DKFPWS} or \cite{Sakai}),  
to name but a few fields, mathematical results for this hybrid 
problem are very sparse.  
To present our framework, let $X=(X_i)_{i\ge 1}$ and $Y=(Y_i)_{i\ge 1}$ be two 
infinite sequences whose coordinates take their values 
in $\cala_m=\{\aa_1<\aa_2< \cdots <\aa_m\}$, a finite totally ordered alphabet of 
cardinality $m$.   
Next, $\lcin$, the length of the longest common and (weakly) increasing 
subsequences of the words $X_1\cdots X_n$ and $Y_1\cdots Y_n$ is the maximal 
integer $k\in \{1,\dots, n\}$, such that there 
exist $1\le i_1< \cdots < i_k\le n$ and $1\le j_1<\cdots <j_k\le n$, satisfying 
the following two conditions:
\begin{itemize}
\item[(i)]  $X_{i_s}=Y_{j_s}$, for all $s=1,2,\dots, k$,
\item[(ii)] $X_{i_1}\le X_{i_2}\le \cdots \le X_{i_k}$ and $Y_{j_1}\le Y_{j_2}
\le \cdots \le Y_{j_k}$. 
\end{itemize}
(Asymptotically, the strictly increasing case is of little interest, having $m$ as a pointwise limiting behavior.) $\lcin$ is a measure of the similarity/dissimilarity of the random words often used in pattern matching, and its asymptotic behavior is the purpose of our study. 
This limiting behavior differs from the one of another better-known, measure of similarity/dissimilarity, namely, ${\rm LC}_n$, the length of the longest common subsequences of two or more random words.  Indeed, after renormalization, the first result on ${\rm LC}_n$, obtained in \cite{HI}, reveals, under a sublinear variance lower bound assumption, a normal limiting law.   
In contrast, for $\lcin$, we have:   

\begin{thm}
\label{thm1.1}
Let $X=(X_i)_{i\ge 1}$ and $Y=(Y_i)_{i\ge 1}$ be two sequences of iid random variables uniformly distributed on $\cala_m=\{\aa_1<\aa_2< \cdots <\aa_m\}$, a totally ordered finite alphabet of cardinality $m$.  Let $\lcin$ be the length of the longest common and increasing subsequences of $X_1 \cdots X_n$ and $Y_1 \cdots Y_n$. 
Then, as $n\to +\infty$,   
\begin{eqnarray}
\label{eq00}
\frac{\lcin - {n/m}}{\sqrt{n/m}} \Longrightarrow 
\max_{0=t_0 \le t_1 \le \dots \le t_{m-1} \le t_m=1}
\!\min\!\left(\!-\frac{1}{m}\!\sum_{i=1}^m\!B^{(i)}_1\!(1)
+\sum_{i=1}^m\!\left(\!B^{(i)}_1\!(t_i)-B^{(i)}_1\!(t_{i-1})\right), \right.\nonumber\\ 
\left.\!\! \!-\frac{1}{m}\!\sum_{i=1}^m\!B^{(i)}_2\!(1)+ 
\sum_{i=1}^m\left(B^{(i)}_2\!(t_i)-B^{(i)}_2\!(t_{i-1})\right)\!\right)\!\!, 
\end{eqnarray}
where $B_1$ and $B_2$ are two $m$-dimensional standard 
Brownian motions on $[0,1]$.
\end{thm}

The main motivation for our work has its origins in the identification, first 
obtained by Kerov~\cite{kerov}, of the limiting length 
(properly centered and scaled) 
of the longest increasing subsequence of a random word, as the maximal 
eigenvalue of a certain Gaussian random matrix.  When combined with 
results of Baryshnikov~\cite{Bar} or 
Gravner, Tracy and Widom~\cite{GTW} (see also \cite{BGH}), this limiting law 
has a representation as a Brownian functional. 
Moreover, 
the longest increasing subsequence corresponds to the first row of the RSK Young diagrams 
associated with the random word 
and \cite[Chap. 3, Sec. 3.4, Theorem 2]{kerov} showed that 
the whole normalized limiting shape of these RSK Young diagrams is the spectrum of 
the traceless Gaussian Unitary Ensemble (GUE).  
Since the length of the top row of the diagrams is 
the length of the longest increasing 
subsequence of the random word, the maximal eigenvalue result is recovered.  
(The asymptotic length result was rediscovered by Tracy and 
Widom~\cite{TW} and the asymptotic shape one 
by Johansson~\cite{KJAOM2001}.  
Extensions to non-uniform letters were also obtained    
by Its, Tracy and Widom~\cite{ITW1, ITW2}.) 
Another motivation for the present study comes from the interpretation of 
the $\lcin$ functional in terms of last passage time in directed percolation. 
This is detailed in our concluding remarks. 

The asymptotic behavior of the length of the longest common and increasing subsequences has actually already been investigated for binary words ($m=2$) in \cite{HLM}.
However, the methods used there do not allow to consider an alphabet of arbitrary finite size $m$. 
When $m=2$ with letters $\aa_1$ and $\aa_2$, it is enough to consider common subsequences made of a random number of common $\aa_1$'s deterministically completed by the common $\aa_2$'s, so that in a way the corresponding study is reduced to deal with only one type of letter. In contrast, when $m\geq 3$, the situation is much more complicated since a similar strategy reduced the problem to $m-1$ types of letter for which there is still, roughly speaking, too much randomness to successfully handle, in this way, the study of $\lcin$.
A new methodology based on a new representation of $\lcin$ is thus required to deal with general finite alphabet of size $m$. 
This is achieved below where an appropriate representation of $\lcin$,  that allows to investigate its asymptotic behavior for arbitrary $m\geq2$, is obtained. 
Our results thus extend and encompass the binary $\lcin$ result of \cite{HLM}. 
The dependence (or independence) structure between the two sequences  of letters $X$ and $Y$ is carried over at the limit into a similar structure between the two standard Brownian motions $B_1$ and $B_2$.  
Hence, when $X=Y$, our results recover, with the help of \cite{BGH}, the weak limits obtained in \cite{kerov}, \cite{KJAOM2001}, \cite{TW}, \cite{ITW1}, \cite{ITW2}, \cite{HL}, and \cite{HX}, while if $X$ and $Y$ are independent so are $B_1$ and $B_2$.   
As a by-product of our approach, we further fix some loose points present in \cite{HLM}. 
As suggested to us, let us further put our main theorem in context.  
At first, for  
$m=2$, the right hand-side of \eqref{eq00} becomes 
$$\max_{0\le t\le 1}\min\!\left(\!\frac{B_1^{(2)}(1)-B_1^{(1)}(1)}{2}-(B_1^{(2)}(t)-B_1^{(1)}(t)),
\frac{B_2^{(2)}(1)-B_2^{(1)}(1)}{2}-(B_2^{(2)}(t)-B_2^{(1)}(t))
\!\!\right)\!.$$
In case the two-dimensional standard Brownian motions are independent, 
this last expression has the same law as   
$$
{\sqrt 2}\max_{0\le t\le 1}\min\!\left(B_1(t) -\frac{1}{2}B_1(1), 
B_2(t) -\frac{1}{2}B_2(1)
\right),
$$
where, now, $B_1$ and $B_2$ are two independent one-dimensional standard Brownian motions on $[0,1]$. 
Therefore, our limiting result matches the binary one presented in \cite{HLM}.  

Next, and still for further context, let us compare the asymptotic  behavior of $\lcin$ to the one of, say, $L_n$,  
the length of the optimal alignments which align only one type of letters.  
(In case of a single word, $L_n$ could correspond to, e.g, the length of the 
longest constant subsequences).  Clearly, $\lcin \ge {\rm L}_n$ and under a uniform 
assumption, 
\begin{equation}
\lim_{n\to +\infty}\frac{\lcin}{n} = \lim_{n\to +\infty}\frac{{\rm L}_n}{n} = \frac{1}{m}, 
\end{equation}  
with probability one.  Moreover, it is easy to see that, as $n\to +\infty$,  
\begin{equation}\label{lconst}
\frac{{\rm L}_n - n/m}{\sqrt{n/m}} \Longrightarrow \min\left(\sqrt{1-\frac{1}{m}}B_1(1),  
\sqrt{1-\frac{1}{m}}B_2(1)\right),   
\end{equation}  
for, say, two one-dimensional standard Brownian motions $B_1$ and $B_2$.   
Now returning to \eqref{eq00}, note that for $j=1,2$,  
\begin{eqnarray}
&&\!\!\!\!\!\!\!\!-\frac{1}{m}\sum_{i=1}^{m}B^{(i)}_j(1)
+ \sum_{i=1}^m (B^{(i)}_j(t_i)-B^{(i)}_j(t_{i-1})) \nonumber\\
&&\!\!\! =\frac{1}{m}\left(\!(m-1)B^{(m)}_j(1) - \sum_{i=1}^{m-1}B^{(i)}_j(1)\!\right) 
+ \sum_{i=1}^{m-1} (B^{(i)}_j(t_i)-B^{(i)}_j(t_{i-1})) - B^{(m)}_j(t_{m-1}), 
\end{eqnarray}
where the random variable $\big((m-1)B^{(m)}_j(1) - \sum_{i=1}^{m-1}B^{(i)}_j(1)\big)/m$ 
has exactly the same law as $\sqrt{1-1/m}B_j(1)$. 
Therefore, the presence of the extra terms involving the $t_i's$ on the right hand-side of \eqref{eq00} allows to distinguish the renormalized limit of $\lcin$ from that of ${\rm L}_n$ and ensures that the latter limit is still almost surely dominated by the former.  
This observation should be contrasted with the non-uniform case where a single letter is attained with maximal probability $p_{\max}$, and where ${\rm L}_n$ aligns this letter.   
Indeed, in view of \eqref{eq00max} below, when centered by $np_{\max}$ and scaled by 
$\sqrt{np_{\max}}$, both $\lcin$ and ${\rm L}_n$ converge to $\min(\sqrt{1-p_{\max}}B_1(1), \sqrt{1-p_{\max}}B_2(1))$. 

A natural question arising from this study is the random matrix interpretation of our limiting distribution \eqref{eq00}. 
Another natural question is to interpret $\lcin$ in terms of RSK Young 
diagrams and to investigate, more generally, the shape of a RSK counterpart of $\lcin$. 
Both questions go actually far beyond the scope of this paper but will be the subject of forthcoming investigations.

\medskip

As for the content of the paper, the next 
section (Section \ref{sec:combinatorics}) establishes a 
pathwise representation for the length of the longest common and increasing 
subsequence of the two words as a max/min functional.  
In Section~\ref{sec:probability}, the probabilistic framework is 
initiated, the representation becomes 
the maximum over a random set of the minimum of random sums of randomly 
stopped random variables.  The various random variables involved are studied 
and their (conditional) laws found. 
In Section \ref{sec:Uniform}, the limiting law is obtained.  This is done in part by a 
derandomization procedure (of the random sums and of the random constraints) leading 
to the Brownian functional \eqref{eq00} of Theorem~\ref{thm1.1}.  
In the last section (Section~\ref{sec:concluding}), various extensions 
and generalizations are discussed as well as some open questions related to this problem.  
Finally, Appendix~\ref{sec:Appendix_lemma} completes the proof of some technical results and while Appendix~\ref{sec:appendix_HLM} gives missing steps in the proof of the main theorem in \cite{HLM} as well as corrections to arguments presented there; providing, in the much simpler binary case, a rather self-contained proof.


\section{Combinatorics}
\label{sec:combinatorics}

The aim of this section is to obtain a pathwise representation for the length of the longest common and increasing subsequences of two finite words.  
Throughout the paper, $X=(X_i)_{i\ge 1}$ and 
$Y=(Y_i)_{i\ge 1}$ are two infinite sequences whose coordinates take 
their values in $\cala_m=\{\aa_1<\aa_2< \cdots <\aa_m\}$, a finite 
totally ordered alphabet of cardinality $m$.   
Recall next that $\lcin$ is the maximal integer $k\in \{1,\dots, n\}$, such
that there exist $1\le i_1< \cdots < i_k\le n$ and $1\le j_1<\cdots <j_k\le n$,
satisfying the following two conditions:

\begin{itemize}
\item[(i)]  $X_{i_s}=Y_{j_s}$, for all $s=1,2,\dots, k$,
\item[(ii)] $X_{i_1}\le X_{i_2}\le \cdots \le X_{i_k}$ and $Y_{j_1}\le Y_{j_2}
\le \cdots \le Y_{j_k}$. 
\end{itemize}

Now that $\lcin$ has been formally defined, let us set some standing 
notation. Let $N_r(X)$, $r=1,\dots, m$, be the number of $\aa_rs$ 
in $X_1,X_2,\dots, X_n$, i.e., 
\begin{align}
\label{nrx}
N_r(X)&=\#\big\{i=1,\dots, n: X_i=\aa_r\big\} =\sum^n_{i=1}\ind_{\{X_i=\alpha_r\}}, 
\end{align}
and similarly let $N_r(Y)$, $r=1,\dots, m$, be the 
number of $\aa_rs$ in $Y_1, Y_2, \dots, Y_n$. Clearly, 
\begin{equation*}
\label{sumnrxy}
\sum^m_{r=1}N_r(X)=\sum^m_{r=1}N_r(Y)=n.
\end{equation*}

\medskip\noindent
Let us further set a convention:  
{\it Throughout the paper when there is no ambiguity or when a property is valid 
for both sequences $X=(X_i)_{i\ge 1}$ and $Y=(Y_i)_{i\ge 1}$ we often omit 
the symbol $X$ or $Y$ and, e.g., write $N_r$ for either $N_r(X)$ or $N_r(Y)$ 
or, below, $H$ for either $H_X$ or $H_Y$.}

\medskip
Continuing on our notational path, for each $r=1,\dots, m$, let $N^{s,t}_r(X)$ be the 
number of $\aa_rs$ in $X_{s+1},X_{s+2},\dots, X_t$, i.e., 
\begin{equation}
\label{eq4}
N^{s,t}_r(X)=\#\big\{i=s+1,\dots, t:X_i=\aa_r\big\}
=\sum_{i=s+1}^{t}\ind_{\{X_i=\aa_r\}},
\end{equation}
with a similar definition for $N^{s,t}_r(Y)$. Again, it is trivially verified that 
\begin{equation*}
\label{eq5}
\sum^m_{r=1}N^{s,t}_r(X)=\sum^m_{r=1}N^{s,t}_r(Y)=t-s,
\end{equation*}
and, of course, $N^{0,n}_r=N_r$. Still continuing with our 
notations, let $T^j_r(X)$, $r=1,\dots, m$, be the location 
of the $j^{\mbox{\footnotesize th}}$ $\aa_r$ in the {\it infinite} 
sequence $X_1,X_2,\dots, X_n,\dots$, with the convention that $T_r^0(X)=0$.  
Then, for $j=1,2,\dots$, $T^j_r(X)$ 
can be defined recursively via, 
\begin{equation}
\label{eq:Tj}
T^{j}_r(X)=\min \big\{s\in \bbn: s>T^{j-1}_r(X),  X_s=\aa_r\big\}
\end{equation}
where as usual $\nit=\{0,1,2,\dots\}$. Again replacing $X$ by $Y$ gives the 
corresponding notion for the sequence $Y=(Y_i)_{i\ge 1}$.

Next, let us begin our finding of a representation for $\lcin$ via the random 
variables defined to date. First, let $H_X(k_1,k_2,\dots, k_{m-1})$ be the maximal 
number of $\aa_ms$ contained in an 
increasing subsequence, of $X_1 X_2\cdots X_n$, 
containing $k_1$ $\aa_1s$, $k_2$ $\aa_2s$, $\dots$, $k_{m-1}$ $\aa_{m-1}s$ picked in that order.  
Replacing $X=(X_i)_{i\ge 1}$ by 
$Y=(Y_i)_{i\ge 1}$, it is then clear that   
\begin{equation}
\label{eq6}
\min\Big(k_1+\cdots +k_{m-1}+H_X(k_1,\dots, k_{m-1}),
k_1+\cdots + k_{m-1}+H_Y(k_1,\dots, k_{m-1})\Big),
\end{equation}
is, therefore, the length of the longest common and 
increasing subsequence of $X_1 X_2\cdots X_n$ and $Y_1Y_2\cdots Y_n$ containing 
exactly $k_r$ $\aa_rs$, for all $r=1,2,\dots, m-1$, 
the letters being picked in an increasing order.  
Hence, to find $\lcin$, the function $H$ needs to be 
identified and \eqref{eq6} needs to be maximized over all 
possible choices of $k_1,k_2,\dots, k_{m-1}$.

Let us start with the maximizing constraints. Assume, for a while, that a single 
word, say, $X_1 \cdots X_n$, is considered. 
First, and clearly, $0\le k_1\le N_1$. Next, $k_2$ is the number of $\aa_2s$ 
present in the sequence after the $k_1^{\mbox{\footnotesize{th}}}$ $\aa_1$.  
Any letter $\aa_2$ is admissible but the ones occurring 
before the $k_1^{\mbox{\footnotesize{th}}}$ $\aa_1$, attained at the location 
$T_1^{k_1}\wedge n$. 
Since there are $n$ letters, considered so far, there are 
thus $N_2^{0, T_1^{k_1}\wedge n}$ inadmissible $\aa_2 s$ and the 
requirement on $k_2$ writes $k_2\leq  N_2-N_2^{0, T_1^{k_1}\wedge n}$. 
Similarly for each $r=3,\dots, m-1$, $k_r$ is the number of 
letters $\aa_r$ minus the inadmissible $\aa_r s$ which occur 
during the recuperation, of the $k_1$ $\aa_1 s$, followed by the $k_2$ $\aa_2 s$, 
followed by the $k_3$ $\aa_3 s$, {\it etc} in that order. 
Thus the requirement on $k_r$ is of the form $k_r\leq N_r-\widetilde N_r^*$, 
where $\widetilde N_r^*$ is the number of $\aa_r s$ occurring 
before the $k_i$ $\aa_i s$, $i\leq r-1$, picked in the order just described. 
For $r=1,2$, and as already shown, $\widetilde N_1^*=0$ and 
$\widetilde N_2^*=N_2^{0, T_1^{k_1}\wedge n}$. 
Assume next that, for $r\ge 3$, $\widetilde N_{r-1}^*$ is well 
defined, then $\widetilde N_r^*$ is the number of $\aa_r s$ occurring 
before, in that order, the $k_1$ $\aa_1 s, \dots,$ the $k_{r-1}$ $\aa_{r-1} s$. 
A little moment of reflection makes it clear that the location 
of the $k_{r-1}^{\mbox{\footnotesize{th}}}$ 
such $\aa_{r-1}$ is $T_{r-1}^{k_{r-1}+ \widetilde N_{r-1}^*}$, from which 
it recursively follows that:   
\begin{equation*}
\label{eq:tildeNr*}
\widetilde N_r^*=N_r^{0,T_{r-1}^{k_{r-1}+\widetilde N_{r-1}^*}\wedge n}. 
\end{equation*}
 
\begin{rem}
\label{rem:depk}
{\rm 
Note that $\widetilde N_r^*$ as well as $N_r^*$ defined below in \eqref{eq:Nr*} actually depend on $k_1, \dots, k_{r-1}$, but in order to not overload our notation we will omit this dependency thereafter.
}
\end{rem}

\medskip\noindent
Returning to two sequences $X_1, \dots, X_n$ and $Y_1, \dots, Y_n$, the condition on $k_r$, $1\leq r\leq m-1$, writes as
\begin{equation*}
\label{eq7}
0\le k_r\le\left(N_r(X) - \widetilde N_r^*(X)\right)\wedge\left(N_r(Y) 
- \widetilde N_r^*(Y)\right).  
\end{equation*}
From these choices of indices and \eqref{eq6}, 
\begin{equation}
\label{eq8}
\lcin=\max_{\widetilde\Ccal_n}
\min\left( \sum^{m-1}_{i=1}k_i+H_X(k_1,\dots, k_{m-1}),
\sum^{m-1}_{i=1}k_i+H_Y(k_1,\dots, k_{m-1})\right), 
\end{equation}
where the outer maximum is taken over $(k_1, \dots, k_{m-1})$ in 
\begin{equation}
\label{eq8Ctilde}
\widetilde\Ccal_n=\Big\{(k_1, \dots, k_{m-1}) : k_1\in \widetilde\Ccal_{n,1}, k_2\in \widetilde\Ccal_{n,2}(k_1), k_3\in \widetilde\Ccal_{n,3}(k_1, k_2), k_{m-1}\in \widetilde\Ccal_{n,m-1}(k_1, \dots, k_{m-2})\Big\},
\end{equation}
where $\widetilde\Ccal_{n,1}= \Big\{ 0\le k_1\le\big(N_1(X)-\widetilde N_1^*(X)\big)\bwedge \big(N_1(Y)-\widetilde N_1^*(Y)\big)\Big\}$ and for $i=2, \dots, m-1$,
\begin{equation}
\label{eq8Citilde}
\widetilde\Ccal_{n,i}{(k_1,\dots, k_{i-1})}= \Big\{ 0\le k_i\le\big(N_i(X)-\widetilde N_i^*(X)\big)\bwedge
\big(N_i(Y)-\widetilde N_i^*(Y)\big)\Big\}.
\end{equation}

Next, observe that if $T_{r-1}^{k_{r-1}+ \widetilde N_{r-1}^*}>n$, 
then $N_r - \widetilde N_r^* = 0$.  Also, since the above maximum does not 
change under vacuous constraints, one can replace in the defining 
constraints, $\widetilde N_r^*$ by $N_r^*$ recursively 
given via:  $N_1^*=0$ and for $r=2,\dots, m-1$,
\begin{equation}
\label{eq:Nr*}
N_r^*=N_r^{0,T_{r-1}^{k_{r-1}+N_{r-1}^*}}. 
\end{equation}
The combinatorial expression \eqref{eq8} then becomes
\begin{equation*}
\label{eq9}
\lcin=\max_{\Ccal_n}\min\left(\sum^{m-1}_{i=1}k_i+H_X(k_1,\dots, k_{m-1}),
\sum^{m-1}_{i=1}k_i+H_Y(k_1,\dots, k_{m-1})\right),
\end{equation*}
where the outer maximum is taken over $(k_1, \dots, k_{m-1})$ in $\Ccal_n$ with $\Ccal_n$ and $\Ccal_{n,i}$, $i=1, \dots, m-1$, respectively defined as in \eqref{eq8Ctilde} and in \eqref{eq8Citilde} but with $\widetilde N_i^*$ replaced by $N_i^*$, $i=1, \dots, m-1$. 
and, of course, 
\begin{equation*}
\label{eq10bis}
\sum_{i=1}^mN_i(X)=\sum_{i=1}^mN_i(Y)=n.  
\end{equation*}

After this identification, recall that $H$ is the maximal number 
of $\aa_m$ after, in that order, the $k_1$ $\aa_1 s$, $k_2$ $\aa_2 s$, 
$\dots$, $k_{m-1}$ $\aa_{m-1} s$.  
Counting the $\aa_ms$ present between the various locations 
of the $\aa_i$, $i=1,\dots, m-1$, and after another moment of reflection, 
it is clear that  
$$
H=N_m-R,
$$
where
\begin{equation}
\label{eq11}
R=
\sum^{m-1}_{i=1}\sum^{N^*_i+k_i}_{j=N^*_i+1} N^{T^{j-1}_i,T^j_i}_m
-\sum_{i=1}^{m-1}N_m^{T_i^{N_i^*}, T_{i-1}^{N_{i-1}^*+k_{i-1}}}
:=R_1-R_2,
\end{equation}
where the $N^*_i$ are given by \eqref{eq:Nr*} and where
$$
R_1=\sum^{m-1}_{i=1}\sum^{N^*_i+k_i}_{j=N^*_i+1} N^{T^{j-1}_i,T^j_i}_m
=\sum_{i=1}^{m-1} N_m^{T_i^{N_i^*}, T_i^{N_i^*+k_i}},
$$
while 
$$
R_2=\sum_{i=1}^{m-1}N_m^{T_i^{N_i^*}, T_{i-1}^{N_{i-1}^*+k_{i-1}}},
$$
and therefore as expected $R=N_m^{0, T_{m-1}^{N_{m-1}^*+k_{m-1}}}$.  
Recall also that according to Remark~\ref{rem:depk}, $R$ actually depends  on $k_1, \dots, k_{m-1}$, but that for the sake of readability this dependency is omitted from our notations.
Summarizing our results leads so far to:  
\begin{thm}
\label{thm2.1}
Let $X=(X_i)_{i\ge 1}$ and $Y=(Y_i)_{i\ge 1}$ be two sequences whose coordinates take their values in $\cala_m=\{\aa_1<\aa_2< \cdots <\aa_m\}$, a totally ordered finite alphabet of cardinality~$m$. 
Let $\lcin$ be the length of the longest common and increasing subsequences of $X_1 \cdots X_n$ and $Y_1 \cdots Y_n$. 
Then,  
\begin{equation}
\label{eq13}
\lcin=\max_{\Ccal_n}\min\left(
\sum^{m-1}_{i=1}k_i+N_m(X)-R(X), \sum^{m-1}_{i=1}k_i+N_m(Y)-R(Y)\right),
\end{equation}
where the outer maximum is taken over $(k_1, \dots, k_{m-1})$ in 
\begin{equation}
\label{eq8C}
\Ccal_n=\Big\{(k_1, \dots, k_{m-1}) : k_1\in \Ccal_{n,1}, k_2\in \Ccal_{n,2}(k_1), k_3\in \Ccal_{n,3}(k_1, k_2), k_{m-1}\in\Ccal_{n,m-1}(k_1, \dots, k_{m-2})\Big\},
\end{equation}
where $\Ccal_{n,1}= \Big\{0\le k_1\le\big(N_1(X)-N_1^*(X)\big)\bwedge \big(N_1(Y)-N_1^*(Y)\big)\Big\}$ and for $i=2, \dots, m-1$,
\begin{equation}
\label{eq8Ci}
\Ccal_{n,i}{(k_1,\dots, k_{i-1})}= \Big\{ k=(k_1, \dots, k_{n-1})\ :\ 0\le k_i\le\big(N_i(X)-N_i^*(X)\big)\bwedge \big(N_i(Y)-N_i^*(Y)\big)\Big\},
\end{equation}
and where
$$
R=\sum^{m-1}_{i=1}\sum^{N^*_i+k_i}_{j=N^*_i+1}N^{T^{j-1}_i,T^j_i}_m
-\sum_{i=1}^{m-1}N_m^{T_i^{N_i^*}, T_{i-1}^{N_{i-1}^*+k_{i-1}}},
$$
with the various $N$'s and $T$'s given above by \eqref{nrx}, \eqref{eq4}, \eqref{eq:Tj} and \eqref{eq:Nr*}.
\end{thm}

The representation \eqref{eq13} has the great advantage of 
(essentially) only involving the quantities $N_i$, $N_i^*$, $i=1, 2, \dots, m-1$ 
and $T^j_i$, $i=1,2,\dots, m-1$, $j=1, 2, \dots$, and $N_m$.


\section{Probability}
\label{sec:probability}

Let us now bring our probabilistic framework into the picture by first studying the random variables $N^{T^{j-1}_i,T^j_i}_m$, $i=1,2,\dots, m-1$ and $j=1,2,\dots$ and then the random variables $N_i^*$, $i=1,2,\dots, m-1$, appearing in $R$ in \eqref{eq11}.

\begin{prop}
\label{lemind} 
Let $(Z_n)_{n\ge 1}$ be a sequence of iid random variables 
with $\bbp (Z_1=\aa_i)=p_i$, $i=1,\dots,m$.  
For each $i=1, 2, \dots, m$, let $T^0_i=0$, and let $T^j_i$, $j=1, 2, \dots$ be 
the location of the $j^{\mbox{\footnotesize{th}}}$ $\aa_i$ in the infinite 
sequence $(Z_n)_{n\ge 1}$.  Let $i, r\in \{1, \dots, m\}$, with $r\neq i$.  Then, 
for any $j=1, 2,\dots$, the conditional law 
of $N^{T^{j-1}_i,T^j_i}_r$ 
given $(T^{j-1}_i,T^j_i)$, is binomial 
with parameters $T^j_i-T^{j-1}_i-1$ and $p_r/(1-p_i)$, which we 
denote by ${\cal B}\big(T^j_i-T^{j-1}_i-1,p_r/(1-p_i)\big)$.  
Moreover, the conditional law  of $\big(N^{T^{j-1}_i,T^j_i}_r\big)_{r=1, \dots, m, r\neq i}$ 
given $(T^{j-1}_i,T^j_i)$, is multinomial with 
parameters $T^j_i-T^{j-1}_i-1$ and $(p_r/(1-p_i))_{r=1, \dots, m, r\neq i}$, which we 
denote by ${\cal M}{\bf{\it ul}}\big(T^j_i-T^{j-1}_i-1,(p_r/(1-p_i))_{r=1, 
\dots, m, r\neq i}\big)$.
Finally, for each $i\not=r$, the random variables $\big(N^{T^{j-1}_i,T^j_i}_r\big)_{j\ge 1}$,  
are independent with mean $p_r/p_i$ and variance $(p_r/p_i)(1+p_r/p_i)$; and, moreover,   
they are identically distributed in case the 
$(Z_n)_{n\ge 1}$, are uniformly distributed.
\end{prop}

\begin{proof} 
Let us denote by $\call\big(N^{T^{j-1}_i,T^j_i}_r \big| T^{j-1}_i,T^j_i\big)$ 
the conditional law of $N^{T^{j-1}_i,T^j_i}_r$ given $T^{j-1}_i,T^j_i$. 
Recall, see \eqref{eq:Tj}, that $T^{j-1}_i$ and $T^j_i$ are the respective locations 
of the 
$(j-1)^{\mbox{\footnotesize{th}}}$ $\aa_i$ 
and the $j^{\mbox{\footnotesize{th}}}$ $\aa_i$ 
in the infinite sequence $(Z_n)_{n\geq 1}$.
Thus between $T^{j-1}_i+1$ and $T^j_i$, 
there are $T^j_i-T^{j-1}_i-1$ free spots 
and each one is equally likely to contain $\aa_r$, $r\ne i$,
with probability 
$p_r/(\sum^m_{\stackrel{\ell=1}{\ell\ne i}}p_\ell)= p_r/(1-p_i)$. 
Therefore,
\begin{equation}
\label{eq14}
\call\left(N^{T^{j-1}_i,T^j_i}_r\big| T^{j-1}_i,T^j_i\right)=
{\cal B}\left(T^j_i-T^{j-1}_i-1,\frac{p_r}{1-p_i}\right).
\end{equation}
Let us now compute the probability generating function of the 
random variables $N^{T^{j-1}_i,T^j_i}_r$, $i\not=r$. 
First, via \eqref{eq14}
\begin{align}
\nonumber
\ee\left[x^{N^{T^{j-1}_i,T^j_i}_r}\right]&=
\ee \left[\ee\left[ x^{N^{T^{j-1}_i,T^j_i}_r}\Big| T^{j-1}_i, T^j_i\right]
\right]\\
\nonumber
&= \sum^\infty_{\ell=1}\left(1-\frac{p_r}{1-p_i}+\frac{p_r}{1-p_i}x\right)
^{\ell-1} p_i(1-p_i)^{\ell-1}\\
\nonumber
&=\frac{p_i}{1-(1-p_i)\left(1-\frac{p_r}{1-p_i}+\frac{p_r}{1-p_i} x\right)}
\\
\label{eq:PGFN}
&=\frac{p_i}{p_i+p_r-p_r x},
\end{align}
since $T_i^j$ is a negative binomial (Pascal) random variable with 
parameters $j$ and $p_i$ which we shall denote ${\cal BN}(j,p_i)$ in the 
sequel and $T_i^j-T_i^{j-1}$ is a geometric random variables 
with parameter $p_i$, which we shall denote ${\cal G}(p_j)$. 
Therefore,
\begin{align}
\nonumber
\ee \left[N^{T^{j-1}_i,T^j_i}_r\right] &=\frac{p_r}{p_i},\\
\label{eq3.4}
\var\left(N^{T^{j-1}_i, T^j_i}_r\right) &= \frac{p_r}{p_i}\left(
1+\frac{p_r}{p_i}\right).
\end{align}
In the uniform case, i.e., $p_i=1/m$, $i=1,\dots, m$, 
the $N^{T^{j-1}_i,T^j_i}_r$, $i=1,\dots, m$, $i\ne r$, $j=1,2,\dots$ are clearly seen to be 
identically distributed, via \eqref{eq:PGFN}.  
The multinomial part of the statement is proved in a very similar manner.  
The $T^j_i-T^{j-1}_i-1$ free spots are to contain the letters $\aa_r$, $r\in \{1, \dots, m\}, 
r \neq i$,
with respective probabilities $p_r/(1-p_i)$. 
Therefore,
\begin{equation}
\label{eq14bis}
\call\left(\big(N^{T^{j-1}_i,T^j_i}_r\big)_{r=1, \dots m, r\neq i}\big| T^{j-1}_i,T^j_i\right)=
{\cal M}{\bf\it ul}\left(T^j_i-T^{j-1}_i-1,
\left(\frac{p_r}{1-p_i}\right)_{r=1, \dots, m, r\neq i}\right).
\end{equation}
Via \eqref{eq14bis}, the probability generating function of the 
random vector $\big(N^{T^{j-1}_i,T^j_i}_r\big)_{r=1, \dots, m, r\neq i}$ is then given by:    
\begin{align}
\nonumber
\ee\left[\prod_{r=1,r\neq i}^mx_r^{N^{T^{j-1}_i,T^j_i}_r}\right]&=
\ee \left[\ee\left[\prod_{\begin{subarray}{c} 
r=1\\r\neq i\end{subarray}}^m x_r^{N^{T^{j-1}_i,T^j_i}_r}\Big| T^{j-1}_i, T^j_i\right]
\right]\\
\nonumber
&= \sum^\infty_{\ell=1}\left(\sum_{\begin{subarray}{c} r=1\\r\neq i\end{subarray}}^m
\frac{p_r}{1-p_i}x_r\right)
^{\ell-1} p_i(1-p_i)^{\ell-1}\\
&=\frac{p_i}{1-\sum_{r=1, r\neq i}^mp_rx_r}.  
\label{eq:PGFNbis}
\end{align}
As a direct consequence of \eqref{eq:PGFNbis} and for $r\neq i, s \neq i$,   
\begin{equation*}
\label{eq:cov}
\cov\left(N^{T^{j-1}_i, T^j_i}_r, N^{T^{j-1}_i, T^j_i}_s \right) = \frac{p_rp_s}{p_i^2}.
\end{equation*}

The proof of the proposition will be complete once, for each $i\not=r$, the random 
variables $N^{T^{j-1}_i,T^j_i}_r$, $j\geq 1$, are shown to be independent. 
First, note that given $T^{j-1}_i, T^j_i, T^{k-1}_i, T^k_i$, the random variables 
$N^{T^{j-1}_i,T^j_i}_r=\sum_{\ell=T^{j-1}_i+1}^{T^{j}_i} \ind_{\{X_\ell=\aa_r\}}$ 
and $N^{T^{k-1}_i,T^k_i}_r=\sum_{\ell=T^{k-1}_i+1}^{T^{k}_i} \ind_{\{X_\ell=\aa_r\}}$ 
are independent since the intervals $[T^{j-1}_i+1, T^j_i]$ and $[T^{k-1}_i+1, T^k_i]$ 
are disjoint, 
and since the $(X_\ell)_{\ell\ge 1}$ are also independent.
Moreover, recall that conditional distributions are given by \eqref{eq14}, and 
so, for instance,
\begin{align*}
{\cal L}\left(N^{T^{j-1}_i,T^j_i}_r\big| T^{j-1}_i, T^j_i, T^{k-1}_i, T^k_i\right)
&={\cal L}\left(N^{T^{j-1}_i,T^j_i}_r\big| T^{j-1}_i, T^j_i\right)\\
&={\cal B}\left(T^j_i-T^{j-1}_i-1,\frac{p_r}{1-p_i}\right).
\end{align*} 
Therefore, for any measurable functions $f, g :\bbR_+\to\bbR_+$, and 
if $\ee_{{\cal B}(n,p)}$ denotes the expectation with respect to a 
binomial ${\cal B}(n,p)$ distribution then
\begin{align}
\nonumber
&\ee\left[f\big(N^{T^{j-1}_i,T^j_i}_r\big)g\big(N^{T^{k-1}_i,T^k_i}_r\big)\right]\\
\nonumber
&=\ee\left[\ee\left[f\big(N^{T^{j-1}_i,T^j_i}_r\big)g\big(N^{T^{k-1}_i,T^k_i}_r\big)
\big| T^{j-1}_i, T^j_i, T^{k-1}_i, T^k_i\right]\right]\\
\label{eq:indep1}
&=\ee\left[\ee\left[f\big(N^{T^{j-1}_i,T^j_i}_r\big)\big| T^{j-1}_i, T^j_i, 
T^{k-1}_i, T^k_i\right]
\ee\left[g\big(N^{T^{k-1}_i,T^k_i}_r\big)\big| T^{j-1}_i, T^j_i, 
T^{k-1}_i, T^k_i\right]\right]\\
\nonumber
&=\ee\left[\ee_{{\cal B}\left(T^j_i-T^{j-1}_i-1,\frac{p_r}{1-p_i}\right)}[f]\ 
\ee_{{\cal B}\left(T^k_i-T^{k-1}_i-1,\frac{p_r}{1-p_i}\right)}[g]\right]\\
\label{eq:indep2}
&=\ee\left[\ee_{{\cal B}\left(T^j_i-T^{j-1}_i-1,\frac{p_r}{1-p_i}\right)}[f]\right] \ 
\ee\left[\ee_{{\cal B}\left(T^k_i-T^{k-1}_i-1,\frac{p_r}{1-p_i}\right)}[g]\right]\\
\nonumber
&=\ee\left[f\big(N^{T^{j-1}_i,T^j_i}_r\big)\right]\ 
\ee\left[g\big(N^{T^{k-1}_i,T^k_i}_r\big)\right], 
\end{align}
where the equality in \eqref{eq:indep1} is due to the conditional 
independence property, while the one in \eqref{eq:indep2} follows from that 
$$
\ee_{{\cal B}\left(T^j_i-T^{j-1}_i-1,\frac{p_r}{1-p_i}\right)}[f]=F\big(T^j_i-T^{j-1}_i\big)
\quad \mbox{ and }\quad
\ee_{{\cal B}\left(T^k_i-T^{k-1}_i-1,\frac{p_r}{1-p_i}\right)}[g]=G\big(T^k_i-T^{k-1}_i\big), 
$$
for some functions $F, G$, and from the independence of 
$T^j_i-T^{j-1}_i$ and $T^k_i-T^{k-1}_i$. 
The argument can then be easily adapted to justify the mutual 
independence of the random variables $\big(N^{T^{j-1}_i,T^j_i}_r\big)_{j\ge 1}$.
 \end{proof}

\noindent
With the help of the previous proposition and in order to prepare our first fluctuation result, it is relevant to rewrite the representation \eqref{eq13} as 
\begin{align}\label{eq17}
\lcin&=
 \max_{\Ccal_n}\min
\left\{\sum^{m-1}_{i=1} k_i+N_m(X)-G_{n,m}(X),
\sum^{m-1}_{i=1}k_i+N_m(Y)-G_{n,m}(Y)\right\},
\end{align}
where
\begin{equation}\label{indG}
G_{n,m}=\sum^{m-1}_{i=1}
\sum^{N^*_i+k_i}_{j=N^*_i+1}\Biggl(\left(\frac{N^{T^{j-1}_i,T^j_i}_m-
\frac{p_m}{p_i}}{\sqrt{\frac{p_m}{p_i}\left(1+\frac{p_m}{p_i}\right)n}}\right)
\sqrt{\frac{p_m}{p_i}\left(1+\frac{p_m}{p_i}\right)n}+\frac{p_m}{p_i}\Biggr)
-\sum_{i=1}^{m-1}N_m^{T_i^{N_i^*}, T_{i-1}^{N_{i-1}^*+k_{i-1}}},
\end{equation}
and where $p_i(X)=\bbp(X_1=\aa_i)$ and $p_i(Y)=\bbp(Y_1=\aa_i)$, $1\leq i\leq m$.
Recall once more that  $G_{n,m}$ actually depends on $k_1, \dots, k_{m-1}$ but that, for the sake of readability, this dependency is omitted from our notations, see Remark~\ref{rem:depk}.

\medskip
Via \eqref{eq17} and \eqref{indG}, $\lcin$ is now represented as a max/min over 
random constraints of random sums of randomly stopped independent random variables, 
except for the presence of $N_m(X)$ and $N_m(Y)$. 
Our next result also represents, up to a small error term, both $N_m(X)$ and $N_m(Y)$ 
via the same random variables.
\begin{prop} 
\label{prop:S}
For each $i=1,2,\dots, m$, and $r\not=i$, 
\begin{equation}
\label{eq18}
N_r=\frac{p_r}{p_i}N_i+\sum^{N_i}_{j=1}
\frac{\left(N^{T^{j-1}_i,T^j_i}_r-\frac{p_r}{p_i}\right)}
{\sqrt{\frac{p_r}{p_i}\left(1+\frac{p_r}{p_i}\right)n}}
\sqrt{\frac{p_r}{p_i}\left(1+\frac{p_r}{p_i}\right) n} + S_{i,r}^{(n)}, 
\end{equation}
where $\lim_{n\to+\infty} {S_{i,r}^{(n)}}/{\sqrt n}=0$, in probability.
In particular, for each $r=1, 2, \dots, m$, 
\begin{equation}
\label{eq19}
N_r= np_r+\sum^{m}_{\begin{subarray}{c}i=1\\ i\not=r\end{subarray}}\sqrt{\frac{p_r}
{p_i}\left(1+\frac{p_r}{p_i}\right)n}
p_i\sum^{N_i}_{j=1}\frac{\left(N^{T^{j-1}_i,T^j_i}_r-\frac{p_r}{p_i}\right)}
{\sqrt{\frac{p_r}{p_i}\left(1+\frac{p_r}{p_i}\right) n}}
+\sum^{m}_{\begin{subarray}{c} i=1\\i\not=r\end{subarray}}p_iS_{i,r}^{(n)}.
\end{equation}
\end{prop}

\begin{proof}
Let us start the proof of \eqref{eq18} by identifying the random variable $S_{i,r}^{(n)}$ and show that, when scaled by $\sqrt n$, they converge in probability to zero. 
Clearly, for $i=1,\dots, m$, $i\not=r$,
\begin{equation*}
\label{eq22}
0\le S_{i,r}^{(n)}:=N_r-\sum^{N_i}_{j=1} N^{T^{j-1}_i,T^j_i}_r=N_r-N_r^{T_i^*}.
\end{equation*}
In other words,  $S_{i,r}^{(n)}$ is the number of $\aa_r$ in the interval
$[T_i^*+1,n]$, where $T_i^*$ is the location of the last $\aa_i$ in $[1,n]$.
Therefore,
\begin{equation}\label{eq23}
0\le S_{i,r}^{(n)}\le n-T_i^*=n-(T^{N_i}_i\wedge n).
\end{equation}
But, $\bbp(T_i^*=n-k)=p_i(1-p_i)^k$, $k=0,1,\dots, n-1$ and $\bbp (T_i^*=0)=(1-p_i)^n$. 
Hence, for all $\epsilon >0$, and $n$ large enough, 
\begin{equation}\label{convS}
\bbp \left(\frac{S_{i,m}^{(n)}}{\sqrt n}\ge\epsilon\right)
\leq \bbp (n-T_i^*\ge\epsilon\sqrt n)
\leq \sum_{l=[\epsilon \sqrt n]}^n p_i(1-p_i)^l
\leq (1-p_i)^{[\epsilon \sqrt n]} 
\mathop{\longrightarrow}\limits_{n\to+\infty} 0.
\end{equation}
Let us continue with the proof of \eqref{eq19}. 
Summing over $i=1,\dots, m$, $i\not=r$, both sides 
of \eqref{eq18}, we get
\begin{align}
\label{eq20}
\sum^{m}_{\begin{subarray}{c} i=1\\i\not=r\end{subarray}}\frac{p_i}{p_r}N_r
=\sum^{m}_{\begin{subarray}{c} i=1\\i\not=r\end{subarray}}N_i
+\sum^{m}_{\begin{subarray}{c} i=1\\i\not=r\end{subarray}}
\sqrt{\frac{p_r}{p_i}\left(1+\frac{p_r}{p_i}\right) n}
\frac{p_i}{p_r}\left(\sum^{N_i}_{j=1}\frac{\left(N^{T^{j-1}_i,T^j_i}_r-
\frac{p_r}{p_i}\right)}{\sqrt{\frac{p_r}{p_i}\left(1+\frac{p_r}{p_i}\right) n}}
\right)+\sum^{m}_{\begin{subarray}{c} i=1\\i\not=r\end{subarray}}\frac{p_i}{p_r} S_{i,r}^{(n)}. 
\end{align}
But, $\sum^m_{i=1}N_i=n$, and so \eqref{eq20} becomes 
\begin{equation*}
\label{eq21}
N_r=np_r+\sum^{m}_{\begin{subarray}{c} i=1\\i\not=r\end{subarray}}\sqrt{\frac{p_r}
{p_i}\left(1+\frac{p_r}{p_i}\right) n}p_i
\left(\sum^{N_i}_{j=1}\frac{\left(N^{T^{j-1}_i,T^j_i}_r-\frac{p_r}{p_i}\right)}
{\sqrt{\frac{p_r}{p_i}\left(1+\frac{p_r}{p_i}\right) n}}\right)
+\sum^{m}_{\begin{subarray}{c} i=1\\i\not=r\end{subarray}}p_iS_{i,r}^{(n)},
\end{equation*}
which is precisely \eqref{eq19}. 
\end{proof}
\begin{rem}
\label{rem:Sim}
{\rm
For all $i\not=r$, $\lim_{n\to+\infty} \ee\big[(S_{i,r}^{(n)})^2/n\big]=0$.
Indeed, 
\begin{eqnarray*}
\ee\big[(S_{i,r}^{(n)})^2/n\big]
&=& \int_0^{+\infty} \pp\big((S_{i,m}^{(n)})^2\geq xn\big)dx
\leq\int_0^{+\infty} (1-p_i)^{[\sqrt{xn}]}dx\\
&\leq&\int_0^{+\infty} (1-p_i)^{\sqrt{xn}-1}dx
=\frac2{n(1-p_i)(\ln (1-p_i))^2}.
\end{eqnarray*}

}
\end{rem}

\medskip\noindent
Returning to the representation \eqref{eq17}, the previous proposition allows us to rewrite $\lcin$ as:
\begin{eqnarray}
\nonumber
\lcin&=&\max_{\bigcap^{m-1}_{i=1}\Ccal_{n,i}}\min
\Biggl(np_m(X)+\sum^{m-1}_{i=1}k_i-p_m(X)\sum^{m-1}_{i=1}\frac{k_i}{p_i(X)}\\
\nonumber
&&\hspace{4cm} +H_{m,n}(X)+K_{m,n}(X)
+\sum^{m-1}_{i=1}p_i(X)S_{i,m}^{(n)}(X),\\
\nonumber
&&
np_m(Y)+\sum^{m-1}_{i=1}k_i-p_m(Y)\sum^{m-1}_{i=1}\frac{k_i}{p_i(Y)}\\
\label{eq24}
&&\hspace{4cm}+H_{m,n}(Y)+K_{m,n}(Y)
+\sum^{m-1}_{i=1}p_i(Y)S_{i,m}^{(n)}(Y)\!\Biggr),
\end{eqnarray}
where omitting the dependency in  $k_1, \dots, k_{m-1}$ (see Remark~\ref{rem:depk}),
\begin{eqnarray}
\label{eq25}
\lefteqn{H_{m,n}=\sum^{m-1}_{i=1}\sqrt{\frac{p_m}{p_i}\left(1+\frac{p_m}{p_i}\right)
n} p_i\sum^{N_i}_{j=1}
\frac{\left(N^{T^{j-1}_i,T^j_i}_m-\frac{p_m}{p_i}\right)}
{\sqrt{\frac{p_m}{p_i}\left(1+\frac{p_m}{p_i}\right) n}}}
\nonumber\\
&&\hspace{3cm}-\sum^{m-1}_{i=1} \sqrt{\frac{p_m}{p_i}\left(1+\frac{p_m}{p_i}\right) n}
\sum^{N^*_i+k_i}_{j=N^*_i+1}
\frac{\left(N^{T^{j-1}_i,T^j_i}_m-\frac{p_m}{p_i}\right)}
{\sqrt{\frac{p_m}{p_i}\left(1+\frac{p_m}{p_i}\right) n}},
\end{eqnarray}
and 
\begin{equation}
\label{eq:Kmn}
K_{m,n}=\sum_{i=1}^{m-1} N_m^{T_i^{N_i^*}, T_{i-1}^{N_{i-1}^*+k_{i-1}}}. 
\end{equation}
We now study some of the properties of the random variables $N_i^*$ which are present in both the random constraints and the random sums.  
The random variables $N_i^*$ are defined recursively by \eqref{eq:Nr*} with $N_1^*=0$. We fix ${\bf k}=(k_1, \dots, k_{m-1})$ where $k_i$ is the number of letters $\aa_i$ present in the common increasing subsequences. 
The random variables $N_i^*$, $i\ge 2$, depend on $\bf k$, actually $N_i^*=N_i^*(k_1, \dots, k_{i-1})$. 
We write 
\begin{equation}
\label{eq:NiNij}
N_i^*=\sum_{j=1}^{i-1} N_{i,j}^*
\end{equation}
where $N_{i,j}^*=N_{i,j}^*(k_j)$ is the number of letters $\aa_i$ present 
in the step $j\le i-1$ consisting in collecting the $k_j$ letters $\aa_j$, $j\le i-1$. 
(In the sequel, in order not to further burden the notations, 
we shall skip the symbols ${k_j}$, 
$j=1, \dots, i-1,$ in $N_i^*$ and $N_{i,j}^*$.)   
The following diagram encapsulates the drawing of the letters: 
{\tiny
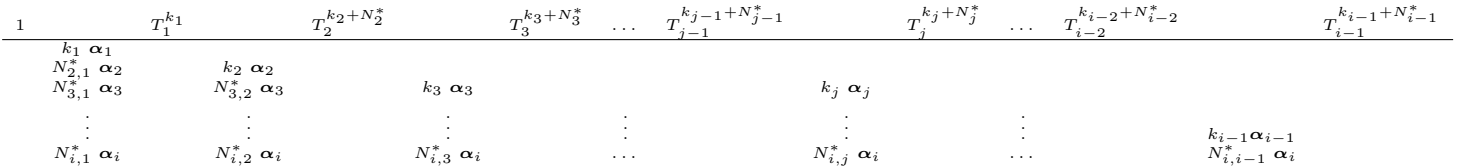
$$
\hskip -20pt
\begin{array}{ccccccccccccccc}
1&&T_1^{k_1}&&T_2^{k_2+N_2^*}&&T_3^{k_3+N_3^*}&\dots&T_{j-1}^{k_{j-1}+N_{j-1}^*}
&&T_j^{k_j+N_j^*}&\dots&T_{i-2}^{k_{i-2}+N_{i-2}^*}&&T_{i-1}^{k_{i-1}+N_{i-1}^*}\\
\hline
&k_1 \ \aa_1&&&&&&&&&&&&&\\
&N_{2,1}^*\ \aa_2&&k_2\ \aa_2&&&&&&&&&&\\
&N_{3,1}^* \ \aa_3&&N_{3,2}^* \ \aa_3&& k_3 \ \aa_3&&&& k_j \ \aa_j&&&&\\
&\vdots&&\vdots&&\vdots&&\vdots&&\vdots&&\vdots&&k_{i-1}\aa_{i-1}\\
&N_{i,1}^* \ \aa_i&& N_{i,2}^*\  \aa_i&& N_{i,3}^* \ \aa_i&&\dots&&N_{i,j}^*\ \aa_i&&\dots
&&N_{i, i-1}^* \ \aa_i
\end{array}
$$
}
In Step $j\le i-1$, there are $T_j^{k_j+N_j^*}-T_{j-1}^{k_{j-1}+N_{j-1}^*}$ letters 
selected but $k_j$ letters are $\aa_j$, $N_{j+1,j}^*$ 
are $\aa_{j+1}$, \dots, $N_{i-1,j}^*$ are $\aa_{i-1}$, (for $j=i-1$, there are 
also $k_j$ letters $\aa_j$ but none of the others $\aa_{j+1}$, {\it etc}). 

Moreover, there are $T_j^{k_j+N_j^*}-T_{j-1}^{k_{j-1}+N_{j-1}^*}-k_j-N_{j+1,j}^*-\dots-N_{i-1,j}^*$ 
possible spots ($T_j^{k_j+N_j^*}-T_{j-1}^{k_{j-1}+N_{j-1}^*}-k_j$ in case $j=i-1$) 
in which the probability 
of having a $\aa_i$ is  $p_{i,j}:=p_i/(1-p_j-\dots-p_{i-1})$.  
Therefore, conditionally on 
$$
{\cal G}_{i,j}({\bf k})=\sigma\left(N_{j+1,j}^*, 
\dots, N_{i-1,j}^*, T_{j-1}^{k_{j-1}+N_{j-1}^*}, T_j^{k_j+N_j^*}\right),
$$ 
(the $\sigma$-field generated by $N_{j+1,j}^*, 
\dots, N_{i-1,j}^*, T_{j-1}^{k_{j-1}+N_{j-1}^*}, T_j^{k_j+N_j^*}$) it follows that  
\begin{equation}
\label{eq:Nij}
N_{i,j}^* \sim {\cal B}\Big(T_j^{k_j+N_j^*}-T_{j-1}^{k_{j-1}+N_{j-1}^*}
-k_j-N_{j+1,j}^*-\dots-N_{i-1,j}^*, p_{i,j}\Big).
\end{equation}

\medskip\noindent
The two forthcoming propositions respectively characterize the 
laws of $N_{i,j}^*$ and of $N_i^*$. 
\begin{prop} 
\label{prop:law_Nij}
For each $i=2, \dots, m$, the probability generating 
function of $N_{i,j}^*$, $1\leq j\leq i-1$, is given by
\begin{equation}
\label{eq:LaplaceNij}
\ee\Big[x^{N_{i,j}^*}\Big]=\left(\frac{p_j}{p_j+p_i-p_ix}\right)^{k_j}.
\end{equation}
Therefore, $N_{i,j}^*$ is distributed as $\sum_{\ell=1}^{k_j} (G_\ell-1)$, 
where $(G_\ell)_{1\leq \ell\leq k_j}$ are independent 
with geometric law ${\cal G}\big(p_j/(p_j+p_i)\big)$ and so,
\begin{equation}
\label{eq:EVar_Nij}
\ee[N_{i,j}^*]=\frac{p_i}{p_j}k_j
\quad \mbox{ and } \quad 
\Var(N_{i,j}^*)=\left(1+\frac{p_i}{p_j}\right)\frac{p_i}{p_j}k_j.
\end{equation}
\end{prop}
\begin{proof}
Recall that, for $N\sim {\cal B}(n,p)$, $\ee[x^N]=(1-p+px)^n$ 
while, for $N\sim{\cal G}(p)$, $\ee[x^N]=
{px/(1-(1-p)x)}$. 
Using \eqref{eq:Nij}, we then have for 
$N=T_j^{k_j+N_j^*}-T_{j-1}^{k_{j-1}+N_{j-1}^*}-k_j-N_{j+1,j}^*-\cdots-N_{i-1,j}^*$,
\begin{eqnarray}
\nonumber
\ee\Big[x^{N_{i,j}^*}\Big]
&=&\ee\Big[\ee\big[x^{N_{i,j}^*}\big|N\big]\Big]\\
\nonumber
&=&\ee\Big[(1-p_{i,j}+p_{i,j}x)^{T_j^{k_j+N_j^*}-T_{j-1}^{k_{j-1}+N_{j-1}^*}
-k_j-N_{j+1,j}^*-\dots-N_{i-1,j}^*}\Big]\\
\label{eq:techLaplace}
&=&\ee\Big[y^{U-V}\Big],
\end{eqnarray}
setting $y=(1-p_{i,j}+p_{i,j}x)$, and 
\begin{eqnarray}
\label{eq:U-L}
U&:=&T_j^{k_j+N_j^*}-T_{j-1}^{k_{j-1}+N_{j-1}^*}-k_j 
\sim {\cal BN}(k_j, p_j)\ast\delta_{-k_j}\\
\label{eq:V-L}
V&:=&\sum_{r=j+1}^{i-1}N_{r,j}^*\sim {\cal B}
\Big(U, \sum_{r=j+1}^{i-1} \frac{p_r}{1-p_j}\Big),
\end{eqnarray}
where for $j=i-1$, we also set $V=0$. 
The notation ${\cal BN}(k, p)$ above stands for the negative binomial (Pascal)
distribution with 
parameters $k$ and $p$. The parameters of the binomial random variables $V$ 
in \eqref{eq:V-L} stem from 
that $V$ counts the number of letters $\aa_r$, $j+1\leq r\leq i-1$, between 
two letters $\aa_j$, 
while exactly $k_j$ such letters are obtained, so that 
each $\aa_r$ has probability $p_r/(1-p_j)$ to appear. 
Hence, 
\begin{eqnarray*}
\ee\Big[y^{U-V}\Big]&=&\ee\Big[\ee\big[y^{U-V}|U\big]\Big]\\
&=&\ee\Big[y^U\ee\big[y^{-V}|U\big]\Big]\\
&=&\ee\Big[y^U\Big(1-\sum_{r=j+1}^{i-1} \frac{p_r}{1-p_j}+\frac{\sum_{r=j+1}^{i-1} p_r}
{(1-p_j)y}\Big)^U\Big]\\
&=&\ee\Big[\Big(\big(1-\sum_{r=j+1}^{i-1} \frac{p_r}{1-p_j}\big)y+\sum_{r=j+1}^{i-1} 
\frac{p_r}{1-p_j}\Big)^{G_1-1}\Big]^{k_j},
\end{eqnarray*}
since, from \eqref{eq:U-L}, $U\sim\sum_{\ell=1}^{k_j}(G_\ell-1)$, 
where the $G_\ell$, $1\leq \ell \leq k_j$, are iid with distribution ${\cal G}(p_j)$. 
Finally,  
\begin{eqnarray*}
\ee\Big[y^{U-V}\Big]&=&
\left(\frac{p_j}{1-(1-p_j)\Big(\big(1-\sum_{r=j+1}^{i-1}\frac{p_r}{1-p_j}\big)y
+\sum_{r=j+1}^{i-1}\frac{p_r}{1-p_j}\Big)}\right)^{k_j}\\
&=&\left(\frac{p_j}{p_j+p_i-p_ix}\right)^{k_j},
\end{eqnarray*}
since $p_{i,j}=p_i/\big(1-\sum_{r=j}^{i-1}p_r\big)$. 
The expressions for the expectation and for the 
variance in \eqref{eq:EVar_Nij} follow from straightforward computations. 
\end{proof}

Recall that by convention, $N_1^*=0$, and for $2\le i\le m$, the 
following proposition gives the law of $N_i^*$: 
\begin{prop} 
\label{prop:law_Ni}
For each $i=2, \dots, m$, the random 
variables $(N_{i,j}^*)_{1\leq j\leq i-1}$ are independent. 
Hence, the probability generating function of $N_i^*$ is given by
\begin{eqnarray}
\label{eq:LaplaceNi}
\ee\Big[x^{N_i^*}\Big]&=&\prod_{j=1}^{i-1}
\left(\frac{p_j}{p_j+p_i-p_ix}\right)^{k_j}, 
\end{eqnarray}
and so, 
\begin{equation}
\label{eq:EVar_Ni}
\ee[N_i^*]=\sum_{j=1}^{i-1} \frac{p_i}{p_j}k_j
\quad \mbox{ and } \quad 
\Var(N_i^*)
=\sum_{j=1}^{i-1}\left(1+\frac{p_i}{p_j}\right)\frac{p_i}{p_j}k_j.
\end{equation}
\end{prop} 
\begin{proof}
In view of Proposition~\ref{prop:law_Nij} and of \eqref{eq:NiNij}, it is enough 
to prove the first part of the proposition, i.e., to prove that the random 
variables $N_{i,j}^*$, $1\leq j\leq i-1$, are independent. 
In order to simplify notations, we only show that $N_{i,1}^*$ and $N_{i,2}^*$ are 
independent, but the argument can easily be extended to prove the full independence 
property. 
Since the $T_i^k$'s are stopping times, by the strong Markov property, observe that
$\sigma\big(X_1, \dots , X_{T_1^{k_1}}\big)
\underset{{T_1^{k_1}}}{\indep}
\sigma\big(X_{T_1^{k_1}+1}, \dots , X_{T_2^{k_2+N_2^*}}\big)$
where, again $\sigma(X_1, \dots , X_n)$ denotes the $\sigma$-field generated 
by the random variables $X_1, \dots, X_n$, while $\underset{{T_1^{k_1}}}{\indep}$ stands 
for independence conditionally on $T_1^{k_1}$.  
Moreover, $T_1^{k_1}$ and $\sigma\big(X_{T_1^{k_1}+1}, \dots , X_{T_2^{k_2+N_2^*}}\big)$ 
are independent, 
and thus so are 
$\sigma\big(X_1, \dots , X_{T_1^{k_1}}\big)$ and $\sigma\big(X_{T_1^{k_1}+1}, 
\dots , X_{T_2^{k_2+N_2^*}}\big)$. 
The independence of $N_{i,1}^*$ and $N_{i,2}^*$ becomes clear, 
since $N_{i,1}^*$ is $\sigma\big(X_1, \dots , X_{T_1^{k_1}}\big)$-measurable 
while $N_{i,2}^*$ is $\sigma\big(X_{T_1^{k_1}+1}, \dots , X_{T_2^{k_2+N_2^*}}\big)$-measurable.
The whole conclusion of the proposition then follows.
\end{proof}


\section{The Uniform Case}
\label{sec:Uniform}

In this section, we specialize ours results to the case where the letters are 
uniformly drawn from the alphabet, i.e., $p_i(X)=p_i(Y)= 1/m$, for all $1\leq i\leq m$.
Hence, the functional $\lcin$ in \eqref{eq24} rewrites as
\begin{eqnarray}
\nonumber
\lefteqn{\lcin=\max_{\Ccal_n}\min\left(
\frac nm+H_{m,n}(X)+K_{m,n}(X)+\frac 1m\sum_{i=1}^{m-1} S_{i,m}^{(n)}(X), \right . }
\\
&&
\label{eq:lcin}
\hspace{5cm} \left . 
\frac nm+H_{m,n}(Y)+K_{m,n}(Y) +\frac 1m\sum_{i=1}^{m-1} S_{i,m}^{(n)}(Y)\right),
\end{eqnarray}
and therefore 
\begin{eqnarray}\label{eq:lcin2}
\frac{\lcin- n/m}{\sqrt{2n}}
&=&\max_{\Ccal_n}\min\left( \frac{H_{m,n}(X)}{\sqrt{2n}}+\frac{K_{m,n}(X)}{\sqrt{2n}}+\frac 1{m\sqrt{2n}}
\sum_{i=1}^{m-1} S_{i,m}^{(n)}(X), \right.\nonumber \\
&&\qquad\qquad\qquad\qquad\left. \frac{H_{m,n}(Y)}{\sqrt{2n}}+\frac{K_{m,n}(Y)}{\sqrt{2n}}+\frac 1{m\sqrt{2n}}
\sum_{i=1}^{m-1} S_{i,m}^{(n)}(Y)\right).
\end{eqnarray}
The following simple inequality, a version of which is already present in \cite{HLM}, will be of multiple use (see Appendix~\ref{sec:Appendix_lemma} for a proof):
\begin{lem} 
\label{lemme:ineg}
Let $a_k, b_k, c_k, d_k$, $1\leq k\leq K$, be reals.  Then, 
\begin{equation}
\label{eq:ineg}
\left|\max_{k=1,\dots, K} \big(a_k\wedge b_k\big)-\max_{k=1,\dots, K} \big((a_k+c_k)
\wedge (b_k+d_k)\big)\right|
\leq \max_{k=1,\dots, K}\Big(|c_k|\vee|d_k|\Big).
\end{equation}
\end{lem}

The previous lemma entails 
\begin{eqnarray*}
\nonumber
\lefteqn{\left|
\max_{\Ccal_n}\min\left( \frac{H_{m,n}(X)}{\sqrt{2n}}
+\frac{K_{m,n}(X)}{\sqrt{2n}}
+\frac 1{m\sqrt{2n}}
\sum_{i=1}^{m-1} S_{i,m}^{(n)}(X), 
\right . \right .
}
\\
&&\hspace{4cm}\left . \frac{H_{m,n}(Y)}{\sqrt{2n}}
+\frac{K_{m,n}(Y)}{\sqrt{2n}}
+\frac 1{m\sqrt{2n}}\sum_{i=1}^{m-1} 
S_{i,m}^{(n)}(Y)\right)
 \\
&&\qquad \qquad \qquad \qquad \qquad \qquad \qquad \qquad \qquad \qquad -\left . 
\max_{\Ccal_n}\min\left( \frac{H_{m,n}(X)}{\sqrt{2n}}, 
\frac{H_{m,n}(Y)}{\sqrt{2n}}\right) \right | \\ 
&\le& \frac 1{m\sqrt{2n}}\left(\left|\sum_{i=1}^{m-1} S_{i,m}^{(n)}(X)+\frac{K_{m,n}(X)}{\sqrt{2n}}\right|\vee 
\left|\sum_{i=1}^{m-1} S_{i,m}^{(n)}(Y)+\frac{K_{m,n}(Y)}{\sqrt{2n}}\right|\right).
\end{eqnarray*}
But, from Proposition~\ref{prop:S}, as $n\to +\infty$, 
both $S_{i,m}^{(n)}(X)/\sqrt n\cvP 0$ and  
$S_{i,m}^{(n)}(Y)/\sqrt n \cvP 0$, for all $1\leq i\leq m-1$ (see \eqref{convS}). 
Let us now show that similarly  $K_{m,n}(X)/\sqrt n \cvP 0$ and $K_{m,n}(Y)/\sqrt n \cvP 0$ as $n\to+\infty$. 
Doing it for the $X$-sequence, and dropping the $X$-index in the notation, from \eqref{eq:Kmn}, it is enough to deal, for any $2\leq i\leq m-1$, with the number $N_m^{T_i^{N_i^*}, T_{i-1}^{N_{i-1}^*+k_{i-1}}}$ of $\alpha_m$ between the last ($k_{i-1}$-th) $\alpha_{i-1}$ selected and the previously ending $\alpha_i$. 
Counting in a backward manner the letters from the last $\alpha_{i-1}$ selected, the first $\alpha_i$ letter in this backward scheme (i.e., the first $\alpha_i$ before the last $\alpha_{i-1}$ selected) comes after a geometric number of letters. 
Since any letter but $\alpha_{i-1}$ is possible, this geometric random variable has parameter $p_i/(1-p_{i-1})=1/(m-1)$. 
In turn, we derive $\ee\big[N_m^{T_i^{N_i^*}, T_{i-1}^{N_{i-1}^*+k_{i-1}}}\big]\leq (1-p_{i-1})/p_i=m-1$, from which it follows that $N_m^{T_i^{N_i^*}, T_{i-1}^{N_{i-1}^*+k_{i-1}}}/\sqrt{2n}\cvP 0$ and $K_{m,n}/\sqrt{2n}\cvP 0$ as $n\to +\infty$.

\medskip\noindent
Therefore, the fluctuations of $\lcin$ expressed in \eqref{eq:lcin2} are the same as 
that of 
$$
\max_{\Ccal_n}\min\left( \frac{H_{m,n}(X)}{\sqrt{2n}}, 
\frac{H_{m,n}(Y)}{\sqrt{2n}}\right).
$$ 

\medskip\noindent
For uniform draws, the functional $H_{m,n}$ in \eqref{eq25} rewrites as   
\begin{eqnarray*}
\label{eq:HH2}
H_{m,n}&=&\sum_{i=1}^{m-1} \sqrt{2n} \frac 1m\sum_{j=1}^{N_i} 
\frac{N_m^{T_i^{j-1}, T_i^j}-1}{\sqrt{2n}}-\sum_{i=1}^{m-1} \sqrt{2n} 
\sum_{j=N_i^*+1}^{N_i^*+k_i} \frac{N_m^{T_i^{j-1}, T_i^j}-1}{\sqrt{2n}}\\
&=&
\sqrt{2n}\left( \frac 1m\sum_{i=1}^{m-1} B_n^{(i)}\left(\frac{N_i}n\right)
-\sum_{i=1}^{m-1} \left(B_n^{(i)}\left(\frac{N_i^*+k_i}n\right)
-B_n^{(i)}\left(\frac{N_i^*}{n}\right)\right)\right),
\end{eqnarray*}
where $B_n^{(i)}$ is the Brownian approximation defined from the random 
variables $N_m^{T_i^{j-1}, T_i^j-1}$, $j\geq 1$, which are iid, by Proposition~\ref{lemind}, 
centered 
and scaled to have variance one, i.e., $B_n^{(i)}$ is the polygonal process 
on $[0,1]$ defined by linear interpolation between the values
\begin{equation}
\label{interpol}
B_n^{(i)}\left(\frac kn\right)=\sum_{j=1}^k \frac{Z_j^{(i)}}{\sqrt n},
\end{equation}
where
\begin{equation}
\label{eq:Zj}
Z_j^{(i)}=\frac{N_m^{T_i^{j-1}, T_i^j}-1}{\sqrt 2}. 
\end{equation}
Next, we present some heuristic arguments which provide 
the limiting behavior of 
\begin{eqnarray}\label{eq:lcin22}
&&\max_{\Ccal_n}\min\left( 
\!\frac 1m\!\sum_{i=1}^{m-1}\!B_n^{(i),X}\!\left(\!\frac{N_i(X)}n\!\right)
-\sum_{i=1}^{m-1}\!\left(\!B_n^{(i),X}\!\left(\!\frac{N_i^*(X)+k_i}n\right)
-B_n^{(i),X}\left(\frac{N_i^*(X)}{n}\right)\!\right)\! \right ., \nonumber\\
&&\hskip 2cm \left .\frac 1m\!\sum_{i=1}^{m-1}\!B_n^{(i),Y}\!\left(\!\frac{N_i(Y)}n\!\right)
-\sum_{i=1}^{m-1}\!\!\left(\!B_n^{(i),Y}\!\left(\!\frac{N_i^*(Y)+k_i}n\right)
-B_n^{(i),Y}\!\left(\!\frac{N_i^*(Y)}{n}\right)\!\right)\!\!\right), 
\end{eqnarray}
knowing that, by Donsker theorem, 
$(B_n^{(1)}, \dots, B_n^{(m-1)})\stackrel{(C_0([0,1]))^{m-1}}{=\!\!\!=\!\!\!=\!\!\!=\!\!\!=\!\!\!=\!\!\!=\!\!\!=\!\!\!=\!\!\!=\!\!\!\Longrightarrow}
(B^{(1)}, \dots, B^{(m-1)})$, $n\to+\infty$, where $(B^{(1)}, \dots, B^{(m-1)})$ is a drift-less, $(m-1)$-dimensional, correlated Brownian motion on $[0,1]$, which is also zero at the origin. 
The correlation structure of this multivariate Brownian motion is given by that of the $Z_j^{(i)}$, $1\leq i\leq m-1$, which in turn is given by Proposition~\ref{lemind}. 
Above, $\stackrel{(C_0([0,1]))^{m-1}}{=\!\!\!=\!\!\!=\!\!\!=\!\!\!=\!\!\!=\!\!\!=\!\!\!=\!\!\!=\!\!\!\Longrightarrow}$ stands for the convergence in law in the product space of continuous function on $[0,1]$ vanishing at the origin.
Since the multivariate Donsker theorem is crucial is our argument, we give a precise statement:
\begin{thm}[Donsker]
\label{theo:Donsker}
Let $(Z_j)_{j\geq 1}$ be iid square integrable centered random vectors in $\bbR^{m-1}$, $m\geq 2$, with covariance matrix $\Sigma$. 
Let $(B_n)_{t\in [0,1]}$ be the polygonal process defined, for each $n\geq 1$, by 
$$
B_n(t)=\frac 1{\sqrt n} \sum_{k=1}^{[nt]} Z_k + \frac{(nt-[nt])}{\sqrt n} Z_{[nt]+1}, \quad t\in [0,1].
$$
Then $B_n\stackrel{(C_0([0,1]))^{m-1}}{=\!\!\!=\!\!\!=\!\!\!=\!\!\!=\!\!\!=\!\!\!=\!\!\!=\!\!\!=\!\!\!=\!\!\!\Longrightarrow} B$ 
where $B$ is a Brownian motion on $[0,1]^{m-1}$ with covariance matrix $t\Sigma$ and where $\stackrel{(C_0([0,1]))^{m-1}}{=\!\!\!=\!\!\!=\!\!\!=\!\!\!=\!\!\!=\!\!\!=\!\!\!=\!\!\!=\!\!\!=\!\!\!=\!\!\!=\!\!\!\Longrightarrow}$ stands for the convergence in law in the product space of continuous function on $[0,1]$ vanishing at the origin. 
\end{thm}

\begin{proof}
The multivariate Donsker theorem easily derives from the classical univariate one for which we refer, for instance to \cite[Th.~8.2]{Billingsley} and from the multivariate CLT as follows. 
Recall that the convergence  $B_n\stackrel{(C_0([0,1]))^{m-1}}{=\!\!\!=\!\!\!=\!\!\!=\!\!\!=\!\!\!=\!\!\!=\!\!\!=\!\!\!=\!\!\!\Longrightarrow} B$ is equivalent to the convergence of finite-dimensional distributions of $B_n$ to that  of $B$ and to the tightness of $(B_n)_{n\geq 1}$ in $(C_0([0,1]))^{m-1}$.
First, the multivariate CLT gives the convergence of the finite-dimensional distributions of $(B_n^{(1)}(t), \dots, B_n^{(m-1)}(t))_{0\le t \le 1}$ with a covariance structure given by that of the $Z_1^{(i)}$, $1\leq i\leq m-1$. 
Second, the tightness of $\big(B_n^{(1)}(t), \dots, B_n^{(m-1)}(t)\big)_{0\le t \le 1}$ is obtained from that of its coordinates: since $B_n^{(i)}$ is tight for each $1\leq i\leq m-1$ by the univariate Donsker theorem, for all $\varepsilon>0$, there is a compact $K_i$ of $C_0([0,1])$, the usual space of continuous functions on $[0,1]$ vanishing at the origin, such that $\sup_{n\geq 1}\pp\big(B_n^{(i)}\not\in K_i\big)<\varepsilon$ and we have 
\begin{equation*}
\sup_{n\geq 1}\pp\Big(\big(B_n^{(1)}, \dots, B_n^{(m-1)}\big)
\not\not\in K_1\times\dots\times K_{m-1}\Big)
\leq\sup_{n\geq 1}\sum_{i=1}^{m-1}\pp\big(B_n^{(i)}\not\not\in K_i\big)\\
<(m-1)\varepsilon,
\end{equation*}
with $K_1\times\dots\times K_{m-1}$ compact of $(C_0([0,1]))^{m-1}$ so that $\big(B_n^{(1),X}, \dots, B_n^{(m-1),X}\big)$ is tight in $C_0([0,1])^{m-1}$. 
\end{proof}


\subsection*{Heuristics}
Roughly speaking, there are three limits to handle in \eqref{eq:lcin22}:
\begin{enumerate}
\item The limit of the constraints in the maximum over $\Ccal_n$;
\item The limit of the linear terms: $\sum_{i=1}^{m-1} B_n^{(i),X}\left(\frac{N_i(X)}n\right)$;
\item The limit of the increments: $\sum_{i=1}^{m-1} \left(B_n^{(i),X}
\left(\frac{N_i^*(X)+k_i}n\right)-B_n^{(i),X}\left(\frac{N_i^*(X)}{n}\right)\right)$;
\end{enumerate}
and, similarly, for $X$ replaced by $Y$.  
Below, the symbol $\rightsquigarrow$ indicates a heuristic replacement or a heuristic limit, as $n\to+\infty$.

\medskip\noindent
First Limit (to be treated last, in Section \ref{sec:constraints}):  
Since $\Ccal_{n,i}(k_1, \dots, k_{i-1})=\{{\kk} = (k_1, \dots, k_{m-1}):0\leq k_i\leq 
\min\big(N_i(X)- N_i^*(X), N_i(Y)- N_i^*(Y)\big)\}$, 
(and, again, with vacuous constraints in case either 
$N_i^*(X)>n$ or $N_i^*(Y)>n$) and from the concentration property of the 
$N_i^*$, we expect (with again $k_0=0$, and $t_0 = 0$, below):  
\begin{eqnarray*}
\Ccal_{n,i}(k_1, \dots, k_{i-1})&\rightsquigarrow& \left\{{\kk}=(k_1, \dots, k_{m-1}): 
0\leq k_i\leq \left(\bbe[N_i(X)]-\sum_{j=1}^{i-1} k_j\right)\wedge 
\left(\bbe[N_i(Y)]-\sum_{j=1}^{i-1} k_j\right)\right\}\\
&=&\left\{{\kk}=(k_1, \dots, k_{m-1}): \frac 1n\sum_{j=1}^{i-1} k_j\leq \frac 1n
\sum_{j=1}^{i} k_j\le \frac{\bbe[N_i]}{n}, i =1, \dots, m-1\right\}.
\end{eqnarray*}
Hence, for $\Ccal_n$ defined in \eqref{eq8C}: 
$$
\Ccal_n
\rightsquigarrow {\cal V}\left(\frac 1m, \dots, \frac 1m\right),
$$
where 
${\cal V}(p_1, \dots, p_{m-1})=\big\{{\ttt}\!=\!(t_1, \dots, 
t_{m-1})\!: t_i\ge 0, i =1, \dots, m-1, t_1\le p_1, t_1+t_2\le p_2, 
\dots, t_1+\cdots +t_{m-1}\le p_{m-1}\big\}$.

\medskip\noindent
Second Limit (see Section~\ref{sec:linear}): For each $i=1, \dots, m-1$, 
the random variables $N_i$ are concentrated 
around their respective mean $\bbe[N_i]$ ($=1/m$), and so   
$$
\frac {N_i}n\rightsquigarrow \bbe[N_i]  \quad {\rm and} \quad 
\sum_{i=1}^{m-1} B_n^{(i)}\left(\frac{N_i}n\right)\rightsquigarrow 
\sum_{i=1}^{m-1} B^{(i)}\big(\bbe[N_i]\big)
=\sum_{i=1}^{m-1} B^{(i)}\Big(\frac 1m\Big),
$$
where the limit $B_n^{(i)}\stackrel{C_0([0,1])}{=\!\!=\!\!\Longrightarrow} B^{(i)}$ is 
taken simultaneously. 

\medskip\noindent
Third Limit (see Section~\ref{sec:increments}):  For each $i=1, \dots, m-1$, 
the random variables $N_i^*$ are also concentrated 
around their mean $\ee[N_i^*]=\sum_{j=1}^{i-1} k_j$, and so 
$N_i^*\rightsquigarrow \sum_{j=1}^{i-1} k_j$.  Therefore, 
\begin{eqnarray*}
B_n^{(i),X}\left(\frac{N_i^*(X)+k_i}n\right)-B_n^{(i),X}
\left(\frac{N_i^*(X)}{n}\right)&\rightsquigarrow&
B_n^{(i),X}\left(\sum_{j=1}^{i}\frac{k_j}n\right)
-B_n^{(i),X}\left(\sum_{j=1}^{i-1}\frac{k_j}n\right)\\
&\rightsquigarrow& B^{(i),X}\left(\sum_{j=1}^it_j\right)-B^{(i),X}
\left(\sum_{j=1}^{i-1}t_{j}\right), 
\end{eqnarray*}
and similarly for $X$ replaced by $Y$.  
Hence,    
\begin{eqnarray*}
\!\!\!\frac{\lcin- n/m}{\sqrt{2n}} 
\rightsquigarrow \max_{{\cal V}(1/m, \dots, 1/m)}\!\min\! 
\left(\!\frac 1m\!\sum_{i=1}^{m-1}\!B^{(i),X}\!\!\left(\frac 1m\right)-
\sum_{i=1}^{m-1}\!\left(\!\!B^{(i),X}\!\!\left(\sum_{j=1}^{i}t_{j}\!\right)
\!-B^{(i),X}\!\!\left(\sum_{j=1}^{i-1}t_{j}\!\right)\!\right)\!\!\right., \\
\quad\quad\quad \quad\quad\quad \left.\frac 1m\!\sum_{i=1}^{m-1}\!B^{(i),Y}\!\left(\!\frac 1m\right) -
\sum_{i=1}^{m-1}\!\left(\!B^{(i),Y}\!\left(\sum_{j=1}^{i}t_{j}\right)-B^{(i),Y}\!
\left(\!\sum_{j=1}^{i-1}t_{j}\right)\!\right)\!\right)\\
\stackrel{{\cal L}}{=} \frac{1}{\sqrt m}\max_{0=u_0\le u_1 \le \cdots\le u_{m-1}\le 1}
\min\!\left(\!\frac 1{m} \sum_{i=1}^{m-1} B^{(i),X}(1)
-\sum_{i=1}^{m-1}\left(\!B^{(i),X}(u_i)
-B^{(i),X}\!(u_{i-1})\right)\!\!\right.,\\ 
\quad\quad\quad \quad\quad\quad \left.\frac{1}{m}\sum_{i=1}^{m-1} B^{(i),Y}(1)-
\sum_{i=1}^{m-1} \left(\!B^{(i),Y}\!(u_i)
-B^{(i),Y}\!(u_{i-1})\right)\!\right),   
\end{eqnarray*}
by Brownian scaling and the reparametrization 
$\sum_{j=1}^{i}t_j = u_i/m$, $i=1, \dots m-1$, $u_0=t_0=0$.
In other words, 
\begin{eqnarray*}
\frac{\lcin- n/m}{\sqrt{2n/m}} 
\rightsquigarrow \max_{0=u_0\le u_1 \le \cdots\le u_{m-1}\le 1}\min \left( \frac 1m \sum_{i=1}^{m-1} 
B^{(i),X}(1)-\sum_{i=1}^{m-1} 
\left(\!B^{(i),X}\!(u_i)-B^{(i),X}\!(u_{i-1})\right),\right.\\
\quad\quad\quad \quad\quad\qquad  \left.\frac 1m \sum_{i=1}^{m-1} B^{(i),Y}(1)-
\sum_{i=1}^{m-1} \left(B^{(i),Y}(u_i)
-B^{(i),Y}(u_{i-1})\right)\!\right).  
\end{eqnarray*}
Finally, a linear transformation and Brownian properties allow 
to transform the parameter space 
into the Weyl chamber 
$$
{\cal W}_m(1):=\big\{{\bf t}=(t_0, t_1, \dots, 
t_{m-1}, t_m) : 0=t_0 \le t_1 \le \dots \le t_{m-1}\le t_m =1\big\},
$$ 
and to replace the $(m-1)$-dimensional correlated Brownian motion $B^X$ (resp.~$B^Y$), by  
an $m$-dimensional standard one $B_1$ (resp.~$B_2$).  Combining these facts, 
the expression on the right-hand side above, 
becomes equal, in law, to:

\begin{eqnarray}\label{eqlast}
\!\!\!\!\!\!\!&&\max_{{\bf t}\in{\cal W}_m(1)}
\!\min\!\left(-\frac{1}{m}\!\sum_{i=1}^m\!B^{(i)}_1\!(1)
+\sum_{i=1}^m\left(B^{(i)}_1\!(t_i)-B^{(i)}_1\!(u_{i-1})\right), \right.\nonumber\\ 
&&\qquad \qquad \qquad \left.\hskip 2cm -\frac{1}{m}\!\sum_{i=1}^m\!B^{(i)}_2\!(1)
+\sum_{i=1}^m\left(B^{(i)}_2\!(t_i)-B^{(i)}_2\!(u_{i-1})\right)\!\right), \nonumber   
\end{eqnarray}
which is the final form of our result, Theorem \ref{thm1.1}. 
In the sequel, we make precise the previous heuristic arguments.     


\medskip 
All along, we use different sets constraints. 
For easy references, we gather here the references to these notations:
$\widetilde\Ccal_n$  is defined in \eqref{eq8Ctilde},
$\widetilde\Ccal_{n,i}(k_1,\dots, k_{i-1})$ in \eqref{eq8Citilde},
$\Ccal_n$ in \eqref{eq8C}, 
$\Ccal_{n,i}(k_1, \dots, k_{i-1})$ in \eqref{eq8Ci}, 
$\Ccal_{n,i}^*$ in \eqref{eq:Cni*}, 
$\Ccal_n^*$ in \eqref{eq:Cn*}, 
$\Ccal_{n,i}$ above \eqref{eq:Cni**}, 
$\Ccal_{n,i}^\#$ in \eqref{eq:Cni**}, 
$\Ccal_n^{\pm}$ in \eqref{eq:Cpm}. 


\subsection{The Linear Terms}
\label{sec:linear}

Set 
$$
R_1(X)=\sum_{i=1}^{m-1} \left(B_n^{(i),X}\left(\frac{N_i^*+k_i}n\right)-
B_n^{(i),X}\left(\frac{N_i^*}n\right)\right),
$$
where again the dependency of $R(X)$ in $(k_1, \dots, k_{m-1})$ is omitted (see Remark~\ref{rem:depk}), so that with the help of \eqref{eq:lcin22}, \eqref{eq:lcin2} rewrites as:

\begin{eqnarray}
\label{deux}
&&\frac{\lcin-n/m}{\sqrt{2n}}=\max_{\Ccal_n}\min\left(\frac 1m\sum_{i=1}^{m-1} B_n^{(i),X}
\left(\frac{N_i(X)}n\right)-R_1(X),\right. \nonumber \\
&&\hskip 6cm\left. \frac 1m\sum_{i=1}^{m-1} 
B_n^{(i),Y}\left(\frac{N_i(Y)}n\right)-R_1(Y)
\right)+o_\pp(1), 
\end{eqnarray}
where, throughout, $o_\pp(1)$ indicates a term, which might be different 
from an expression to another, 
converging to zero, in probability, as $n$ converges to infinity.     

\medskip\noindent
Next, by Lemma~\ref{lemme:ineg}, 
\begin{eqnarray}
&&\left|
\max_{\Ccal_n}\min\left(\frac 1m\sum_{i=1}^{m-1} B_n^{(i),X}
\Big(\frac{N_i(X)}n\Big)-R_1(X), 
\frac 1m\sum_{i=1}^{m-1} B_n^{(i),Y}\Big(\frac{N_i(Y)}n\Big)-R_1(Y) \right)\right. \nonumber 
\\
 \nonumber 
&& \quad \left .
-\max_{\Ccal_n}\min\left(\!\frac 1m\sum_{i=1}^{m-1}\!B_n^{(i),X}
\Big(\!\frac{\ee[N_i(X)]}n\Big)-R_1(X), \frac 1m\sum_{i=1}^{m-1}\!B_n^{(i),Y}
\Big(\!\frac{\ee[N_i(Y)]}n\Big)-R_1(Y)\!\right)\right|
\\
 \nonumber 
&\leq&\max_{\Ccal_n}\left|
\min\left(\frac 1m\sum_{i=1}^{m-1} B_n^{(i),X}\Big(\frac{N_i(X)}n\Big)-R_1(X), 
\frac 1m\sum_{i=1}^{m-1} B_n^{(i),Y}\Big(\frac{N_i(Y)}n\Big)-R_1(Y)\right)\right .
\\
 \nonumber 
&&\quad \quad \quad \quad \left. -\min\left(\!\frac 1m\sum_{i=1}^{m-1}\! 
B_n^{(i),X}\Big(\frac{\ee[N_i(X)]}n\Big)-R_1(X), \frac 1m\sum_{i=1}^{m-1}\! 
B_n^{(i),Y}\Big(\frac{\ee[N_i(Y)]}n\Big)-R_1(Y)\!\right)\right| 
\\
 \nonumber 
&\leq& \max_{\Ccal_n}\left(\max\left(\frac 1m\left|\sum_{i=1}^{m-1}
\Big(B_n^{(i),X}\Big(\frac{N_i(X)}n\Big)
-B_n^{(i),X}\Big(\frac{\ee[N_i(X)]}n\Big)\Big)\right|, \right.\right.
\\ 
\label{premred}
&&\hskip 5cm\left.\left. \frac 1m\left|\sum_{i=1}^{m-1}
\Big(B_n^{(i),Y}\Big(\frac{N_i(Y)}n\Big)
-B_n^{(i),Y}\Big(\frac{\ee[N_i(Y)]}n\Big)\Big)\right|\right)\right).
\end{eqnarray}
We now wish to show that the right-hand side of \eqref{premred} converges 
to zero, in probability. 
First note that for each $2\le i\leq m-1$, 
$\Ccal_{n,i}(k_1 \dots, k_{i-1})\subset 
\big\{{\kk} = (k_1, \dots, k_{m-1}): 0\leq k_i\leq \min\left(N_i(X), N_i(Y)\right)\big\} 
\subset \big\{{\kk} = (k_1, \dots, k_{m-1}): 0\le k_i\le n\big\}$, and the same holds true for $\Ccal_{n,1}$, see \eqref{eq8C}.  
But, $B_n^{(i)}\left({N_i}/n\right)-B_n^{(i)}(\ee [N_i]/n)$, where we have dropped 
$X$ and $Y$, does not depend on $\kk$.
Therefore, the maximum can be skipped and the problem reduces 
to showing that, for all  $1\le i \le m-1$: 
\begin{equation}
\label{eq:max1}
\left|B_n^{(i)}\left(\frac{N_i}n\right)-B_n^{(i)}\left(\frac{\ee[N_i]}n\right)\right|\cvP 0,  
\end{equation}
as $n\to +\infty$.  
This follows from the forthcoming lemma applied, for each $i=1, \dots, m-1$, to 
the random variables $Z_j^{(i)}=\big(N_m^{T_i^{j-1}, T_i^j}-1\big)/\sqrt {2}$ , 
present in both \eqref{interpol} and \eqref{eq:Zj} and which, by Proposition~\ref{lemind}, 
are iid with mean zero and variance one. 
Note that the lemma below (see Appendix~\ref{sec:Appendix_lemma} for a proof) can indeed be brought into play since Hoeffding's inequality, 
applied to the random variables $N_i$, ensures that  for $x_n=\sqrt n \ln n$, 
\begin{equation}
\label{hoe}
\lim_{n\to +\infty}\bbp\big(|N_i-\ee[N_i]|\ge x_n\big)
\le \lim_{n\to +\infty} {2e^{-2x_n^2/n}}=0.
\end{equation}  
\begin{lem}
\label{lemme:devZ}
Let $(Z_j)_{j\geq 1}$ be iid centered random variables with unit variance,  
and for each $n\in \nit$, let $N^{(n)}$ be an $\nit$-valued random variable such 
that $\lim_{n\to+\infty}\bbp\big(|N^{(n)}- \ee[N^{(n)}]| \geq x_n\big)=0$, where $x_n\geq 0$ is such that $\lim_{n\to+\infty} x_n/n=0$.
Then,  
$$
\sum_{j\in [N^{(n)}, \ee[N^{(n)}] ]} \frac{Z_j}{\sqrt n}\stackrel{\bbp}{\longrightarrow} 0, 
$$
where $[N^{(n)}, \ee[N^{(n)}] ]$ is short for $[\min(N^{(n)},\ee[N^{(n)}]),  \max(N^{(n)},\ee[N^{(n)}])]$.  
\end{lem} 

\noindent
At this stage, \eqref{eq:max1} is proved and therefore,     
\begin{eqnarray}
\label{trois}
\frac{\lcin-n/m}{\sqrt{2n}}=\max_{\Ccal_n}\min\left(\frac 1m\sum_{i=1}^{m-1} B_n^{(i),X}
\left(\frac{\ee[N_i(X)]}n\right)-R_1(X),\right. \nonumber \\
\left.  \frac 1m\sum_{i=1}^{m-1} 
B_n^{(i),Y}\left(\frac{\ee[N_i(Y)]}n\right)-R_1(Y)
\right)+o_\pp(1), 
\end{eqnarray}
finishing the first part of the proof of Theorem~\ref{thm1.1}.  
Indeed, $(N_1, \dots, N_m)$ is multinomial with parameters $n$ and $(p_1, \dots, p_m)$.  
So, for uniform draws, $\ee[N_i(X)]=\ee[N_i(Y)]=np_i=n/m$. 
Then, by the multivariate Donsker theorem, see Th.~\ref{theo:Donsker}, and scaling, 
\begin{equation}
\label{eq:CVlaw_mB}
\sum_{i=1}^{m-1} \frac 1mB_n^{(i),X}\left(\frac{\ee[N_i(X)]}n\right)
\Longrightarrow \sum_{i=1}^{m-1} \frac 1{m\sqrt m}B^{(i),X}(1), 
\quad n\to +\infty,
\end{equation}
where $\big(B^{(1),X}(t), \dots, B^{(m-1),X}(t)\big)_{0\le t \le 1}$ is a $(m-1)$-dimensional Brownian motion and similarly for $Y$. 
As shown next, the covariance matrix of this Brownian motion at time $t$ is $t\Sigma = t(\sigma_{k,l})_{1\le k,l\le m-1}$, where 
\begin{equation}
\label{eq:Cov_mB}
\Sigma=\left(\begin{array}{cccc}
1 &1/2&\dots&1/2\\
1/2&1&1/2&\vdots\\
\vdots&&\ddots&1/2\\
1/2&\dots&1/2&1
\end{array}
\right).
\end{equation}  
Indeed, $\Sigma$ in \eqref{eq:Cov_mB} is obtained as follows: 
First, since 
$$
\big(B_n^{(1),X}, \dots, B_n^{(m-1),X}\big)
\stackrel{(C_0([0,1]))^{m-1}}{=\!\!\!=\!\!\!=\!\!\!=\!\!\!=\!\!\!=\!\!\!=\!\!\!=\!\!\!=\!\!\!=\!\!\!\Longrightarrow}
\big(B^{(1),X}, \dots, B^{(m-1),X}\big),
$$
while uniform integrability (see Lemma~\ref{lemme:Bn_ui} below) entails
$$
\lim_{n\to+\infty}\cov\big(B_n^{(k),X}(1),B_n^{(l),X}(1)\big)
=\cov\big(B^{(k),X}(1),B^{(l),X}(1)\big)=\sigma_{k,l}.
$$
Next, in the uniform case, Proposition~\ref{prop:S} writes, for $i=1,\dots, m-1$, as
$$
N_m=N_i+\sqrt{2n} B_n^{(i),X}\Big(\frac{N_i}n\Big)+o_\pp(\sqrt n),
$$
so that using also Remark~\ref{rem:Sim}
\begin{equation}
\label{eq:Cov_tech1}
\cov\Big(B_n^{(k),X}\Big(\frac{N_k}n\Big),B_n^{(l),X}\Big(\frac{N_l}n\Big)\Big)=\frac 1{2n}\cov\big(N_m-N_k,N_m-N_l\big)+o(1).
\end{equation}
But $(N_1, \dots, N_m)\sim Mult\big(n, (\frac 1m, \dots, \frac 1m)\big)$, $N_i/n\to 1/m$, $\Var(N_i)=n(m-1)/m^2$ and when, $i\not =j$, $\cov(N_i, N_j)=-n/m^2$. 
Therefore, 
\begin{equation}
\label{eq:Cov_tech2}
\frac 1{2n}\cov\big(N_m-N_k,N_m-N_l)=\frac 1{2n}\Big(\frac{n(m-1)}{m^2}+\frac n{m^2}-\frac n{m^2}-\frac n{m^2}\Big)=\frac1{2m}.
\end{equation}
Since by Lemma~\ref{lemme:Bn_ui}, $\lim_{n\to+\infty}\ee\big[\big(B_n^{(k)}\big(N_k/n\big)-B_n^{(k)}\big(1/m\big)\big)^2\big]=0$, it follows 
\begin{eqnarray}
\nonumber
\lim_{n\to+\infty}\cov\Big(B_n^{(k),X}\Big(\frac{N_k}n\Big),B_n^{(l),X}\Big(\frac{N_l}n\Big)\Big)
&=&\lim_{n\to+\infty}\cov\Big(B_n^{(k),X}\Big(\frac1m\Big),B_n^{(l),X}\Big(\frac1m\Big)\Big)\\
\nonumber
&=&
\cov\Big(B^{(k),X}\Big(\frac1m\Big),B^{(l),X}\Big(\frac1m\Big)\Big)\\
\label{eq:Cov_tech3}
&=&\frac 1m\cov\Big(B^{(k),X}(1),B^{(l),X}(1)\Big).
\end{eqnarray}
Finally, \eqref{eq:Cov_tech1}, \eqref{eq:Cov_tech2}, \eqref{eq:Cov_tech3} ensure the expression \eqref{eq:Cov_mB} for the covariance.
To finish, let us state a lemma, just used above and, whose proof is presented in Appendix~\ref{sec:Appendix_lemma}.

\medskip
\begin{lem}
\label{lemme:Bn_ui}
The sequences $\big(B_n^{(k)}(N_k/n)^2\big)_{n\geq 1}$ and $\big(B_n^{(k)}(1/m)^2\big)_{n\geq 1}$, 
$k=1, \dots, m-1$, are uniformly integrable and  
\begin{equation}
\label{eq:Bn_approxL2}
\lim_{n\to+\infty} \ee\Big[\Big(B_n^{(k)}\Big(\frac{N_k}n\Big)-B_n^{(k)}\Big(\frac1m\Big) \Big)^2\Big]=0.
\end{equation}
\end{lem}


\subsection{The Increments}
\label{sec:increments}
In this section, we compare the maximum of two different quantities over 
the same set of constraints 
in order to simplify the quantities to be maximized (before simplifying the constraints $\Ccal_n$ 
themselves, in the next section).  
The quantities to compare are:  
\begin{eqnarray}
\label{eq:max11}
\max_{\kk\in \Ccal_n}\left\{\left(\frac 1m\sum_{i=1}^{m-1} B_n^{(i),X}(p_i(X))
-\sum_{i=1}^{m-1} \left(B_n^{(i),X}\left(\frac{N_i^*(X)+k_i}{n}\right)-B_n^{(i),X}
\left(\frac{N_i^*(X)}n\right)\right)\right) \right. \bwedge \nonumber \\ 
\left. \left(\frac 1m\sum_{i=1}^{m-1} B_n^{(i),Y}(p_i(Y))
-\sum_{i=1}^{m-1} \left(B_n^{(i),Y}\left(\frac{N_i^*(Y)+k_i}{n}\right)-B_n^{(i),Y}
\left(\frac{N_i^*(Y)}n\right)\right)\right)\right\},
\end{eqnarray}
and 
\begin{eqnarray}
\label{eq:max12}
\max_{\kk\in \Ccal_n}\left\{\left(\frac 1m\sum_{i=1}^{m-1} B_n^{(i),X}(p_i(X))-\sum_{i=1}^{m-1} 
\left(B_n^{(i),X}\left(\frac{\sum_{j=1}^{i}k_j}{n}\right)-B_n^{(i),X}
\left(\frac{\sum_{j=1}^{i-1}k_j}n\right)\right)\right)
\right.  \bigwedge \nonumber 
\\
\left.  \left(\frac 1m\sum_{i=1}^{m-1} B_n^{(i),Y}(p_i(Y))-\sum_{i=1}^{m-1} 
\left(B_n^{(i),Y}\left(\frac{\sum_{j=1}^{i}k_j}{n}\right)-B_n^{(i),Y}
\left(\frac{\sum_{j=1}^{i-1}k_j}n\right)\right)\right)\right\}.
\end{eqnarray}
Using \eqref{eq:ineg} in Lemma~\ref{lemme:ineg}, their absolute difference is upper-bounded by 
\begin{eqnarray*}
&&\max_{\kk\in \Ccal_n}\left\{
\left|\sum_{i=1}^{m-1}\left( B_n^{(i),X}\left(\frac{N_i^*(X)+k_i}{n}\right)
-B_n^{(i),X}\left(\frac{N_i^*(X)}n\right)\right)\right.\right.
\\
&&
\qquad \qquad \qquad \qquad \qquad \left.-\sum_{i=1}^{m-1} \left(B_n^{(i),X}
\left(\frac{\sum_{j=1}^{i}k_j}{n}\right)-B_n^{(i),X}
\left(\frac{\sum_{j=1}^{i-1}k_j}n\right)\right)\right| 
\\
&&\qquad \qquad \quad \bigvee  \!\left|\sum_{i=1}^{m-1}
\left( B_n^{(i),Y}\left(\frac{N_i^*(Y)+k_i}{n}\right)
-B_n^{(i),Y}\left(\frac{N_i^*(Y)}n\right)\right)\right.
\\
&& \qquad \qquad \qquad \qquad \qquad \left.-\sum_{i=1}^{m-1} 
\left(B_n^{(i),Y}\left(\frac{\sum_{j=1}^{i}k_j}{n}\right)
-B_n^{(i),Y}\left(\frac{\sum_{j=1}^{i-1}k_j}n\right)\right)\right|\Bigg\}
\\
&&\le
\max_{\kk\in \Ccal_n} \Bigg\{
\left|\sum_{i=1}^{m-1}\left(B_n^{(i),X}\left(\frac{(N_i^*(X)+k_i)}{n}\right)-B_n^{(i),X}
\left(\frac{\sum_{j=1}^{i}k_j}n\right)\right)\right|
\\
&&\hskip 4.5cm\bigvee \left|\sum_{i=1}^{m-1}\left(B_n^{(i),Y}
\left(\frac{(N_i^*(Y)+k_i)}{n}\right)-B_n^{(i),Y}
\left(\frac{\sum_{j=1}^{i}k_j}n\right)\right)\right|\Bigg\}
\\
&&+
\max_{\kk\in \Ccal_n} \Bigg\{\left|\sum_{i=1}^{m-1}
\left(B_n^{(i),X}\left(\frac{N_i^*(X)}n\right)-B_n^{(i),X}
\left(\frac{\sum_{j=1}^{i-1}k_j}n\right)\right)\right|
\\
&&\hskip 4.5cm\bigvee \left|\sum_{i=1}^{m-1}\left(B_n^{(i),Y}\left(\frac{N_i^*(Y)}n\right)
-B_n^{(i),Y}\left(\frac{\sum_{j=1}^{i-1}k_j}n\right)\right)\right|\Bigg\}. 
\end{eqnarray*}
Recall that $N_1^*(X)=N_1^*(Y)=0$.  Hence, for $i=1$, 
$$
B_n^{(i),X}\left(\frac{N_i^*(X)+k_i}{n}\right)-B_n^{(i),X}
\left(\frac{\sum_{j=1}^{i}k_j}{n}\right)
=B_n^{(i),X}\left(\frac{N_i^*(X)}{n}\right)-B_n^{(i),X}
\left(\frac{\sum_{j=1}^{i-1}k_j}{n}\right)
=0,  
$$ 
with the same property for functionals relative to $Y$.  
Therefore, we are left with investigating terms of the form
\begin{eqnarray}\label{eq:form1}
&&\max_{\kk\in \Ccal_n} \left\{\left|B_n^{(i),X}
\left(\frac{N_i^*(X)+k_i}{n}\right)
-B_n^{(i),X}\left(\frac{\sum_{j=1}^{i}k_j}{n}\right)\right| \right.\nonumber \\
&& \hskip 3.5cm \left. \bigvee \left|B_n^{(i),Y}\left(\frac{N_i^*(Y)+k_i}n\right)-
B_n^{(i),Y}\left(\frac{\sum_{j=1}^{i}k_j}{n}\right)\right|\right\},
\end{eqnarray}
and 
\begin{equation}
\label{eq:form2}
\max_{\kk\in \Ccal_n}
\left\{\left|B_n^{(i),X}\!\left(\frac{N_i^*(X)}n\right)
-B_n^{(i),X}\!\left(\frac{\sum_{j=1}^{i-1}k_j}n\right)\right|
\bigvee \left|B_n^{(i),Y}\!\left(\frac{N_i^*(Y)}n\right)
-B_n^{(i),Y}\!\left(\frac{\sum_{j=1}^{i-1}k_j}n\right)\right|\right\},
\end{equation}
for $2\leq i\leq m-1$. Above, all the quantities considered only depend 
on a single sequence, say $X$  or $Y$, 
except for the constraints in $\Ccal_n$ which depend on both $X$ and $Y$.  
However, 
\begin{equation}
\label{eq:Cni*}
\Ccal_{n,i}(k_1,\dots, k_{i-1})\subset {\Ccal_{n,i}^*}(X):=\big\{{\kk} = (k_1, \dots, k_{m-1}):0\leq k_i
\leq N_i(X)- N_i^*(X)\big\},
\end{equation} 
(resp.~$\Ccal_{n,i}(k_1,\dots, k_{i-1}) \subset {\Ccal_{n,i}^*}(Y)$) and the same for $\Ccal_{n,1}$, and so upper-bounding, in \eqref{eq:form1} and 
\eqref{eq:form2}, the inner maxima by sums and the maxima over $\Ccal$ by maxima over 
\begin{equation}
\label{eq:Cn*}
{\Ccal_n^*}(X):= \bigcap_{i=1}^{m-1}{\Ccal_{n,i}^*}(X),
\end{equation}
(resp.~${\Ccal_n^*}(Y)$),
we are left with investigating, for $2\leq i\leq m-1$, the convergence in probability of terms of the form  
\begin{equation}\label{eq:form12}
\max_{\kk\in{\Ccal_n^*}(X)} 
\left\{\left|B_n^{(i),X}\left(\frac{N_i^*(X)+k_i}{n}\right)
-B_n^{(i),X}\left(\frac{\sum_{j=1}^{i}k_j}{n}\right)\right| \right\},
\end{equation}
and 
\begin{equation}
\label{eq:form22}
\max_{\kk\in{\Ccal_n^*}(X)} 
\left\{\left|B_n^{(i),X}\left(\frac{N_i^*(X)}n\right)
-B_n^{(i),X}\left(\frac{\sum_{j=1}^{i-1}k_j}n\right)\right|\right\},
\end{equation}
and, similarly with $X$ replaced by $Y$.  
Omitting the reference to either $X$ or $Y$, 
the terms to control are, from \eqref{interpol} and for 
each, $2\le i \le m-1$, of the form:   
\begin{equation}
\label{eq:max33}
\max_{\kk\in\Ccal_n^*}\left|\sum_{j=k_1+\cdots+k_i+1}^{N_i^*+k_i} 
\frac{Z_j^{(i)}}{\sqrt n}\right|,
\end{equation}
and 
\begin{equation}
\label{eq:max34}
\max_{\kk\in\Ccal_n^*}\left|\sum_{j=k_1+\cdots+k_{i-1}+1}^{N_i^*} 
\frac{Z_j^{(i)}}{\sqrt n}\right|, 
\end{equation}
where the $Z_j^{(i)}$, $j\ge 1$, are defined in \eqref{eq:Zj} 
and where 
$$
\Ccal_n^* =\bigcap_{i=1}^{m-1}\Ccal_{n,i}^*, \quad \mbox{ with } 
\Ccal_{n,i}^*=\big\{{\kk} = (k_1, \dots, k_{m-1}):0\leq k_i
\leq N_i- N_i^*\big\}. 
$$ 
In \eqref{eq:max33}, \eqref{eq:max34} and henceforth, we write $\sum_{j=n_1}^{n_2}$ regardless of the order of $n_1$ and $n_2$, i.e., by convention this sum is $\sum_{j=n_2}^{n_1}$ when $n_2<n_1$. 

Since \eqref{eq:max34} is similar, but easier to tackle 
than \eqref{eq:max33}, we only deal with  \eqref{eq:max33}.  
Again, as in Section~\ref{sec:linear}, 
let $D_n^i=\big\{\left|N_i-\bbe[N_i]\right| \le \sqrt n\ln n\big\}$ 
for $i=1,2, \dots, m-1$, and, thus, for $\varepsilon>0$,  
\begin{eqnarray} 
\label{eq:maj61a}
\lefteqn{\pp\left(\max_{\kk\in\Ccal_n^*} 
\left|\sum_{j=k_1+\cdots+k_i+1}^{N_i^*+k_i}  
\frac{Z_j^{(i)}}{\sqrt n}\right|\ge\varepsilon\right)}\\
\nonumber
&\leq&\pp\left(\left\{\max_{\kk\in\Ccal_n^*} 
\left|\sum_{j=k_1+\cdots+k_i+1}^{N_i^*+k_i}  
Z_j^{(i)}\right|\ge\varepsilon\sqrt n\right\}\cap
\bigcap_{i=1}^{m-1}D_n^i\right)+\sum_{i=1}^{m-1}\pp\big((D_n^i)^c\big).
\end{eqnarray}
Let $\Ccal_n^{i-1}=\bigcap_{j=1}^{i-1}\big\{k_j \le \bbe[N_j] +\sqrt n\ln n\big\}$
and let $\Ccal^{\#}_{n,i}$ be the set of indices $k_1,\dots, k_i$ 
\begin{equation}
\label{eq:Cni**}
\Ccal^{\#}_{n,i}
=\Ccal_n^{i-1}\cap \Big\{\bbe[N_i^*]\le \ell_i=k_1+\cdots +k_i
\le \bbe[N_i] +\sqrt n\ln n -(N_i^* - \bbe[N_i^*]) \Big\} 
\end{equation}
where we set $\ell_i:= k_1+\cdots +k_i$. 
Since under $\bigcap_{i=1}^{m-1}D_n^i$, $\Ccal_{n,i}^*\subset\big\{k_j\le \bbe[N_j] +\sqrt n\ln n\big\}$ 
and since Proposition~\ref{prop:law_Ni}, specialized to the uniform case, gives $\ee[N_i^*]=\sum_{j=1}^{i-1} k_j$, it follows that 
$\Ccal_n^*\subset \Ccal^{\#}_{n,i}$ and \eqref{eq:maj61a} is thus further upper-bounded by 
\begin{equation} 
\label{eq:maj61}
\pp\left(\max_{\kk\in\Ccal^{\#}_{n,i}} 
\left|\sum_{j=\ell_i + 1}^{\ell_i+N_i^*-(k_1+\cdots+k_{i-1})} Z_j^{(i)}\right|
\ge\varepsilon\sqrt n\right) +\sum_{i=1}^{m-1}\pp\big((D_n^i)^c\big).
\end{equation}
Now, in view of \eqref{hoe}, it is enough to show the convergence to zero  
of the first term on the right-hand side of \eqref{eq:maj61}.
To do so, set $E_n^1=\Omega$ and, for $2\leq i\leq m-1$,
$$
E_n^i(k_1, \dots, k_{i-1})
=\Big\{|N_i^*(k_1, \dots, k_{i-1})-\ee[N_i^*(k_1, \dots, k_{i-1})]|\leq x_n\Big\},
$$
with 
\begin{equation}
\label{eq:xn}
x_n=\sqrt n \ln n, 
\end{equation}
and let 
\begin{equation}
\label{eq:Eni}
 E_n^i=\bigcap_{(k_1, \dots, k_{i-1})\in \Ccal_n^{i-1}} E_n^i(k_1, \dots, k_{i-1}).
\end{equation}
Our next goal is to show that asymptotically, $E_n^i$ has full probability.   
\begin{prop}
\label{prop:maxC1}
Let $2\le i \le m-1$, then $\lim_{n\to+\infty} \pp\big((E_n^i)^c\big)=0$.
\end{prop}
In order to prove Proposition~\ref{prop:maxC1}, we first need the following technical result, proved in Appendix~\ref{sec:Appendix_lemma}:   
\begin{lem}
\label{lemme:ThetaF}
For $x\in[-n,+\infty)$, let  
$$
K_n(x)=\frac{(x+2n)^{x+2n}}{(2x+2n)^{x+n}(2n)^n}.
$$  
Then, for some constants $c, C\in (0,+\infty)$,  
\begin{equation}
\label{eq:ThetaF}
K_n(x)\le C\exp\left(-cn\min\Big(\frac {|x|}n, \frac{x^2}{n^2}\Big)\right).
\end{equation}
\end{lem}
We proceed now to the proof of Proposition~\ref{prop:maxC1}:
\begin{proof}(Prop.~\ref{prop:maxC1})
Clearly, 
\begin{eqnarray*}
\pp\big((E_n^i)^c\big)&\le& \sum_{(k_1, \dots, k_{i-1})\in \Ccal_n^{i-1}}
\pp\big((E_n^i(k_1, \dots, k_{i-1}))^c\big)\\ 
&\leq& n^{i-1} 
\max_{(k_1, \dots, k_{i-1})\in \Ccal_n^{i-1}}\pp\big((E_n^i(k_1, \dots, k_{i-1}))^c\big). 
\end{eqnarray*}
Therefore, to prove the lemma, it is enough to show that:  
\begin{equation}
\label{eq:maxC1}
\lim_{n\to+\infty} n^{i-1} 
\max_{(k_1, \dots, k_{i-1})\in \Ccal_n^{i-1}}\pp\big((E_n^i(k_1, \dots, k_{i-1}))^c\big) = 0.
\end{equation} 
Now, for each $2\leq i\leq m-1$, Propositions \ref{prop:law_Nij} and \ref{prop:law_Ni} assert that,  
\begin{equation*}
\label{eq:N2N2j}
N_i^*=N_i^*(k_1, \dots, k_{i-1})=\sum_{j=1}^{i-1} N_{i,j}^*,  
\end{equation*}
where the $(N_{i,j}^*)_{1\le j\le i-1}$ are independent and with probability generating function  
\begin{equation*}
\label{eq:Ni*}
\ee\left[x^{N_{i,j}^*}\right]=\left(\frac 1{2-x}\right)^{k_j}. 
\end{equation*}
Next, 
\begin{eqnarray}
\label{eq:ab0}
\hspace{-.8cm} \pp\left((E_n^i(k_1, \dots, k_{i-1}))^c\right)
&=&\pp\big(\left|N_i^*-\ee[N_i^*]\right| > x_n\big) \nonumber\\
&=&\pp\left(\sum_{j=1}^{i-1}\!\left(N_{i,j}^*-k_j\right) > x_n\!\right)+
\pp\left(\sum_{j=1}^{i-1}\!\left(k_j- N_{i,j}^*\right) > x_n\!\right).
\end{eqnarray}
The first term in \eqref{eq:ab0} is bounded by 
$\Theta_{k_1+\cdots+k_{i-1}}^r(x_n)$, where
\begin{eqnarray}
\label{eq:thetar0} 
\Theta_k^r(x)&:=&
\min_{t>0}\Big(\exp\left(-\left(t(x+k) +k\ln(2-e^t)\right)\right)\Big)\\
\label{eq:thetar}
&=&\frac{(x+2k)^{x+2k}}{(2x+2k)^{x+k}(2k)^k}, 
\end{eqnarray}
since the minimization in \eqref{eq:thetar0} occurs at $t=\ln\big((2x+2k)/(x+2k)\big)$. 

\medskip\noindent
The second term in \eqref{eq:ab0} is bounded by $\Theta_{k_1+\cdots+k_{i-1}}^l(x_n)$, where
\begin{eqnarray}
\label{eq:thetal}
\Theta_{k}^l(x)&:=&
\min_{t>0}\Big(\exp\left(-\left(t(x-k)+k\ln(2-e^{-t})\right)\right)\Big)\\
\label{eq:thetalbis}
&=& \frac{(2k-x)^{2k-x}}{(2k-2x)^{k-x}(2k)^k},
\end{eqnarray}
observing that, for $x\leq k$, the minimization in \eqref{eq:thetal} occurs 
at $t=\ln\big((2k-x)/(2k-2x)\big)$.  

\medskip\noindent
From the previous bounds and \eqref{eq:ab0}, it is clear that 
\eqref{eq:maxC1} will follow from 
\begin{equation}
\label{eq:thetabullet}
\lim_{n\to+\infty} n^{i-1}\max_{(k_1, \dots, k_{i-1})\in {\Ccal}_n^{i-1}} 
\Theta_{k_1+\cdots+k_{i-1}}^\bullet(x_n)=0, 
\end{equation}
for $\bullet\in\{l,r\}$.  To obtain such a limit, we make use 
of Lemma~\ref{lemme:ThetaF}, with $x=x_n=\sqrt n \ln (n)$, 
noting also that 
${\Ccal}_n^{i-1} \subset \big\{(k_1, \dots, k_{i-1}): k_1+\cdots+k_{i-1}
\le \sum_{j=1}^{i-1}\bbe[N_j] + (i-1)\sqrt n \ln n\big\} \subset 
\big\{(k_1, \dots, k_{i-1}): k_1+\cdots+k_{i-1}
\le (i-1)(\max_{j=1, \dots, i-1}p_j n+\sqrt n \ln n)\big\}$.  
 
First, for $\bullet=r$, when $k_1+\cdots+k_{i-1}\le x_n$, \eqref{eq:ThetaF} writes as
\begin{eqnarray*}
\Theta_{k_1+\cdots+k_{i-1}}^r(x_n)&\leq&
C\exp\!\Big(\!\!-c(k_1+\cdots+k_{i-1})\min\!\Big(\frac{x_n}{k_1+\cdots+k_{i-1}},
\Big(\frac{x_n}{k_1+\cdots+k_{i-1}}\Big)^2\Big)\!\Big) \\
&=&C\exp(-cx_n), 
\end{eqnarray*}
so that 
\begin{eqnarray*}
n^{i-1}\max_{k_1+\cdots+k_{i-1}\leq x_n} \Theta_{k_1+\cdots+k_{i-1}}^r(x_n)
\leq Cn^{i-1} e^{-c \sqrt n \ln n} \to 0, \quad n\to+\infty,
\end{eqnarray*}
where above, and below, $C$ is a finite positive constant whose value might change from 
a line to another. 
For $x_n\leq k_1+\cdots+k_{i-1}\leq (i-1)(n\max_{j=1, \dots, i-1}p_j+\sqrt n\ln n)
= (i-1)(n/m+\sqrt n\ln n)$, 
\eqref{eq:ThetaF} writes as 
\begin{eqnarray*}
\!\!\Theta_{k_1+\cdots+k_{i-1}}^r(x_n)&\leq&C\exp\!\Big(\!\!-c(k_1+\cdots+k_{i-1})
\!\min\!\Big(\frac{x_n}
{k_1+\cdots+k_{i-1}},\Big(\frac{x_n}{k_1+\cdots+k_{i-1}}\Big)^2\Big)\!\Big) \\
&=&C\exp\Big(-c\frac{x_n^2}{k_1+\cdots+k_{i-1}}\Big)\\
&\le& C\exp\left(-c\frac{x_n^2}{(i-1)(n/m+\sqrt n\ln n)}\right), 
\end{eqnarray*}
so that 
\begin{eqnarray*}
&&n^{i-1}\max_{x_n\leq k_1+\cdots+k_{i-1}\leq (i-1)(n/m+\sqrt n \ln n)} 
\Theta_{k_1+\cdots+k_{i-1}}^r(x_n)\\
&&\hskip 2.5cm 
\leq n^{i-1}\exp\left(-c\frac{n(\ln n)^2}{(i-1)(n/m+\sqrt n\ln n)}\right)\to 0, 
\quad n\to+\infty, 
\end{eqnarray*}
guaranteeing \eqref{eq:thetabullet} with $\bullet=r$.  

Next, let $\bullet=l$ and consider the following three cases: 
$k_1+\cdots+k_{i-1}\leq x_n/2$, $x_n/2\leq k_1+\cdots+k_{i-1}\leq x_n$ 
and $x_n\leq k_1+\cdots+k_{i-1}
\leq (i-1)(n\max_{j=1, \dots, i-1}p_j+\sqrt n\ln n)=(i-1)(n/m+\sqrt n\ln n)$. 
When $k_1+\cdots+k_{i-1}\leq x_n/2$, \eqref{eq:thetal} ensures that for all $t>0$: 
\begin{eqnarray}
\label{eq:Thetar11}
\Theta_{k_1+\cdots+k_{i-1}}^l(x_n)
&\leq& \exp\Big(t(k_1+\cdots+k_{i-1}-x_n)
-(k_1+\cdots+k_{i-1})\ln (2-e^{-t})\Big)\nonumber\\
&\leq& \exp\Big(-\frac t2x_n\Big).
\end{eqnarray}
When $x_n/2\leq k_1+\cdots+k_{i-1}\leq x_n$, \eqref{eq:thetal} 
ensures that for all $t>0$: 
\begin{eqnarray}
\label{eq:Thetar12}
\Theta_{k_1+\cdots+k_{i-1}}^l(x_n)&\leq& 
\exp\Big(t(k_1+\cdots+k_{i-1}-x_n)-(k_1+\cdots+k_{i-1})\ln (2-e^{-t})\Big)\nonumber \\
&\leq& \exp\Big(-\frac{x_n}{2} \ln(2-e^{-t})\Big).
\end{eqnarray}
When $x_n\leq k_1+\cdots+k_{i-1}\leq (i-1)(n/m+\sqrt n\ln n)$, \eqref{eq:thetalbis} and 
\eqref{eq:ThetaF} in Lemma \ref{lemme:ThetaF} ensure that: 
\begin{eqnarray}
\Theta_{k_1+\cdots+k_{i-1}}^l(x_n)&\leq& 
C\exp\!\Big(\!-c(k_1+\cdots+k_{i-1})\!\min\!\Big(\frac{x_n}{k_1+\cdots+k_{i-1}},
\Big(\frac{x_n}{k_1+\cdots+k_{i-1}}\Big)^2\Big)\!\Big) \nonumber\\
&=&C\exp\Big(-c\frac{x_n^2}{k_1+\cdots+k_{i-1}}\Big)\nonumber\\
&\leq&C\exp\Big(-c\frac{n(\ln n)^2}{(i-1)(n/m+\sqrt n\ln n)}\Big).  
\label{eq:Thetar13} 
\end{eqnarray}
Gathering together the bounds \eqref{eq:Thetar11}, \eqref{eq:Thetar12} and 
\eqref{eq:Thetar13} proves \eqref{eq:thetabullet}, for $\bullet=l$.  
Combining this last fact with the corresponding result 
for $\bullet=r$, and via \eqref{eq:thetabullet} and \eqref{eq:maxC1}, proves Proposition~\ref{prop:maxC1}. 
\end{proof}

\medskip\noindent
Now, thanks to Proposition~\ref{prop:maxC1}, to prove the convergence to zero, 
as $n\to+\infty$, of the first term on the right-hand side of \eqref{eq:maj61}, 
it is enough to prove the same result for
\begin{equation}
\label{eq:ppZ}
\pp\left(\left\{\max_{\kk\in\Ccal^{\#}_{n,i}} 
\left|\sum_{j=\ell_i + 1}^{\ell_i+N_i^*-(k_1+\cdots+k_{i-1})} Z_j^{(i)}\right|
\ge\varepsilon\sqrt n\right\}\cap E_n^i \right), 
\end{equation}
where the $Z_j^{(i)}$ are given in \eqref{eq:Zj}, i.e., 
$Z_j^{(i)}=\big(N_m^{T_i^{j-1}, T_i^j}-1\big)/\sqrt 2$, $i=1, \dots, m-1$, $j\ge 1$.  
Our next elementary proposition, the ultimate before closing this section, provides tail estimates on the partial sums of the $Z_j$ (omitting the indices $i$ for a while).   

\begin{prop}
\label{prop:Zjdev}
Let $(Z_j)_{j\geq1}$ be iid random variables as in \eqref{eq:Zj}.  Then, for suitable positive and finite constants $c$ and $C$, all $x>0$, and all positive integer $k$, 
\begin{eqnarray}
\label{eq:Zjdevaa} 
\pp\Big(\sum_{j=1}^k Z_j\geq x\Big)&\leq& \min_{t>0}\!\Big(\!\exp\!\left(\!-\left(t(x\sqrt 2+k)+k\ln(2-e^{t})\right)\right)\!\Big)\!=:\Theta_k^r(x\sqrt 2),\\
\label{eq:Zjdeva}
\pp\Big(\sum_{j=1}^k Z_j\geq x\Big)&\leq& C\exp\Big(-c\min\Big(\frac xk, 
\Big(\frac xk\Big)^2\Big)\Big),\\
\label{eq:Zjdevb}
\pp\Big(\sum_{j=1}^k Z_j\leq -x\Big)&\leq& 
\min_{t>0}\!\Big(\!\exp\!\left(\!-\left(t(x\sqrt 2-k)+k\ln(2-e^{-t})\!\right)\!\right)\!\!\Big)
\!=:\Theta_k^l(x\sqrt 2),\\
\label{eq:Zjdevbb}
\pp\Big(\sum_{j=1}^k Z_j\leq -x\Big)&\leq&C\exp\Big(-c\min\Big(\frac xk, 
\Big(\frac xk\Big)^2\Big)\Big),\quad \mbox{ for } x\leq k.
\end{eqnarray}
\end{prop}

\begin{proof}
Recall from \eqref{eq:Zj} that $Z_j=\big(N_m^{T_i^{j-1}, T_i^j}-1\big)/\sqrt 2$, $i\neq m$, 
and from \eqref{eq:PGFN},  
\begin{equation}
\label{eq:FCNApp}
\ee\Big[x^{N_m^{T_i^{j-1}, T_i^j}}\Big]=\frac 1{2-x}. 
\end{equation}
Hence, using the notation in \eqref{eq:thetar0},
\begin{eqnarray*}
\nonumber
\pp\Big(\sum_{j=1}^k Z_j\geq x\Big)
\leq \min_{t>0} \left(e^{-t(x\sqrt 2+k)}\  
\ee\Big[\exp\big(tN_m^{T_i^{j-1}, T_i^j}\big)\Big]^k\right)=\Theta_k^r(x\sqrt 2), 
\end{eqnarray*}
and \eqref{eq:Zjdeva} follows from \eqref{eq:thetar0} and \eqref{eq:thetar} 
in (the proof of) Proposition~\ref{prop:maxC1} (with its notation)  
and from \eqref{eq:ThetaF} in Lemma~\ref{lemme:ThetaF}. 
Similarly, using the notation in \eqref{eq:thetal}
\begin{eqnarray}
\nonumber
\pp\Big(\sum_{j=1}^k Z_j\leq -x\Big)&=&\pp\left(\sum_{j=1}^k 
\Big(1-N_m^{T_i^{j-1}, T_i^j}\Big)\geq x\sqrt 2\right)\\
\nonumber
&\leq& \min_{t>0} \left(e^{-t(x\sqrt 2-k)}\  
\ee\Big[\exp\big(-tN_m^{T_i^{j-1}, T_i^j}\big)\Big]^k\right)=\Theta_k^l(x\sqrt 2).
\end{eqnarray}
which is \eqref{eq:Zjdevb}. 
As previously observed via \eqref{eq:thetalbis}, when $x\leq k$, the minimization for $\Theta_k^l(x)$ occurs at $t=\ln\big((2k-x)/(2k-2x)\big)$, 
and, once again, \eqref{eq:ThetaF} in Lemma~\ref{lemme:ThetaF} ensures \eqref{eq:Zjdevbb}.
\end{proof}

We are now ready to move towards completing this section.  
From its very definition in~\eqref{eq:Cni**},  
\begin{eqnarray*}
&&{\Ccal}^{\#}_{n,i} = 
\Ccal_n^{i-1} \cap \left\{\bbe[N_i^*]\le \ell_i=k_1+\cdots +k_i
\le \bbe[N_i] +\sqrt n\ln n -(N_i^* - \bbe[N_i^*]) \right\} \\ 
&&\qquad \subset 
\left\{k_1+\cdots+k_{i-1}\le \sum_{j=1}^{i-1}\bbe[N_j] + (i-1)\sqrt n\ln n \right\} \\
&&\qquad \qquad \cap 
\left\{\bbe[N_i^*]\le \ell_i=k_1+\cdots +k_i
\le \bbe[N_i] + \sqrt n\ln n -(N_i^* - \bbe[N_i^*]) \right\} \\
&&\qquad \subset 
\left\{k_1+\cdots+k_{i-1}\le (i-1)(n\max_{j=1, \dots, i-1}p_j + \sqrt n\ln n) \right\}\\
&&\qquad \qquad \cap 
\left\{\bbe[N_i^*]\le \ell_i=k_1+\cdots +k_i
\le \bbe[N_i] + \sqrt n\ln n -(N_i^* - \bbe[N_i^*]) \right\}.
\end{eqnarray*}  
Therefore, recalling also from \eqref{eq:EVar_Ni} 
that $\bbe[N_i^*] = k_1+\cdots+k_{i-1}$, \eqref{eq:ppZ} 
is upper bounded by:   
\begin{eqnarray}
&&\pp\left(\left\{\max_{\begin{subarray}{l}k_1+\cdots+k_{i-1}\leq (i-1)(n/m+\sqrt n\ln n)
\\k_1+\cdots+k_{i-1}
\leq \ell_i
\leq \bbe[N_i]+\sqrt n\ln n-(N_i^*-(k_1+\cdots+k_{i-1}))\end{subarray}} 
\left|\sum_{j=\ell_i+1}^{\ell_i+N_i^*-(k_1+\cdots+k_{i-1})} Z_j\right|\ge 
\varepsilon\sqrt n\right\}\cap E_n^i\!\right) \nonumber \\
\nonumber
&&\hskip 1.7cm \le \pp\left(\max_{\begin{subarray}{l}k_1+\cdots+k_{i-1}
\le (i-1)(n/m+\sqrt n\ln n)
\\k_1+\cdots+k_{i-1}
\leq \ell_i
\leq \bbe[N_i]+\sqrt n\ln n + x_n\end{subarray}}
\max_{|n_i|\leq x_n} \left|\sum_{j=\ell_i+1}^{\ell_i+n_i} Z_j\right|
\ge\varepsilon\sqrt n\right)
\quad \mbox{(recall \eqref{eq:Eni})}\\
\nonumber
&&\hskip 1.7cm \le \pp\left(\max_{\ell_i\leq \bbe[N_i]
+\sqrt n\ln n+x_n}\max_{|n_i|\leq x_n} 
\left|\sum_{j=\ell_i+1}^{\ell_i+n_i} Z_j\right|\ge \varepsilon\sqrt n\right)
\\
&&\hskip 1.7cm \le 3nx_n\max_{\begin{subarray}{c}\ell_i\leq \bbe[N_i]
+\sqrt n\ln n+x_n\\|n_i|\leq x_n\end{subarray}} 
\pp\left(\left|\sum_{j=\ell_i+1}^{\ell_i+n_i} Z_j\right|
\ge\varepsilon\sqrt n\right)\nonumber \\
&&\hskip 1.7cm \le 3nx_n\max_{\begin{subarray}{c}\ell_i
\le \bbe[N_i]+\sqrt n\ln n
+x_n\\0\leq n_i
\leq x_n\end{subarray}} \left(\Theta_{n_i}^l(\varepsilon\sqrt{2n})
+\Theta_{n_i}^r(\varepsilon\sqrt{2n})\right),\label{eq:techmax}
\end{eqnarray}
where, in the next to last inequality, we used the usual (sharp in the iid case) bounding 
of the maximum via the number of terms times the maximal probability; 
while in the last one, $|n_i|\leq x_n$ was changed into 
$0\leq n_i\leq x_n = \sqrt n \ln n$.

Our final task is to show that 
\begin{equation}
\label{eq:thetabullet2}
\lim_{n\to+\infty} nx_n\max_{0\leq n_i\leq x_n} 
\Theta_{n_i}^\bullet(\varepsilon\sqrt{2n})=0, 
\end{equation}
for $\bullet\in\{l,r\}$. 
This relies again on Lemma~\ref{lemme:ThetaF} and Proposition~\ref{prop:Zjdev}.
For $\bullet=r$,  when $k<\varepsilon \sqrt{2n}$, \eqref{eq:Zjdevaa} 
and \eqref{eq:Zjdeva} entail that,    
\begin{equation}\label{eq:thetar1}
\Theta_k^r(\varepsilon \sqrt{2n})\leq C\exp(-c\varepsilon \sqrt{2n}); 
\end{equation}
while, for $\varepsilon\sqrt{2n}\leq k\leq x_n$, they entail that,  
\begin{equation}\label{eq:thetar2}
\Theta_k^r(\varepsilon\sqrt{2 n})
\leq C\exp\big(-2c\varepsilon^2n/k\big)
\leq C\exp\big(-2c\varepsilon^2n/x_n\big)
=C\exp\big(-2c\varepsilon^2\sqrt n /\ln n\big).
\end{equation}
Therefore, for $\bullet=r$, \eqref{eq:thetabullet2} follows from \eqref{eq:thetar1} and 
\eqref{eq:thetar2}. 
Let us now turn our attention to $\bullet=l$.  
When $\varepsilon\sqrt{2n}\leq k\leq x_n$, \eqref{eq:Zjdevbb} entails that, 
\begin{equation}
\label{eq:thetal1}
\Theta_k^l(\varepsilon\sqrt{2n})
\leq C\exp\big(-2c\varepsilon^2n/k\big)\leq C\exp\big(-2c\varepsilon^2n/x_n\big)
= C\exp\big(-2c\varepsilon^2 \sqrt n/\ln n\big).
\end{equation}
For $k\leq \varepsilon\sqrt{n/2}$, \eqref{eq:Zjdevb} entails that, for any $t>0$,  
\begin{equation}
\label{eq:thetal2Est}
\Theta_k^l(\varepsilon\sqrt{2n})\leq 
\exp\Big(t(k-\varepsilon\sqrt{2n})-k\ln (2-e^{-t})\Big)
\leq \exp\Big(-\varepsilon t\sqrt{n/2}\Big).
\end{equation}  
For $\varepsilon\sqrt{n/2}\leq k\leq \varepsilon\sqrt{2n}$, \eqref{eq:Zjdevb} 
entails that, for any $t>0$, 
\begin{equation}
\label{eq:thetal3}
\Theta_k^l(\varepsilon\sqrt{2 n})\leq\exp\Big(t(k-\varepsilon\sqrt{n/2})-k\ln (2-e^{-t})\Big)
\leq \exp\Big(-\varepsilon t\sqrt{n/2} \ln(2-e^{-t})\Big).
\end{equation}
Therefore, for $\bullet=l$, \eqref{eq:thetabullet2} follows from \eqref{eq:thetal1}, 
\eqref{eq:thetal2Est}, and \eqref{eq:thetal3}. 
Gathering all the intermediate results, for any $i=2, \dots, m-1$,  
$$
\lim_{n\to+\infty}
\pp\left(\left\{\max_{\kk\in\Ccal^{\#}_{n,i}} 
\left|\sum_{j=\ell_i + 1}^{\ell_i+N_i^*-(k_1+\cdots+k_{i-1})} Z_j^{(i)}\right|
\ge\varepsilon\sqrt n\right\}\cap E_n^i \right)=0, 
$$
and therefore,  
$$
\lim_{n\to+\infty}\pp\left(\max_{\kk\in\Ccal_n^*} 
\left|\sum_{j=k_1+\cdots+k_i+1}^{N_i^*+k_i}  
\frac{Z_j^{(i)}}{\sqrt n}\right|\ge\varepsilon\right)=0.
$$
The goal of this section has thus been achieved:
the quantities \eqref{eq:max11} and \eqref{eq:max12} have the same weak limit.


\subsection{The Constraints}
\label{sec:constraints}

To deal with the third heuristic limit, we now need to obtain the convergence 
of the random set of constraints towards a deterministic set of constraints.  
This fact will follow from the various reductions obtained to date as well as new 
arguments developed from now on. To start with, let us recall two elementary facts 
about convergence in distribution.

The first fact asserts that if $(f_n)_{1\leq n\leq \infty}$ is a sequence of 
Borel functions such that $x_n\to x_\infty$ implies that 
$f_n(x_n)\to f_\infty(x_\infty)$, and  
if $(X_n)_{n\ge 1}$ is a sequence of random variables such that 
$X_n\Rightarrow X_\infty$, then $f_n(X_n)\Rightarrow f_\infty(X_\infty)$. 
Indeed, via the Skorohod representation theorem for $C_0([0,1])$-valued random variables, there exist a probability space and $C_0([0,1])$-valued random variables $Y_n$, $1\le n\le \infty$, such 
that $Y_n\stackrel{\call}{=} X_n$, $1\le n\le \infty$, and $Y_n\to Y_\infty$ 
with probability one.  But, by hypothesis, $f_n(Y_n)\to f_\infty (Y_\infty)$, 
with probability one. Therefore $f_n(X_n)\Rightarrow f_\infty (X_\infty)$.

The second elementary fact is as follows: Let $(X_n)_{n\ge 1}$ be a 
sequence of random variables such that $X^\pm_n\Rightarrow Y$, 
then $X_n\Rightarrow Y$, where $x^+=\max(x,0)$ and $x^-=\min(x,0)$.  
Indeed, $\bbp (X^+_n\le x)\le \bbp(X_n\le x)\le\bbp (X^-_n\le x)$, 
for all $x\in\bbR$.

Using these two elementary facts, let us return to our derandomization problem. 
Recalling \eqref{eq:max12}, and using the polygonal structure of 
the processes $B_n^X$ and $B_n^Y$, we have 
\begin{eqnarray}
\nonumber
M_n&:=&\max_{\kk\in\Ccal_n}\left(F_X\left(B^X_n,\frac{\kk}{n}\right)\wedge
F_Y\left(B^Y_n,\frac{\kk}{n}\right)\right),
\end{eqnarray}
where
\begin{eqnarray}
\label{eq:funcFX}
F_X\big({\bf u},{\bf t}\big)&=&
\frac 1m\sum_{i=1}^{m-1} u_i(p_i(X))-\sum_{i=1}^{m-1} 
\left(u_i\Big(\sum_{j=1}^{i}t_j\Big)-u_i\Big(\sum_{j=1}^{i-1}t_j\Big)\right),\\
\label{eq:funcFY}
F_Y\big({\bf u},{\bf t}\big)&=&
\frac 1m\sum_{i=1}^{m-1} u_i(p_i(Y))-\sum_{i=1}^{m-1} 
\left(u_i\Big(\sum_{j=1}^{i}t_j\Big)-u_i\Big(\sum_{j=1}^{i-1}t_j\Big)\right),
\end{eqnarray}
for ${\bf u}=(u_1, \dots, u_{m-1})\in \big(C_0([0,1])\big)^{m-1}$ and ${\bf t}=(t_1, \dots, t_{m-1})\in[0,1]^{m-1}$.
Now, let
\begin{equation}
\label{eq:Cpm}
\Ccal^{\pm}_n=\left\{\kk=(k_i)_{1\leq i\leq m-1}:\forall\ i=1,\dots, m-1, 0\leq k_i\leq n \mbox{ and }
\sum^i_{j=1} \frac{k_j}n\leq p_i\pm 2x_n\right\},
\end{equation}
with $x_n=\sqrt n\ln(n)$
as in \eqref{eq:xn},
and let
\begin{equation}
\label{eq4b}
M^\pm_n=\max_{\kk \in\Ccal^{\pm}_n}\left(F_X\left(B^X_n,\frac{\kk}n\right)
\wedge F_Y\left(B^Y_n,\frac{\kk}{n}\right)\right).
\end{equation}
Since
\begin{eqnarray*}
N_i(X)-N^*_i(X)&=&np_i-\sum_{j=1}^{i-1}k_i+\Big(\big(N_i(X)-\ee[N_i(X)]\big)-\big(N_i^*(X)-\ee[N_i^*(X)]\big)\Big),\\
\end{eqnarray*}
with a similar statement replacing $X$ by $Y$, the condition
$$
k_i\leq \big(N_i(X)-N^*_i(X)\big)\wedge \big(N_i(Y)-N^*_i(Y)\big),
$$
in the definition \eqref{eq8C}--\eqref{eq8Ci} of $\Ccal_n$, writes as $\sum_{j=1}^ik_i/n\leq p_i+R_n^i(X,Y)/n$ where
\begin{eqnarray}
\nonumber
R_n^i(X,Y)&=&\Big(\big(N_i(X)-\ee[N_i(X)]\big)-\big(N_i^*(X)-\ee[N_i^*(X)]\big)\Big)\\
\label{eq:Rni}
&&\hspace{2cm}\wedge\Big(\big(N_i(Y)-\ee[N_i(Y)]\big)-\big(N_i^*(Y)-\ee[N_i^*(Y)]\big)\Big).
\end{eqnarray}
Now let
$$
F_n:=\bigcap_{i=1}^{m-1}\big\{|N_i-\ee[N_i]|\leq x_n\big\}\cap E_n^i,
$$
with $E_n^i$ defined in \eqref{eq:Eni}.
From \eqref{hoe} and Proposition~\ref{prop:maxC1}, we have $\lim_{n\to+\infty}\pp\big(F_n^c\big)=0$ 
and, on $F_n$,  $R_n^i(X,Y)\leq 2x_n$, for all $1\leq i\leq m$.
Therefore, when $F_n$ is realized, $\Ccal_n$ in \eqref{eq8C} is encapsulated as follows: $\Ccal_n^{-}\subset\Ccal_n\subset\Ccal^{+}_n$ , and
\begin{equation}
\label{eq6b}
M^-_n\leq M_n\leq M^+_n.
\end{equation}
Clearly, 
\begin{equation}
\label{eq4bb}
M^\pm_n=\max_{{\bf t}\in\Ccal^{\pm}_n}\left(F_X\left(B^X_n,{\bf t}\right)
\wedge F_Y\left(B^Y_n,{\bf t}\right)\right),
\end{equation}
where now 
\begin{equation}
\label{eq3bb}
\Ccal^{\pm}_n=\left\{{\bf t}=(t_i)_{1\leq i\leq m-1}\in {}[0,1]^{m-1} : 
\forall\ i=1,\dots, m-1, \sum^i_{j=1} t_j\leq p_i\pm 2\frac{x_n}n\right\}.
\end{equation}
Next,
\begin{align*}
\pp\big(M_n\leq x\big)
&\leq \pp\Big(\{M_n\leq x\}\cap F_n\Big)
+\pp\big(F_n^c\big)\\
&\leq \pp\big(\{M^-_n\leq x\}\cap F_n\big)
+\pp\big(F_n^c\big)\\
&\leq \pp\big(M^-_n\leq x\big)+\pp\big(F_n^c\big),
\end{align*}
therefore 
\begin{equation}
\label{eq7b}
\limsup_{n\to+\infty} \pp(M_n\leq x)\leq\limsup_{n\to +\infty}\pp(M^-_n\leq x). 
\end{equation}
Similarly, 
\begin{align*}
\pp\big(M_n\leq x\big)
&= \pp\Big(\{M_n\leq x\}\cap F_n\Big)+\pp\Big(\{M_n\le x\}\cap F_n^c\Big)\\
&\geq \pp\Big(\{M^+_n\leq x\}\cap F_n \Big)\\
&\geq \pp\big(M^+_n\leq x\big)-\pp\big(F_n^c\big),
\end{align*}
and therefore
\begin{equation}
\label{eq8b}
\liminf_{n\to +\infty} \pp(M_n\leq x)\geq\liminf_{n\to+\infty}\pp\big(M^+_n\leq x\big).
\end{equation}
Combining \eqref{eq7b} and \eqref{eq8b} with the second elementary fact described 
above, our goal is now to show that the convergence in distribution of both $M^+_n$ and $M^-_n$ towards
\begin{equation}
\label{eq9b}
M_\infty=\max_{{\bf t}\in \calv}\big(F_X(B^X,{\bf t})\wedge F_Y(B^Y,{\bf t})\big),
\end{equation}
holds true, where
$$
\calv:=\calv(p_1, \dots, p_{m-1})=\left\{{\bf t}=(t_j)_{1\leq j\leq m-1} 
\in[0,1]^{m-1} : \forall i=1, 
\dots, m-1, \sum_{j=1}^i t_j\leq p_i\right\}.
$$
To do so, first note that by Donsker's theorem $(B^X_n, B^Y_n)\Rightarrow (B^X,B^Y)$ 
and we now wish to apply the first elementary fact, recalled above, to the functions
\begin{eqnarray}
\label{eq:fnpm}
f^\pm_n({\bf u},{\bf v})
=\max_{{\bf t}\in\Ccal^\pm_n}\big(F_X\left({\bf u},{\bf t}\right)
\wedge F_Y\left({\bf v},{\bf t}\right)\big), 
\end{eqnarray}
and 
\begin{eqnarray}
\label{eq:fpm}
f_\infty({\bf u},{\bf v})=\max_{{\bf t}\in\calv}
\big(F_X({\bf u},{\bf t})\wedge F_Y({\bf v},{\bf t})\big).
\end{eqnarray}
With these notations, $M_n^\pm=f_n^\pm(B_n^X, B_n^Y)$ and $M_\infty=f_\infty(B^X, B^Y)$. 
In other words, we wish to show that 
$({\bf u}_n,{\bf v}_n)\to ({\bf u},{\bf v})$ in $(C_0([0,1]))^{m-1}$ implies that $f_n({\bf u}_n,{\bf v}_n)\to f_\infty({\bf u},{\bf v})$.
To start with, 
\begin{equation}
\label{eq10b}
|f^\pm_n({\bf u}_n,{\bf v}_n)-f_\infty({\bf u},{\bf v})|
\leq|f^\pm_n({\bf u}_n,{\bf v}_n)-f^\pm_n({\bf u},{\bf v})|
+|f^\pm_n({\bf u},{\bf v})-f_\infty({\bf u},{\bf v})|, 
\end{equation}
and we continue by 
estimating $|f^\pm_n({\bf u}_n,{\bf v}_n)-f^\pm_n({\bf u},{\bf v})|$. 
But,   
\begin{align}
\nonumber
&|f^\pm_n({\bf u}_n,{\bf v}_n)-f^\pm_n({\bf u},{\bf v})|\\
&\qquad \leq \max_{{\bf t}\in\Ccal^\pm_n}
\Bigl|\Big(F_X\left({\bf u}_n,{\bf t}\right)\wedge
F_Y\left({\bf v}_n,{\bf t}\right)\Big) -
\Big(F_X\left({\bf u},{\bf t}\right) \wedge F_Y\left({\bf v},{\bf t}\right)
\Big)\Bigr| \nonumber\\
\label{eq:techineg2}
&\qquad\leq\max_{{\bf t}\in\Ccal^\pm_n}\max\Big(\left|
F_X\left({\bf u}_n,{\bf t}\right) -F_X\left({\bf u},{\bf t}\right)\right|,
\left|F_Y\left({\bf v}_n,{\bf t}\right)-F_Y\left({\bf v},{\bf t}\right)
\right|\Big)
\\
\label{eq:tech9887}
&\qquad \leq c\max_{{\bf t}\in\Ccal^\pm_n}\max
\Big(\left| {\bf u}_n({\bf t})-{\bf u}({\bf t})\right|,
\left| {\bf v}_n({\bf t})-{\bf v}({\bf t})\right|\Big), 
\end{align}
making use of Lemma~\ref{lemme:ineg} in \eqref{eq:techineg2}, and 
by the linearity of both $F_X$ and $F_Y$, with respect to their first 
argument, in \eqref{eq:tech9887} and where, further, $c$ is a finite positive constant (depending explicitly on $m$). 
Therefore, 
$$
|f^\pm_n({\bf u}_n,{\bf v}_n)-f^\pm_n({\bf u},{\bf v})|
\leq c\max\big(\| {\bf u}_n-{\bf u}\|_\infty, \|{\bf v}_n-{\bf v}\|_\infty\big), 
$$
and so if $({\bf u}_n,{\bf v}_n)\to ({\bf u},{\bf v})$, it follows 
that $f^\pm_n({\bf u}_n,{\bf v}_n)-f^\pm_n({\bf u},{\bf v})\to 0$.

\medskip\noindent
In order to complete the proof of $M^\pm_n\Rightarrow M_\infty$ and thus that 
of $M_n\Rightarrow M_\infty$, let us now estimate the right-most expression 
in \eqref{eq10b}.  

At first, note that $\Ccal^-_n\subset \calv\subset \Ccal^+_n$, hence 
\begin{equation}
\label{eq12}
f^-_n({\bf u},{\bf v})\leq f_\infty({\bf u},{\bf v})\leq f^+_n({\bf u},{\bf v}).
\end{equation}
Next, via \eqref{eq:fnpm} and \eqref{eq:fpm}, set  
$f^+_n({\bf u},{\bf v})=\max_{{\bf t}\in\Ccal^+_n}\theta_{{\bf u},{\bf v}}({\bf t})$, 
and 
$f_\infty({\bf u},{\bf v})=\max_{{\bf t}\in\calv}\theta_{{\bf u},{\bf v}}({\bf t})$, 
where $\theta_{{\bf u},{\bf v}}({\bf t})=F_X({\bf u},{\bf t})\wedge F_Y({\bf v},{\bf t})$.
Since $\Ccal^-_n\subset \Ccal^-_{n+1}$, for $n\geq 1$, it follows (as shown next) 
that $f^-_n({\bf u},{\bf v})\to \max_{{\bf t}\in \bigcup_{n\geq 1}\Ccal^-_n} 
\theta_{{\bf u},{\bf v}}({\bf t})$. 
Indeed, $\lim_{n\to+\infty}f^-_n({\bf u},{\bf v})
\le \max_{{\bf t}\in \bigcup_{n\geq 1}\Ccal^-_n} \theta_{{\bf u},{\bf v}}({\bf t})$ 
and if the previous inequality 
were strict, there would now be $K\in(0,+\infty)$ such that 
$$
\max_{{\bf t}\in\Ccal^-_n} \theta_{{\bf u},{\bf v}}({\bf t})\leq K 
<\max_{{\bf t}\in \bigcup_{n\geq 1}\Ccal^-_n} \theta_{{\bf u},{\bf v}}({\bf t}).
$$
The left-hand side inequality implies that for all $n\geq 1$, 
and ${\bf t}\in\Ccal^-_n$, $\theta_{{\bf u},{\bf v}}({\bf t})\leq K$, 
contradicting the right-hand side inequality. 

\noindent
Since $\Ccal^+_n\supset \Ccal^+_{n+1}$, for $n\geq 1$, it also follows that  
$f^+_n({\bf u},{\bf v})\to \max_{{\bf t}\in \bigcap_{n\geq 1}\Ccal^+_n} 
\theta_{{\bf u},{\bf v}}({\bf t})$. 
Indeed, we have $\lim_{n\to+\infty}f^+_n({\bf u},{\bf v})
\geq \max_{{\bf t}\in \bigcap_{n\geq 1}\Ccal^-_n} \theta_{{\bf u},{\bf v}}({\bf t})$ 
and if the previous inequality were strict, there would be $K\in(0,+\infty)$ such that 
$$
\max_{{\bf t}\in\Ccal^+_n} \theta_{{\bf u},{\bf v}}({\bf t})
\ge K >\max_{{\bf t}\in \bigcap_{n\geq 1}\Ccal^+_n} \theta_{{\bf u},{\bf v}}({\bf t}).  
$$
The left-hand side inequality implies that for any $n\geq 1$, 
there exists ${\bf t}_n\in\Ccal^+_n$ 
with $\theta_{{\bf u},{\bf v}}({\bf t}_n)\ge K$.  
Up to a subsequence ${\bf t}_n\to {\bf t}^\ast \in \bigcap_{n\geq 1} \Ccal^+_n$ 
and by the continuity of $\theta_{{\bf u},{\bf v}}$,  
$\theta_{{\bf u},{\bf v}}({\bf t}^\ast)\geq K$, 
which is inconsistent with the previous right-hand side inequality. 

\noindent 
Finally, since $\bigcup_{n\geq 1}\Ccal^-_n={\cal V}^\circ$, the interior of $\cal V$, and since $\bigcap_{n\geq 1}\Ccal^+_n=\overline{\cal V}={\cal V}$, the closure of $\cal V$, we have 
\begin{equation}
\label{eq:f+f-f}
\lim_{n\to+\infty}f^-_n({\bf u},{\bf v})
=\max_{{\bf t}\in{\cal V}^\circ} \theta_{{\bf u},{\bf v}}({\bf t})
\leq f_\infty({\bf u},{\bf v})
=\max_{{\bf t}\in{\cal V}} \theta_{{\bf u},{\bf v}}({\bf t})=
\lim_{n\to+\infty}f^+_n({\bf u},{\bf v}).  
\end{equation}
It remains to show that the maximum of $ \theta_{{\bf u},{\bf v}}$ on ${\cal V}$ 
is attained on ${\cal V}^\circ$ for $\pp_{(B^X, B^Y)}$-almost 
all $({\bf u},{\bf v})$, i.e., that 
\begin{equation}
\label{eq:Wcirc}
\pp\left(\max_{{\bf t}\in{\cal V}(1/m, \dots, 1/m)^\circ} \theta_{B^X,B^Y}({\bf t})
=\max_{{\bf t}\in{\cal V}(1/m, \dots, 1/m)} \theta_{B^X,B^Y}({\bf t})\right)
=1. 
\end{equation}
With \eqref{eq:Wcirc}, \eqref{eq:f+f-f} entails 
$\lim_{n\to+\infty}f^\pm_n({\bf u},{\bf v})=f_\infty({\bf u},{\bf v})$
for $\pp_{(B^X, B^Y)}$-almost all $({\bf u},{\bf v})$, i.e.,  
the right-most expression in \eqref{eq10b} converges to $0$ and, 
as previously explained, this gives 
$M_n^\pm\Rightarrow M_\infty$ and $M_n\Rightarrow M_\infty$.  

\noindent
In order to complete \eqref{eq:Wcirc} we anticipate, in the second equality below, on the results of Section~\ref{sec:LinearTrans} in which parameters are changed via: $s_1=u_1, s_1+s_2 = u_2, \dots, s_1+\cdots+s_{m-1}=u_{m-1}$ and where we prove that
$$
\big(\theta_{B^X,B^Y}({\bf t})\big)_{{\bf t}\in {\cal V}(1/m, \dots, 1/m)}
\stackrel{{\cal L}}{=}
\frac 1{\sqrt m}\big(\theta_{B^X,B^Y}({\bf s})\big)_{{\bf s}\in {\cal V}(1, \dots, 1)}\\
{=}
\frac 1{\sqrt{2m}} \big(\widetilde\theta_{B_1,B_2}({\bf u})\big)_{{\bf u}\in {\cal W}_m(1)}, 
$$ 
where ${\cal W}_m(1) = \{0=u_0\le u_1\le \cdots \le u_{m-1}\le u_m=1\}$, 
\begin{eqnarray}
\label{eq:tildetheta}
\widetilde\theta_{B_1,B_2}({\bf u})&=&
\left(-\frac 1m\ 
\sum_{i=1}^{m}B^{(i)}_1(1)+\sum_{i=1}^{m} 
\left(\!B^{(i)}_1(u_i)-B^{(i)}_1(u_{i-1})\right)\right)\\
\nonumber&
&\hspace{2cm} \wedge
\left(-\frac 1m \sum_{i=1}^{m} B^{(i)}_2(1)+\sum_{i=1}^{m} \left(B^{(i)}_2(u_i) 
- B^{(i)}_2(u_{i-1})\right)\right), 
\end{eqnarray}
and with $B_1$ and $B_2$ two independent, standard, $m$-dimensional Brownian on $[0,1]$. 
The property \eqref{eq:Wcirc} is thus equivalent to 
\begin{equation}
\label{eq:Wcirc2}
\pp\left(\max_{{\bf u}\in{\cal W}_m(1)^\circ} \widetilde\theta_{B_1,B_2}({\bf u})
=\max_{{\bf u}\in{\cal W}_m(1)} \widetilde\theta_{B_1,B_2}({\bf u})\right)
=1. 
\end{equation}
The advantage of \eqref{eq:Wcirc2} over \eqref{eq:Wcirc} is that the former  
involves two standard Brownian motions each one having {\it independent} coordinates.  
Roughly speaking, the property \eqref{eq:Wcirc2} should be derived from 
the following observation: when ${\bf u}\in\partial{\cal W}_m(1)$, 
then $u_k=u_{k+1}$, for some index $k$, and for such a ${\bf u}$, the sum 
$\sum_{i=1}^{m} \big(B^{(i)}_1(u_i)-B^{(i)}_1(u_{i-1})\big)$ 
contains only $m-1$ terms. 
Letting ${\bf u}_\varepsilon$ be given by 
$$
u_{\varepsilon, i}=u_i, \quad i\not=k+1, 
\quad \mbox{ and } \quad u_{\varepsilon, k+1}=u_k+\varepsilon, 
$$
we have  
\begin{align*}
&\sum_{i=1}^{m} \big(B^{(i)}_1(u_{\varepsilon,i})-B^{(i)}_1(u_{\varepsilon,i-1})\big) \\
&\quad = \sum_{i=1}^{m} \big(B^{(i)}_1(u_i)-B^{(i)}_1(u_{i-1})\big)
+\big(B^{(k+1)}_1(u_k+\varepsilon)-B^{(k+1)}_1(u_k)\big)\\
&\qquad\qquad +\big(B^{(k+2)}_1(u_k)-B^{(k+2)}_1(u_k+\varepsilon)\big).
\end{align*}
The terms $\big(B^{(k+1)}_1(u_k+\varepsilon)-B^{(k+1)}_1(u_k)\big)$ 
and $\big(B^{(k+2)}_1(u_k)-B^{(k+2)}_1(u_k+\varepsilon)\big)$ are independent 
of $\sum_{i=1}^{m} \big(B^{(i)}_1(u_i)-B^{(i)}_1(u_{i-1})\big)$ and from 
standard properties of Brownian motion, almost surely, the sum 
$\big(B^{(k+1)}_1(u_k+\varepsilon)-B^{(k+1)}_1(u_k)\big)
+ \big(B^{(k+2)}_1(u_k)-B^{(k+2)}_1(u_k+\varepsilon)\big)$ takes positive value 
for arbitrarily small $\varepsilon>0$.  
Since the same is true for the second term in \eqref{eq:tildetheta} relative 
to $B_2$, it follows that in the vicinity of each ${\bf u}\in\partial{\cal W}_m(1)$, 
there is ${\bf u}_\varepsilon\in {\cal W}_m(1)$ 
with $\widetilde\theta_{B_1,B_2}({\bf u}_\varepsilon)>\widetilde\theta_{B_1,B_2}({\bf u})$.
Therefore, $\max_{{\bf u}\in{\cal W}_m(1)} \widetilde\theta_{B_1,B_2}({\bf u})$ 
is attained in ${\cal W}_m(1)^\circ$,  
and so both \eqref{eq:Wcirc2} and \eqref{eq:Wcirc} hold true, leading to 
$M_n\Rightarrow M_\infty$.


\subsection{Final Step: A Linear Transformation}
\label{sec:LinearTrans}

By combining the results of the previous three subsections, we proved that
\begin{eqnarray}
\!\!\frac{\lcin- n/m}{\sqrt{2n}} 
\Rightarrow \!\max_{{\cal V}(1/m, \dots, 1/m)}\!\!\!\!\min\!\left(\!\frac 1m\! 
\sum_{i=1}^{m-1}\!B^{(i),X}\!\!\left(\!\frac{1}{m}\right)\!-\sum_{i=1}^{m-1} 
\!\left(\!\!B^{(i),X}\!\!\left(\sum_{j=1}^{i}t_{j}\!\right)\!-B^{(i),X}
\!\!\left(\sum_{j=1}^{i-1}t_{j}\!\right)\!\right)\!,\right.\nonumber \\
\quad\quad\quad\!\!\left.\frac 1m \sum_{i=1}^{m-1}\!B^{(i),Y}\!\!\left(\frac{1}{m}\right)-
\sum_{i=1}^{m-1}\!\left(\!B^{(i),Y}\!\left(\sum_{j=1}^{i}t_{j}\!\right)
\!-B^{(i),Y}\!\left(\sum_{j=1}^{i-1}t_{j}\!\right)\!\right)\!\right),\label{presqder} 
\end{eqnarray}
where the maximum is taken over 
${\ttt}=(t_1, \dots, t_{m-1})\in {\cal V}(1/m, \dots, 1/m)$.
Now, via the linear transformations of the parameters given by $s_i=m\sum_{j=1}^i t_j$, $i=1, \dots, m-1$, $s_0=t_0=0$, and Brownian scaling, the right-hand side of \eqref{presqder} becomes equal, in law, to:   
\begin{eqnarray}
\frac{1}{\sqrt m}\max_{0=s_0\le s_1\le \cdots \le s_{m-1}\le 1}\min\left(\frac 1m\ 
\sum_{i=1}^{m-1}B^{(i),X}(1)-\sum_{i=1}^{m-1} 
\left(\!B^{(i),X}\!(s_i)-B^{(i),X}(s_{i-1})\right),\right.\nonumber \\
\quad\quad\quad \quad\quad\quad \left.\frac 1m \sum_{i=1}^{m-1} B^{(i),Y}(1)-
\sum_{i=1}^{m-1} \left(B^{(i),Y}(s_i) - B^{(i),Y}(s_{i-1})\right)\right).\label{presqder2} 
\end{eqnarray}
Next, for all $t\in [0,1]$ and $i= 1, \dots, m-1$, 
let us introduce the following two pointwise linear transformations:   
\begin{align*}
B^{(i),X}(t) = \frac{B_1^{(m)}(t) - B_1^{(i)}(t)}{\sqrt 2},  \\
B^{(i),Y}(t) = \frac{B_2^{(m)}(t) - B_2^{(i)}(t)}{\sqrt 2},   
\end{align*}
where $B_1$ and $B_2$ are two, standard, $m$-dimensional Brownian motion on $[0, 1]$. 
Clearly $(B^{(1),X}(t), \dots, B^{(m-1),X}(t))_{0\le t \le 1}$ has the correct covariance matrix \eqref{eq:Cov_mB}, and similarly for $B_2$, replacing $X$ by $Y$.  
Moreover, 
\begin{eqnarray}
&&\!\!\frac 1m 
\sum_{i=1}^{m-1}B^{(i),X}(1)-\sum_{i=1}^{m-1} 
\left(\!B^{(i),X}\!(s_i)-B^{(i),X}(s_{i-1})\right)\nonumber \\
&&= -\frac{1}{{\sqrt 2}m}\left(\sum_{i=1}^mB_1^{(i)}(1)\right) 
+ \frac{1}{\sqrt 2}B_1^{(m)}(1)\nonumber \\
&&\quad\quad -\frac{1}{\sqrt 2} \!\sum_{i=1}^{m-1}\!\left(\!B_1^{(m)}\!(s_i)-B_1^{(m)}(s_{i-1})\!\right) 
+ \frac{1}{\sqrt2}\!\sum_{i=1}^{m-1}\!\left(\!B_1^{(i)}\!(s_i)-B_1^{(i)}(s_{i-1})\!\right)  \nonumber\\ 
&&\!\!\!=\!\! \frac{1}{\sqrt 2}\!\left(\!\!-\frac{1}{m}\!\sum_{i=1}^m\!B_1^{(i)}(1) 
+ (B_1^{(m)}\!(1)- B_1^{(m)}\!(s_{m-1}))\! +  
\!\sum_{i=1}^{m-1}\!\!\left(\!B_1^{(i)}\!(s_i)-B_1^{(i)}(s_{i-1})\!\right)\!\!\right).
\label{presqder3} 
\end{eqnarray}
Finally, with the help of \eqref{presqder3} (and the corresponding identity for $Y$),  \eqref{presqder2} becomes:  
\begin{eqnarray}
\frac{1}{\sqrt{2m}}\max_{0=s_0\le s_1\le \cdots \le s_{m-1}\le s_m=1}\min\left(-\frac 1m\ 
\sum_{i=1}^{m}B^{(i)}_1(1)+\sum_{i=1}^{m} 
\left(\!B^{(i)}_1(s_i)-B^{(i)}_1(s_{i-1})\right),\right.\nonumber \\
\quad\quad\quad \quad\quad\quad \left.-\frac 1m \sum_{i=1}^{m} B^{(i)}_2(1)+
\sum_{i=1}^{m} \left(B^{(i)}_2(s_i) - B^{(i)}_2(s_{i-1})\right)\right),\label{der} 
\end{eqnarray}
and the proof of Theorem~\ref{thm1.1} is over.  


\section{Concluding Remarks}
\label{sec:concluding}

Let us discuss below some potential extensions to Theorem~\ref{thm1.1} and some 
questions we believe are of interest.

\medskip

$\bullet$  From the proof presented above, the passage from two to three or more 
sequences is clear: the minimum over two Brownian functionals becomes a minimum 
over three or more Brownian functionals, and such a passage applies 
to the cases touched upon below.    

\medskip

$\bullet$  It is also clear from the proof developed above, that a theorem for two sequences of iid (non-uniform) random variables is also valid.  
Here is what it should look like:  Let $X=(X_i)_{i\geq 1}$ and $Y=(Y_i)_{i\geq 1}$ be two sequences of 
iid random variables with values in  $\cala_m=\{\aa_1<\aa_2< \cdots <\aa_m\}$, a 
totally ordered finite alphabet of cardinality $m$ and with a common law, i.e., 
$X_1\stackrel{{\cal L}}{=}Y_1$.  
Let $\displaystyle p_{\max} = \max_{i=1, 2, \dots, m}\pp(X_1 = \aa_i)$ and let  
$k$ be the multiplicity of $p_{\max}$.  Then,  
\begin{eqnarray}\label{eq00max}
&&\!\!\!\!\!\!\!\!\!\!\!\!\!\!\frac{\lcin - {np_{\max}}}{\sqrt{np_{\max}}} 
\Longrightarrow\nonumber \\
&& \max_{0=t_0 \le t_1 \le \dots \le t_{k-1} \le t_k=1}
\!\min\!\left(\!\frac{\sqrt{1-kp_{\max}}-1}{k}\!\sum_{i=1}^k\!B^{(i)}_1\!(1)
+\sum_{i=1}^k(B^{(i)}_1\!(t_i)-B^{(i)}_1\!(t_{i-1})), \right.\nonumber\\ 
&&\hskip 4.3cm \left.\!\! \frac{\sqrt{1-kp_{\max}}-1}{k}\!\sum_{i=1}^k\!B^{(i)}_2\!(1)+ 
\sum_{i=1}^k(B^{(i)}_2\!(t_i)-B^{(i)}_2\!(t_{i-1}))\!\right)\!\!, 
\end{eqnarray}
where $B_1$ and $B_2$ are two $k$-dimensional standard 
Brownian motions defined on $[0,1]$. So, for instance, 
if $p_{\max}$ is uniquely attained then the limiting law in \eqref{eq00max} 
is the minimum of two centered Gaussian random variables.  
  
Using the sandwiching techniques developed in \cite{HL}, an infinite countable alphabet 
result can also be obtained with \eqref{eq00max}.   

\medskip

$\bullet$  The loss of independence inside the sequences, 
and the loss of identical distributions, both within and between the sequences 
is more challenging.  Results for these situations will be presented elsewhere.  

\medskip

$\bullet$ The length of the longest increasing subsequence of a random word is well known 
to have an equivalent interpretation in percolation theory: Indeed, consider the following 
directed last-passage percolation model in $\bbz^2_+$: let $\Pi_2(n,m)$ be the set of 
directed paths in $\bbz^2_+$ 
from $(0,0)$ to $(n,m)$ with unit steps going either North or East. 
Given random variables $\omega_{i,j}$, $i\geq 0, j\geq 1$, and 
interpreting each $\omega_{i,j}$ as the length of time spent by a path 
at the vertex $(i,j)$, the last-passage time to $(n,m)$ is given by
\begin{equation}
\label{eq:percolation-T}
T_2(n,m)=\max_{\pi\in \Pi_2(n,m)}\Bigg(\sum_{(i,j)\in\pi}\omega_{i,j}\Bigg).
\end{equation}
(See Bodineau and Martin \cite{BM}, and the references therein, for details.)  
In our random word context, when $X=(X_i)_{1\leq i\leq n}$ is a sequence of {\it iid} 
random variables taking their values in a totally ordered finite 
alphabet $\big\{\alpha_1<\alpha_2<\dots< \alpha_m\big\}$ of size $m$, 
taking $\omega_{i,j}=\ind_{\{X_i=\alpha_j\}}$ and $\omega_{0,j}=0$, $j\geq 1$, 
which for each $i$ are {\it dependent} random variables, the length of the longest 
increasing subsequence of the random word is equal to the last 
passage-time $T_2(n,m)$, see \cite{Breton_Houdre}.

Now $\lcin$, the length of the longest {\em common} and increasing subsequences, enjoys a similar percolation theory interpretation, but in $\bbz_+^3$.
Let $\Pi_3(n,n,m)$ be the set of paths in $\bbz^3_+$ from $(0,0,0)$ to $(n,n,m)$ taking either 
unit steps towards the top or steps, of any length, in the horizontal plane 
but neither parallel to the $x$-axis nor to the $y$-axis, i.e., 
\begin{eqnarray*}
\Pi_3(n,n,m)\!\!\!\!&:=&\!\!\!\!\Big\{(u_1, u_2, \dots, u_{n+m})\in \big(\bbz^3_+\big)^{n+m}: 
u_1=(0,0,1), u_{n+m}=(n,n,m),\\
&&
u_{j+1}-u_j\in\big\{\!(0,0,1), (a,b,0) \mbox{ with }a,b\in\bbn\setminus\{0\}\big\}, j=1, 
\cdots, n+m-1\Big\}.
\end{eqnarray*}
Given weights $\omega_{i,j,k}$, $i\geq 0, j\geq 0, k\geq 1$, on the lattice, we can 
consider a quantity analogous to $T_2(n,m)$ in \eqref{eq:percolation-T}, namely,
$$
T_3(n,n,m):=\max_{\pi\in \Pi_3(n,n,m)}\Bigg(\sum_{(i,j,k)\in\pi}\omega_{i,j,k}\Bigg).
$$
In the random word context, taking $\omega_{i,j,k}=\ind_{\{X_i=Y_j=\alpha_k\}}$  
and $\omega_{0,0,k}=0$, $k\geq 1$, as weights, gives $\lcin=T_3(n,n,m)$.
\\
Note that when $X=Y$, $T_3(n,n,m)$ recovers $T_2(n,m)$ since 
$T_2(n,m)$  is unchanged if, in \eqref{eq:percolation-T}, $\Pi_2(n,m)$ is replaced by 
\begin{eqnarray*}
\widetilde\Pi_2(n,m)\!\!&:=&\!\!\Big\{(u_1, u_2, \dots, u_{n+m})
\in \big(\bbz^2_+\big)^{n+m} :u_1=(0,1), u_{n+m}=(n,m),\\
&& \hspace{.7cm} 
u_{j+1}-u_j\in\big\{(0,1), (a,b)\mbox{ with }a,b\in\bbn\setminus\{0\}\big\}, j=1, 
\dots, n+m-1\Big\}.
\end{eqnarray*}

More generally, for $p\geq 3$ sequences of 
letters $X^{(\ell)}=(X_i^{(\ell)})_{1\leq i\leq n}$, $1\le \ell\leq p$, we can similarly consider
\begin{eqnarray*}
\Pi_{p+1}(n,\dots,m)\!\!\!&:=&\!\!\!\Big\{\!(u_1, u_2, \dots, u_{n+m})\in\!\big(\bbz^{p+1}_+\big)^{n+m} 
\!\!\!\!:
u_1=(0,\dots,0,1), u_{n+m}=(n,\dots,n,m)\\
&& \hspace{-2.3cm} 
u_{j+1}-u_j\in\big\{(0,, \dots, 0,1), (a_1,\dots, a_p,0) 
\mbox{ with }a_i\in\bbn\!\setminus\!\{0\}\big\}, j=1, \dots, n+m-1\Big\},
\end{eqnarray*}
and
$$
T_p(n,\dots,n,m):=\max_{\pi\in \Pi_{p+1}(n,\dots,n,m)}\Bigg(\sum_{(i_1,\dots, i_p,k)\in\pi}
\omega_{i_1,\dots, i_p,k}\Bigg).
$$
Then, observe that $\lcin$, for the $p$ sequences, is equal to $T_p(n,\dots,n,m)$, 
where now $\omega_{i_1,\dots, i_p,k}=\ind_{\{X_{i_1}=\dots=X_{i_p}=\alpha_k\}}$ 
and $\omega_{0,\dots,0,k}=0$, $k=1, \dots, m$, are dependent random variables.  

In view of Theorem~\ref{thm1.1} and of \cite{BM}, one would expect that for $m$ fixed and for {\it iid} exponential weights $\omega_{i,j,k}$ with mean one, 
$T_3(n,n,m)$ converges, when properly centered, by $n$, and scaled, by $\sqrt n$, 
towards   
$$\max_{0=t_0 \le t_1 \le \dots \le t_{m-1} \le t_m=1}
\min\left(\sum_{i=1}^m\left(B^{(i)}_1(t_i)-B^{(i)}_1(t_{i-1})\right), 
\sum_{i=1}^m\left(B^{(i)}_2(t_i)-B^{(i)}_2(t_{i-1})\right)\right), 
$$  
with also the trivial modification for $T_p$.  
\medskip

$\bullet$  Starting with Baryshnikov \cite{Bar} and Gravner, Tracy and Widom 
\cite{GTW} (see, also \cite{BGH}, for a 
further description 
and up to date references) a strong interaction has been shown to exist between 
Brownian functionals, originating in queuing theory with  
Glynn and Whitt~\cite{GW} (see also Sepp\"al\"ainen~\cite{Sep}), and 
maximal eigenvalues of Gaussian random matrices.
Likewise, we hypothesize that the max/min functionals obtained here do enjoy a similar 
strong connection (which might extend to spectra and Young diagrams).  
Could it be that the right-hand side of \eqref{eq00} (with or without the linear terms) 
has the 
same law as the maximal eigenvalue of a random matrix model?
Even in the binary case, it would be interesting to find the law of the processes 
$\big({\sqrt 2}\max_{0\leq t \leq 1}\min(B_1(t) -B_1(1)/2, B_2(t)-B_2(1)/2)\big)_{t\ge 0}$ and 
$\big(\max_{0\leq t \leq 1}\min(B_1(t), B_2(t))\big)_{t\ge 0}$ where, say, $B_1$ and $B_2$ are two independent standard linear Brownian motions.
Very preliminary work on these problems was started 
with \href{http://www.proba.jussieu.fr/pageperso/yor/}{Marc Yor}, before his untimely death, 
and this text is dedicated to his memory.

\medskip

$\bullet$ To finish, note that the LCIS problem for two or more uniform random permutations of $\{1,2,\dots, n\}$ has not been studied either, although it certainly deserves to be. 
In point of fact, it is shown in \cite{HI} that, for any two independent uniform random permutations $\sigma_1$ and $\sigma_2$ of $\{1,2,\dots, n\}$ , and  for any $x\in\bbR$, $\bbp(LC_n(\sigma_1,\sigma_2)\le x)=\bbp (LI_n(\sigma_1)\le x)$, where $LI_n(\sigma_1)$ is the length of the longest increasing subsequences of $\sigma_1$. 
Therefore, this equality in law shows the emergence of the Tracy-Widom distribution, which had sometimes been speculated, as the corresponding limiting law. 
Indeed, once we are given the result of Baik, Deift and Johansson \cite{BaikDeiftJohansson99} on the limiting law of $LI_n(\sigma_1)$, a corresponding result (actually equivalent to it) for $LC_n(\sigma_1,\sigma_2)$ is immediate.  
In fact, many of the results on $LI_n(\sigma_1)$ presented in Romik~\cite{Romik}, such as the law of large numbers of Vershik and Kerov~\cite{VK} are instantaneously transferable to equivalent versions for $LC_n(\sigma_1,\sigma_2)$.  

Moreover, for $p\ge 3$ independent and uniform random permutations $\sigma_1,\sigma_2,\dots, \sigma_p$, the methodology developed in \cite{HI} easily shows that $LC_n(\sigma_1,\sigma_2,\dots, \sigma_p)\stackrel{d}{=} LCI_n (\sigma_1,\dots, \sigma_{p-1})$, where $\stackrel{d}{=}$ denotes equality in distribution. 
Therefore, the study of longest common and increasing subsequences in random words or random permutations which might appear, at first, quite artificial is actually intimately related to the study of longest common subsequences.


\appendix

\section{Appendix}

\subsection{Proofs of technical lemmas}
\label{sec:Appendix_lemma}

\subsubsection*{Proof of Lemma~\ref{lemme:ineg}}
First,  
\begin{eqnarray*}
&\Big|\max_{k=1,\dots, K} \big(a_k\wedge b_k\big)-\max_{k=1,\dots, K} 
\big((a_k+c_k)\wedge (b_k+d_k)\big)\Big| \\ 
&\hskip 2.5cm \le \max_{k=1,\dots, K}\big|\big(a_k\wedge b_k\big)-\big((a_k+c_k)
\wedge (b_k+d_k)\big)\big|.
\end{eqnarray*}
Next, the result will follows from the elementary inequality 
\begin{equation}
\label{eq:10b}
(a\wedge b)-(a+c)\wedge (b+d)\leq |c|\vee|d|, 
\end{equation} 
which is valid for all $a, b, c, d \in \bbR$.   Indeed, set $D=(a\wedge b)-(a+c)\wedge (b+d)$ 
and assume (without loss of generality) that $a\leq b$.  If $a+c\leq b+d$, 
then $D=a-(a+c)=-c\leq |c|$.  If $b+d\leq a+c$, then $D=a-b-d$ and so whenever $a\leq b+d$, 
\eqref{eq:10b} is immediate, while if $a\geq b+d$, then $D=a-b-d\leq -d=|d|$ 
since $a-b\leq 0$ and $-d\geq b-a\geq 0$.
\CQFD


\subsubsection*{Proof of Lemma~\ref{lemme:devZ}}
Let $D_n=\big\{\left|N^{(n)}- \ee[N^{(n)}]\right| < x_n\big\}$, 
and for $\varepsilon>0$, let 
$$
A_n(\varepsilon)=\Big\{\Big|\mbox{$\sum_{j\in [N^{(n)}, \ee[N^{(n)}] ]} 
\frac{Z_j}{\sqrt n}$}\Big|\geq\varepsilon\Big\}.
$$ 
Since $\bbp\big(A_n(\varepsilon)\big)\leq \bbp\big(A_n(\varepsilon)\cap D_n\big)+\bbp(D_n^c)$, 
and since $\lim_{n\to\infty}\bbp(D_n^c)=0$, it is enough to show $\lim_{n\to+\infty}\bbp\big(A_n(\varepsilon)\cap D_n\big)=0$.  
But, by Kolmogorov's maximal inequality, 
\begin{eqnarray*}
\bbp\big(A_n(\varepsilon)\cap D_n\big)&\leq& \bbp\left(\max_{|k-\ee[N^{(n)}]|< x_n}\Big|\sum_{j\in [k,\ee[N^{(n)}] ]} \frac{Z_j}{\sqrt n}\Big|\geq \varepsilon\right)\\
&\leq& \frac{x_n\Var(Z_1)}{\varepsilon^2 n}\to 0, \quad n\to+\infty.
\end{eqnarray*}
\CQFD


\subsubsection*{Proof of Lemma~\ref{lemme:Bn_ui}}

First, we show that $\big(B_n^{(k)}\big(N_k/n\big)^2\big)_{n\geq 1}$ is uniformly integrable. 
Proposition~\ref{prop:S} and Remark~\ref{rem:Sim} give
\begin{eqnarray*}
B_n^{(k)}\Big(\frac{N_k}n\Big)&=&\frac{N_m-N_k}{\sqrt{2n}}+o_\pp(1/\sqrt n)\\
\ee\Big[\Big|B_n^{(k)}\Big(\frac{N_k}n\Big)\Big|^p\Big]
&\leq& 2^{p-1}\left(\ee\Big[\Big|\frac{N_m-N_k}{\sqrt{2n}}\Big|^p\Big]+o(n^{-p/2})\right)\\
&=&2^{-1/2}n^{-p/2}\ee\Big[|N_m-N_k|^p\Big]+o(n^{-p/2}).
\end{eqnarray*}
But $N_m-N_k=\sum_{i=1}^n\epsilon_i^{(m,k)}$ where $(\epsilon_i^{(m,k)})_{i\geq 1}$ are iid with $\epsilon_i^{(m,k)}=1$ when $X_i=\alpha_m$, $\epsilon_i^{(m,k)}=-1$ when $X_i=\alpha_k$ and $\epsilon_i^{(m,k)}=0$ otherwise. 
Hence, by the classical Marcinkiewicz-Zygmund inequality, for some constant $C_p$,
\begin{eqnarray*}
\ee\Big[|N_m-N_k|^p\Big]
&=&\ee\Big[\Big|\sum_{i=1}^n\epsilon_i\Big|^p\Big]\\
&\leq& C_p \ee\Big[\Big(\sum_{i=1}^n |\epsilon_i^{(m,k)}|^2|\Big)^{p/2}\Big]\\
&\leq& C_pn^{p/2}.
\end{eqnarray*}
Therefore, for any $p>2$
\begin{equation*}
\label{eq:Bn_ui1}
\sup_{n\geq 1} \ee\Big[\Big|B_n^{(k)}\Big(\frac{N_k}n\Big)\Big|^p\Big]<+\infty
\end{equation*} 
and $\big(B_n^{(k)}\big(N_k/n\big)^2\big)_{n\geq 1}$ is uniformly integrable. 
Next, for $\big(B_n^{(k)}\big(1/m\big)\big)_{n\geq 1}$:
$$
B_n^{(k)}\Big(\frac 1m\Big)=\frac 1{\sqrt n}\sum_{j=1}^{[n/m]}Z_j^{(k)}+\frac{(nm-[nm])}{\sqrt n} Z_{[n/m]+1}^{(k)}.
$$
and 
\begin{eqnarray*}
\ee\Big[\Big|B_n^{(k)}\Big(\frac 1m\Big)\Big|^p\Big]
&\leq & 2^{p-1}n^{-p/2} \ee\bigg[\Big|\sum_{j=1}^{[n/m]}Z_j^{(k)}\Big|^p\bigg] +2^{p-1}n^{-p/2}\ee\Big[\Big|Z_{[n/m]+1}^{(k)}\Big|^p\Big] \\
&\leq & C_p n^{-p/2} \ee\bigg[\Big(\sum_{j=1}^{[n/m]}|Z_j^{(k)}|^2\Big)^{p/2}\bigg] +2^{p-1}n^{-p/2} \ee\Big[\Big|Z_{1}^{(k)}\Big|^p\Big],
\end{eqnarray*}
using again the Marcinkiewicz-Zygmund inequality. 
Continuing, using convexity, 
$$
C_p n^{-p/2} \ee\bigg[\Big(\sum_{j=1}^{[n/m]}|Z_j^{(k)}|^2\Big)^{p/2}\bigg]\leq C_p n^{-p/2} \ee\bigg[[n/m]^{p/2-1}\sum_{j=1}^{[n/m]}|Z_j^{(k)}|^p\bigg]\leq \frac{C_p}{m^{p/2}} \ee\Big[|Z_1^{(k)}|^p\Big].
$$
Hence, for any $p> 2$,
\begin{equation*}
\label{eq:Bn_ui2}
\sup_{n\geq 1} \ee\Big[\Big|B_n^{(k)}\Big(\frac1m\Big)\Big|^p\Big]<+\infty,
\end{equation*}
and $\big(B_n^{(k)}(1/m)^2\big)_{n\geq 1}$ is uniformly integrable and therefore, from above, so is $\big(B_n^{(k)}(N_k/n)-B_n^{(k)}(1/m)\big)^2$.
Finally, in order to show \eqref{eq:Bn_approxL2}, it is enough to prove  
\begin{equation}
\label{eq:Bn_approxP}
B_n^{(k)}\Big(\frac{N_k}n\Big)-B_n^{(k)}\Big(\frac1m\Big) 
\stackrel{\pp}{\longrightarrow}0, \quad n\to+\infty. 
\end{equation}
Setting $A_n=\{|N_k-n/m|\leq \sqrt n\ln n\}$, Hoeffding's inequality ensures that $\lim_{n\to+\infty} \pp(A_n^c)~=~0$.
Therefore, since 
$$
B_n^{(k)}\Big(\frac{N_k}n\Big)-B_n^{(k)}\Big(\frac 1m\Big)
=\frac 1{\sqrt n} \sum_{j=[n/m]+1}^{N_k} Z_j^{(k)},
$$
we have
\begin{eqnarray*}
\pp\Big(\Big|B_n^{(k)}\Big(\frac{N_k}n\Big)-B_n^{(k)}\Big(\frac 1m\Big)\Big|\geq \varepsilon\Big)
&\leq &\pp\bigg(\Big\{\Big|\sum_{j=[n/m]+1}^{N_k} Z_j^{(k)}\Big|\geq \varepsilon \sqrt n\Big\}\cap A_n\bigg)+\pp(A_n^c),
\end{eqnarray*}
and Kolmogorov's maximal inequality entails that 
\begin{eqnarray*}
\pp\bigg(\max_{l\in[n/m-\sqrt n\ln n,n/m+\sqrt n\ln n]}\Big|\sum_{j=[n/m]+1}^{l} Z_j^{(k)}\Big|\geq \varepsilon \sqrt n\bigg)
&\leq&\frac 1{\varepsilon^2n}\ee\bigg[\sum_{j=[n/m]+1}^{n/m+\sqrt n\ln n}(Z_j^{(k)})^2\bigg]\\
&\leq& \frac{(\ln n) \ee\big[(Z_1^{(k)})^2\big]}{\sqrt n},
\end{eqnarray*}
finishing the proof of \eqref{eq:Bn_approxP} and thus of \eqref{eq:Bn_approxL2}. 


\subsubsection*{Proof of Lemma~\ref{lemme:ThetaF}}

Consider three cases: $|x|\ll n$, $x\gg n$ (here $u_n\ll v_n$ means $\lim_{n\to+\infty} u_n/v_n=0$) and $x\approx n$, i.e., $c_1n\leq x\leq c_2n$, for two finite constants $c_1$ and $c_2$, and expand $K_n(x)$ accordingly.
First, let $|x|\ll n$: then, 
\begin{eqnarray*}
K_n(x)&=&\frac{(2n)^{x+2n}}{(2n)^{x+n}(2n)^n}\frac{(1+\frac x{2n})^{x+2n}}{(1+\frac xn)^{x+n}}\\
&=&\exp\left((x+2n)\ln\left(1+\frac x{2n}\right) -(x+n)\ln\left(1+\frac xn\right)\right)\\
&=&\exp\left((x+2n)\left(\frac x{2n}-\frac{x^2}{8n^2}+o\left(\frac{x^2}{n^2}\right)\right)
-(x+n)\left(\frac xn-\frac{x^2}{2n^2}+o\left(\frac {x^2}{n^2}\right)\right)\right)\\
&=&\exp\left(-\frac{x^2}{4n}+\frac{3x^3}{8n^2}+o\left(\frac{x^3}{n^2}\right)+
o\left(\frac{x^2}{n}\right)\right)\\
&=&\exp\left(-\frac{x^2}{4n}+o\left(\frac{x^2}{n}\right)\right), 
\end{eqnarray*}
which yields \eqref{eq:ThetaF} in case $|x|\ll n$.  
Next, let $x\gg n$: then,  
\begin{eqnarray}
K_n(x)&=&\frac{(x+2n)^{x+2n}}{(2x+2n)^{x+n}(2n)^n}
=\frac{x^n}{(4n)^n2^x}\frac{(1+\frac {2n}x)^{x+2n}}{(1+\frac nx)^{x+n}} \nonumber\\
&=&\frac{x^n}{(4n)^n2^x} \exp\left((x+2n) \ln\left(1+\frac {2n}x\right)
-(x+n)\ln\left(1+\frac nx\right)\right) \nonumber\\
&=&\frac{x^{n}}{(4n)^n2^x} \exp\!\left(\!(x+2n)\!\left(\!\frac{2n}x-\frac{2n^2}{x^2}
+o\left(\frac{n^2}{x^2}\right)\!\!\right)-(x+n)\!
\left(\frac nx-\frac{n^2}{2x^2}+o\left(\frac{n^2}{x^2}\right)\!\right)\!\right) \nonumber\\
&=&\frac{x^{n}}{(4n)^n2^x} \exp\left(n+\frac{3n^2}{2x}-\frac{7n^3}{2x^2}
+o\left(\frac{n^2}{x}\right)\right) \nonumber\\
&=&\exp\left(n+\frac{3n^2}{2x}+n\ln\Big(\frac x{4n}\Big)-x\ln 2 
+o\left(\frac{n^2}{x}\right)\right).\label{eq:lemmethetar}\end{eqnarray}
Since $x\gg n$, the larger order in the exponential \eqref{eq:lemmethetar}  
is $x\ln 2$ and, this recovers a bound of the form \eqref{eq:ThetaF} in this case.  
Finally, consider the case $x\approx n$, say $x=\alpha n$ with $\alpha>-1$. 
Then,  
\begin{eqnarray*}
K_n(x)&=&
\frac{\big((\alpha+2)n\big)^{(\alpha+2)n}}{\big((2\alpha+2)n\big)^{(\alpha+1)n}(2n)^n}
=\exp\big(- c(\alpha) n\big), 
\end{eqnarray*}
which is again of the form \eqref{eq:ThetaF}, since
$c(\alpha)=\ln\big({2(2\alpha+2)^{\alpha+1}/(\alpha+2)^{\alpha+2}}\big)$ 
is positive for all $\alpha>-1$ and is also bounded.
\CQFD


\subsection{On \cite{HLM}}
\label{sec:appendix_HLM}
The purpose of this Appendix is to provide some missing steps in the proof of 
the main theorem in \cite{HLM} devoted to the binary case as well as to correct 
the errors present there. The notations and numbering are as in \cite{HLM}.    
In particular, recall that $N_1$ (resp. $N_2$) is the number of 
zeros in $X_1, \dots, X_n$ (resp. $Y_1, \dots, Y_n$). 

\medskip\noindent
{\bf Proof of (13).}
Recall again from \cite{HLM} that  
\begin{eqnarray*}
V_n&=&\max_{0\le k\leq N_1\wedge N_2} \left(\bigwedge_{i=1,2} \left(-\frac 12 
\widehat B_n^{(i)}\left(\frac 12\right)+\widehat B_n^{(i)}\left(\frac kn\right)\right)\right),\\
X_n&=&\max_{0\le t\leq \frac 12} \left(\bigwedge_{i=1,2} 
\left(-\frac 12 \widehat B_n^{(i)}\left(\frac 12\right)+\widehat B_n^{(i)}(t)\right)\right).\\
\end{eqnarray*}
Clearly,  
\begin{equation}
\label{eq:X*}
X_n\ge\bigwedge_{i=1,2}\left(-\frac 12 \widehat B_n^{(i)}\left(\frac 12\right)
+\widehat B_n^{(i)}\left(\frac 12\right)\right)
=\frac 12 \bigwedge_{i=1,2}\widehat B_n^{(i)}\left(\frac 12\right),
\end{equation}
and denote by $i_\ast$ the index for which the minimum in \eqref{eq:X*} 
is attained.  

Next, if $N_1\wedge N_2\leq n/2$, then $V_n\leq X_n$; and  
similarly if the maximum defining $V_n$ is attained at some $k^\ast\leq n/2$, 
then $V_n\leq X_n$.  
Otherwise, $N_1\wedge N_2\geq n/2$ with, moreover, the maximum defining 
$V_n$ attained at $k^\ast\in[n/2, N_1\wedge N_2]$ and so:   
$$
V_n=\bigwedge_{i=1,2} \left(-\frac 12 \widehat B_n^{(i)}\left(\frac 12\right)
+\widehat B_n^{(i)}\left(\frac {k^\ast}n\right)\right). 
$$
Now, via \eqref{eq:X*},  
\begin{eqnarray*}
V_n-X_n&\leq&\bigwedge_{i=1,2} \left(-\frac 12 \widehat B_n^{(i)}\left(\frac 12\right)
+\widehat B_n^{(i)}\left(\frac {k^\ast}n\right)\right)-\bigwedge_{i=1,2} 
\left(-\frac 12 \widehat B_n^{(i)}\left(\frac 12\right)
+\widehat B_n^{(i)}\left(\frac 12\right)\right)\\
&\leq& \left(-\frac 12 \widehat B_n^{(i_\ast)}\left(\frac 12\right)
+\widehat B_n^{(i_\ast)}\left(\frac {k^\ast}n\right)\right)
-\left(-\frac 12 \widehat B_n^{(i_\ast)}\left(\frac 12\right)+
\widehat B_n^{(i_\ast)}\left(\frac 12\right)\right)\\
&=& \widehat B_n^{(i_\ast)}\left(\frac {k^\ast}n\right)
-\widehat B_n^{(i_\ast)}\left(\frac 12\right)\\
&\leq&\max_{t\in\big[\frac 12, \frac{N_{i_\ast}}n\big]} \left(\widehat B_n^{(i_\ast)}(t)
-\widehat B_n^{(i_\ast)}\left(\frac 12\right)\right)\\
&\leq&\bigvee_{i=1,2}\max_{t\in\big[\frac 12, \frac{N_i}n\big]} 
\left(\widehat B_n^{(i)}(t)-\widehat B_n^{(i)}\left(\frac 12\right)\right).
\end{eqnarray*}

\medskip\noindent
{\bf Inequality \eqref{eq:15new} replacing (15) of \cite{HLM} and its proof.}
If $N_1\wedge N_2\geq n/2$, then $X_n\leq V_n$ 
and similarly if the maximum defining $X_n$ is attained for 
some $t\leq (N_1\wedge N_2)/n$, then $X_n=V_n$. 
Therefore, the remaining case in comparing $X_n$ and $V_n$ 
consists in $N_1\wedge N_2\leq n/2$ and a maximum defining $X_n$ 
attained at some $t^\ast \in \big[(N_1\wedge N_2)/n, 1/2\big]$. 
In this case, 
$$
X_n= \bigwedge_{i=1,2} \left(-\frac 12 \widehat 
B_n^{(i)}\left(\frac 12\right)+\widehat B_n^{(i)}(t^\ast)\right),
$$
and 
\begin{equation}
\label{eq:V*}
V_n\geq \bigwedge_{i=1,2} 
\left(-\frac 12 \widehat B_n^{(i)}\left(\frac 12\right)
+\widehat B_n^{(i)}\left(\frac {N_1\wedge N_2}n\right)\right).
\end{equation}
Again, denote by $i_\ast$ the index for which the 
minimum in \eqref{eq:V*} is attained.  Then, 
\begin{eqnarray}
\nonumber
X_n-V_n&\leq&\bigwedge_{i=1,2} \left(-\frac 12 \widehat 
B_n^{(i)}\left(\frac 12\right)+\widehat B_n^{(i)}(t^\ast)\right)-\bigwedge_{i=1,2} 
\left(-\frac 12 \widehat B_n^{(i)}\left(\frac 12\right)+\widehat B_n^{(i)}
\left(\frac {N_1\wedge N_2}n\right)\right)\\
\nonumber
&\leq& \left(-\frac 12 \widehat B_n^{(i_\ast)}\left(\frac 12\right)
+\widehat B_n^{(i_\ast)}(t^\ast)\right)
-\left(-\frac 12 \widehat B_n^{(i_\ast)}\left(\frac 12\right)
+\widehat B_n^{(i_\ast)}\left(\frac {N_1\wedge N_2}n\right)\right)\\
\nonumber
&=&\widehat B_n^{(i_\ast)}(t^\ast)
-\widehat B_n^{(i_\ast)}\left(\frac {N_1\wedge N_2}n\right)\\
\nonumber
&\leq&\max_{t\in\left[\frac{N_1\wedge N_2}n, \frac 12\right]} \left(\widehat B_n^{(i_\ast)}
(t)-\widehat B_n^{(i_\ast)}\left(\frac {N_1\wedge N_2}n\right)\right)\\
\label{eq:15new}
&\leq&\bigvee_{i=1,2}\max_{t\in\left[\frac{N_1\wedge N_2}n, 
\frac 12\right]} 
\left(\widehat B_n^{(i)}(t)-\widehat B_n^{(i)}\left(\frac {N_1\wedge N_2}n\right)\right).
\end{eqnarray}
Since (15) of \cite{HLM} has to be replaced by \eqref{eq:15new}, instead of (16) of \cite{HLM}, 
we now have to prove that for $i=1,2$: 
\begin{equation}
\label{eq:16*}
\max_{t\in\left[\frac{N_1\wedge N_2}n, \frac 12\right]} 
\left(\widehat B_n^{(i)}(t)-\widehat B_n^{(i)}\left(\frac {N_1\wedge N_2}n\right)\right)\cvP0.  
\end{equation}
The difference with (16) of \cite{HLM} is that $N$, therein, is now 
replaced by $N_1\wedge N_2$ which is 
now more complex since one of the two quantities $N_1$ or $N_2$ 
is not independent of $\widehat B_n$. 
To prove \eqref{eq:16*}, and so as not to further burden the notation, 
the superscript $i$ in the Brownian approximation $\widehat B_n^{(i)}$ is dropped. 
First, let 
$$
C_n^1=\Big\{\left|N_1-\frac n2\right| \leq \sqrt n{\ln n}\Big\},
$$ 
and, in a similar fashion, define $C_n^2$ by replacing $N_1$ 
with $N_2$. 
Clearly, $\lim_{n\to+\infty}\pp\big((C_n^1)^c\big)=\lim_{n\to+\infty}\pp((C_n^2)^c)=0$.  
Next, for $\varepsilon>0$, let 
$$
A_n=\left\{\max_{t\in\left[\frac{N_1\wedge N_2}n, \frac 12\right]} 
\left|\widehat B_n(t)-\widehat B_n\left(\frac{N_1\wedge N_2}n\right)\right|
\ge \varepsilon \right\}.  
$$
Then,  
\begin{eqnarray}
\label{eq:An11}
\pp(A_n)&\leq&\pp\big(A_n\cap C_n^1\cap C_n^2\big)+\pp\big((C_n^1)^c\big)
+\pp\big((C_n^2)^c\big), 
\end{eqnarray}
and since on $C_n^1$ (resp. $C_n^2$), 
${N_1}\geq n/2-\sqrt n\ln n$ (resp. ${N_2}\geq n/2-\sqrt n\ln n$),  
\begin{eqnarray}
\label{eq:An*}
\pp(A_n\cap C_n^1\cap C_n^2)
&\leq& \pp\left(\left\{\max_{k=\frac n2-\sqrt n\ln n, \dots, \frac n2}
\left|\sum_{j=N_1\wedge N_2}^k \xi_j\right|
\ge \varepsilon\sqrt{2n}\right\}\cap C_n^1\cap C_n^2\right), 
\end{eqnarray}
where the random variables $\xi_j$ are iid with mean zero and variance one 
and assuming that $n/2-\sqrt n\ln n$ and $n/2$ are integers (if not replace 
throughout, the first value by its integer part and the second by its 
integer part plus one).    
To deal with \eqref{eq:An*}, first note that on $C_n^1\cap C_n^2$, 
$N_1\wedge N_2\in \big[\frac n2-\sqrt n\ln n,n\big]$, the right-hand 
side of \eqref{eq:An*} is clearly upper-bounded by 
\begin{eqnarray}
\nonumber
&&\pp\left(\left\{\max_{\frac n2-\sqrt n\ln n\leq \ell\leq k\leq \frac n2}
\left|\sum_{j=\ell}^{k}\xi_j\right|\ge 
\varepsilon\sqrt{\frac{n}{2}}\right\} \cap C_n^1\cap C_n^2\right)\\
\label{eq:An1}
&&\hskip 1.5cm \leq \pp\left(\left\{\max_{\frac n2-\sqrt n\ln n\leq  k\leq \frac n2}
\left|\sum_{j=k}^{n/2}\xi_j\right|\ge 
\frac \varepsilon 2\sqrt{\frac{n}{2}}\right\} \cap C_n^1\cap C_n^2\right)\\
 \label{eq:AnE0}
&&\hskip 1.5cm \leq \frac{8\ln n}{\varepsilon^2 \sqrt n},
\end{eqnarray}
where the inequality in \eqref{eq:An1} follows from the bound  
\begin{eqnarray*}
\max_{\frac n2-\sqrt n\ln n\leq \ell\leq k\leq \frac n2}\left|\sum_{j=\ell}^{k}\xi_j\right|
&\leq& \max_{\frac n2-\sqrt n\ln n\leq \ell\leq k\leq \frac n2}
\left(\left|\sum_{j=k}^{n/2}\xi_j\right|+\left|\sum_{j=\ell}^{n/2}\xi_j\right|\right)\\
&\leq& 
2\max_{\frac n2-\sqrt n\ln n\leq  k\leq \frac n2}\left|\sum_{j=k}^{n/2}\xi_j\right|,
\end{eqnarray*}
while the one in \eqref{eq:AnE0} is Kolmogorov's maximal inequality. 
Therefore, the right-hand side of \eqref{eq:An*} converges to zero,  
finishing, via \eqref{eq:An11}, the proof of \eqref{eq:16*}.
\CQFD


\section*{Acknowledgments}

Both authors  thank an anonymous referee for valuable comments which helped to improve this manuscript, as well as Cl\'ement Deslandes for pointing out a gap in our published proof.


{\scriptsize
\begin{thebibliography}{999}


\bibitem[BDJ]{BaikDeiftJohansson99}
J.~Baik, P.~Deift, and K.~Johansson.
{\it On the distribution of the length of the longest increasing 
subsequence of random permutations}. J. Amer. Math. Soc., 12(4), pp.1119--1178, 1999.

\bibitem[Bar]{Bar} Y.~Baryshnikov. {\it GUEs and queues}, 
Probab.~Theory Relat.~Fields vol. 119, pp. 256--274, 2001.

\bibitem[BGH]{BGH} F.~Benaych-Georges, C.~Houdr\'e. {\it GUE minors,  
maximal Brownian functionals and longest increasing subsequences in random words}.   
Markov Processes Relat.~Fields, vol. 21, pp. 109-126, 2015.

\bibitem[Bil]{Billingsley} P.~Billingsley. {\it Convergence of probability measures}, Wiley series in Probability and Statistcics, 2nd Edition, 1999. 

\bibitem[BM]{BM} T.~Bodineau, J.~Martin. {\it A universality property 
for last-passage percolation paths close to the axis}. Elec.~Comm.~Prob.~vol. 10, pp. 105--112, 2005 

\bibitem[BH]{Breton_Houdre} J.-C.~Breton, C.~Houdr\'e. {\it Simultaneous asymptotics for the 
shape of random Young tableaux with growingly 
reshuffled alphabets}. Bernoulli, vol. 16, no. 2, pp. 471--492, 2010.

\bibitem[CZFYZ]{CZFYZ} W.T.~Chan, Y.~Zhang, S.P.Y.~Fung, D.~Ye and H.~Zhu. 
{\it Efficient algorithms for finding a longest common increasing subsequence}. 
Lecture Notes in Comput. Sci., vol. 3827, Springer,  Berlin, pp. 655--674, 2005. 

\bibitem[DKFPWS]{DKFPWS} A.L.~Delcher, S.~Kasif, R.D.~Fleischmann, 
J.~Peterson, O.~White and S.L.~Salzberg. {\it Aligment of whole genomes}. 
Nucleic Acids Research, vol. 27, no. 11, pp. 2369--2376, 1999. 

\bibitem[GW]{GW} P.~W.~Glynn, W.~Whitt.
{\it Departure from many queues in series}.
Ann. Appl. Probab., 1(4), pp. 546--572, 1991.

\bibitem[GTW]{GTW} J.~Gravner, C.~A.~Tracy, H.~Widom.  {\it Limit theorems for height fluctuations in a class of discrete space and time growth models},  J.~Stat.~Phys. vol. 102, pp. 1085--1132, 2001.

\bibitem[HI]{HI} C.~Houdr\'e and \"U.~I\c{s}lak. A central limit theorem for the length of the longest common subsequences in random words.
\href{https://arxiv.org/abs/1408.1559}{arXiv:1408.1559}, 2015.

\bibitem[HLM]{HLM} C. Houdr\'e, J. Lember, H. Maztinger. {\it On the longest common increasing binary subsequence}. C.R. Acad. Sci., Paris Ser. I, vol. 343, pp. 589--594,  2006.  

\bibitem[HL]{HL} C. Houdr\'e, T. Litherland. {\it On the longest increasing 
subsequence for finite and countable alphabets}, in {High Dimensional Probability V: The Luminy Volume} 
(Beachwood, Ohio, USA: Institute of Mathematical Statistics), pp. 185--212, 2009.

\bibitem[HX]{HX} C. Houdr\'e, H. Xu. {\it On the limiting shape of Young diagrams associated 
with inhomogeneous random words}, in: {High Dimensional Probability VI: The Banff volume} 
Progress in Probability, 66, Birkhauser, pp. 277--302, 2013.

\bibitem[ITW1]{ITW1} A.~Its, C. A.~Tracy, H.~Widom. {\it Random words, Toeplitz determinants, 
and integrable systems. I.} Random matrix models and their 
applications, pp. 245--258, Math. Sci. Res. Inst. Publ., vol. 40, Cambridge 
Univ. Press, Cambridge, 2001.

\bibitem[ITW2]{ITW2} A.~Its, C. A.~Tracy, H.~Widom. {\it Random words, Toeplitz determinants, 
and integrable systems. II.}  Advances in nonlinear mathematics 
and science. Phys. D., vol. 152-153, pp. 199--224, 2001.

\bibitem[Joh]{KJAOM2001} K. Johansson. {\it Discrete orthogonal polynomial ensembles 
and the Plancherel measure}.
Ann. of Math. (2) 153 (2001), no. 1, 259--296.

\bibitem[Ker]{kerov} S. Kerov. {\it Asymptotic Representation Theory of the 
Symmetric Group and its Applications in Analysis,} Vol.  219. AMS, Translations of Mathematical Monographs, 2003. (Russian edition: D. Sci thesis, 1994)


\bibitem[Rom]{Romik} D.~Romik.  {\it The surprising mathematics of
longest increasing subsequences}. Cambridge University Press, 2014.


\bibitem[Sak]{Sakai} Y.~Sakai. {\it A linear space algorithm for computing a 
longest common increasing subsequence}. Information Processing Letters, vol. 99, 
pp. 203--207, 2006. 


\bibitem[Sep]{Sep}
T.~Sepp\"al\"ainen.
{\it A scaling limit for queues in series}. 
Ann. Appl. Probab., 7(4), pp. 855--872, 1997.



\bibitem[TW]{TW} C. A. Tracy, H. Widom. {\it On the distribution of the lengths of the longest 
increasing monotone  subsequences in random words}. 
Probab.~Theor.~Rel.~Fields. vol. 119, pp. 350--380, 2001.


\bibitem[VK]{VK}
A.~M. Vershik, S.~V. Kerov.
{\it Asymptotic behavior of the {P}lancherel measure of the symmetric 
group and the limit form of {Y}oung tableaux}.  
Soviet Math. Dokl. (English translation), 233(1--6): pp. 527--531, 1977.


\end{thebibliography}

}
\end{document}